\documentclass[10pt]{amsart}
\usepackage{fullpage, amsmath, amsthm,amsfonts,url,amssymb,stmaryrd, mathrsfs}
\usepackage{mathrsfs}
\usepackage{amssymb,amscd}
\usepackage{color}
\usepackage{chemarrow}
\usepackage{pictexwd,dcpic}
\usepackage{extarrows}
\usepackage[colorlinks,linkcolor=red,anchorcolor=green,citecolor=blue]{hyperref}
\usepackage{enumitem}

\newtheorem{theorem}{Theorem}[section]
\newtheorem{corollary}[theorem]{Corollary}
\newtheorem{lemma}[theorem]{Lemma}
\newtheorem{proposition}[theorem]{Proposition}

\newtheorem{definition}[theorem]{Definition}

\newtheorem{remark}[theorem]{Remark}

\newtheorem{example}[theorem]{Example}
\newtheorem{conjecture}[theorem]{Conjecture}

\newcommand{\hooklongrightarrow}{\lhook\joinrel\longrightarrow}
\newcommand{\twoheadlongrightarrow}{\relbar\joinrel\twoheadrightarrow}

\newcommand{\us}{\upsilon}
\newcommand{\ra}{\rightarrow}
\newcommand{\lra}{\longrightarrow}

\newcommand{\ul}{\underline}

\newcommand{\bA}{\mathbb A}
\newcommand{\bC}{\mathbb C}

\newcommand{\bG}{\mathbb G}
\newcommand{\bH}{\mathbb H}

\newcommand{\Q}{\mathbb Q}
\newcommand{\R}{\mathbb R}

\newcommand{\Z}{\mathbb Z}
\newcommand{\bS}{\mathbb S}

\newcommand{\bP}{\mathbb P}

\newcommand{\cN}{\mathcal N}
\newcommand{\cL}{\mathcal L}
\newcommand{\co}{\mathcal O}

\newcommand{\cR}{\mathcal R}
\newcommand{\cH}{\mathcal H}
\newcommand{\cC}{\mathcal C}
\newcommand{\cS}{\mathcal S}

\newcommand{\cW}{\mathcal W}

\newcommand{\cM}{\mathcal M}

\newcommand{\cV}{\mathcal V}
\newcommand{\cE}{\mathcal E}

\newcommand{\cU}{\mathcal U}
\newcommand{\cZ}{\mathcal Z}

\newcommand{\fh}{\mathfrak h}
\newcommand{\fm}{\mathfrak{m}}

\newcommand{\fl}{\mathfrak l}

\newcommand{\fj}{\mathfrak j}
\newcommand{\fc}{\mathfrak c}

\newcommand{\sN}{\mathscr N}

\DeclareMathOperator{\tr}{\mathrm tr}
\DeclareMathOperator{\GL}{\mathrm GL}

\DeclareMathOperator{\et}{\text{\'et}}

\DeclareMathOperator{\Res}{\mathrm Res}

\DeclareMathOperator{\Sym}{\mathrm Sym}
\DeclareMathOperator{\Gal}{\mathrm Gal}
\DeclareMathOperator{\Hom}{\mathrm Hom}

\DeclareMathOperator{\End}{\mathrm End}
\DeclareMathOperator{\cris}{\mathrm cris}
\DeclareMathOperator{\rig}{\mathrm rig}
\DeclareMathOperator{\an}{\mathrm an}
\DeclareMathOperator{\Spec}{\mathrm Spec}

\DeclareMathOperator{\dR}{\mathrm dR}
\DeclareMathOperator{\Frob}{\mathrm Frob}
\DeclareMathOperator{\Ind}{\mathrm Ind}
\DeclareMathOperator{\unr}{\mathrm unr}

\DeclareMathOperator{\Ker}{\mathrm Ker}

\DeclareMathOperator{\LT}{\mathrm{LT}}

\DeclareMathOperator{\Ext}{\mathrm Ext}

\DeclareMathOperator{\Spm}{\mathrm Spm}

\DeclareMathOperator{\Ima}{\mathrm Im}

\DeclareMathOperator{\lalg}{\mathrm lalg}

\DeclareMathOperator{\id}{\mathrm id}

\DeclareMathOperator{\dett}{\mathrm det}
\DeclareMathOperator{\alg}{\mathrm alg}
\DeclareMathOperator{\cyc}{\mathrm cyc}
\DeclareMathOperator{\soc}{\mathrm soc}

\DeclareMathOperator{\red}{\mathrm red}

\DeclareMathOperator{\st}{\mathrm st}
\DeclareMathOperator{\St}{\mathrm St}

\DeclareMathOperator{\Art}{\mathrm Art}
\DeclareMathOperator{\ab}{\mathrm ab}

\DeclareMathOperator{\ur}{\mathrm ur}

\DeclareMathOperator{\rk}{\mathrm rk}

\DeclareMathOperator{\wt}{\mathrm wt}
\begin{document}
\title{$\cL$-invariants, partially de Rham families, and local-global compatibility}
\author{Yiwen Ding}
\maketitle
\begin{abstract}
  Let $F_{\wp}$ be a finite extension of $\Q_p$. By considering partially de Rham families, we establish a Colmez-Greenberg-Stevens formula (on Fontaine-Mazur $\cL$-invariants) for (general) $2$-dimensional semi-stable non-crystalline $\Gal(\overline{\Q_p}/F_{\wp})$-representations. As an application, we prove local-global compatibility results for completed cohomology of quaternion Shimura curves, and in particular the equality of Fontaine-Mazur $\cL$-invariants and Breuil's $\cL$-invariants, in critical case.
\end{abstract}
\tableofcontents
\section*{Introduction}
Let $F$ be a totally real number field, $B$ a quaternion algebra of center $F$ such that there exists only one real place of $F$ where $B$ is split. One can associate to $B$ a system of quaternion Shimura curves $\{M_K\}_{K}$,  proper and smooth over $F$, indexed by compact open subgroups $K$ of $(B\otimes_{\Q} \bA^{\infty})^{\times}$. We fix a prime number $p$, and suppose that there exists only one place $\wp$ of $F$ above $p$, let $\Sigma_{\wp}$ be the set of embeddings of $F_{\wp}$ in $\overline{\Q_p}$.  Suppose $B$ is split at $\wp$, i.e.  $(B\otimes_{\Q} \Q_p)^{\times} \cong \GL_2(F_{\wp})$ (where $F_{\wp}$ denotes the completion of $F$ at $\wp$). Let $E$ be a finite extension of $\Q_p$ sufficiently large containing all the embeddings of $F_{\wp}$ in $\overline{\Q_p}$, with $\co_E$ its ring of integers and $\varpi_E$ a uniformizer of $\co_E$.

Let $\rho$ be a $2$-dimensional continuous representation of $\Gal_F:=\Gal(\overline{F}/F)$ over $E$ such that $\rho$ appears in the \'etale cohomology of $M_K$ for $K$ sufficiently small (so $\rho$ is associated to certain Hilbert eigenforms). By Emerton's completed cohomology theory \cite{Em1}, one can associate to $\rho$ a unitary admissible Banach representation $\widehat{\Pi}(\rho)$ of $\GL_2(F_{\wp})$ as follows: put
\begin{equation*}
  \widetilde{H}^1(K^p,E):=\Big(\varprojlim_{n}\varinjlim_{K_p'} H^1_{\et}\big(M_{K^pK_p'} \times_{F} \overline{F}, \co_E/\varpi_E^n\big)\Big)\otimes_{\co_E} E
\end{equation*}
where $K^p$ denotes the component of $K$ outside $p$, and $K_p'$ runs over open compact subgroups of $\GL_2(F_{\wp})$. This is an $E$-Banach space equipped with a continuous action of $\GL_2(F_{\wp}) \times \Gal_F \times \cH^{p}$, where $\cH^{p}$ denotes certain commutative Hecke algebra outside $p$ over $E$. Put
\begin{equation*}
  \widehat{\Pi}(\rho):=\Hom_{\Gal(\overline{F}/F)}\big(\rho, \widetilde{H}^1(K^{p},E)\big).
\end{equation*}
This is an admissible unitary Banach representation of $\GL_2(F_{\wp})$ over $E$, which plays an important role in $p$-adic Langlands program \cite{Br0}. In \cite{Ding3}, it's proved that if  the local Galois representation  $\rho_{\wp}:=\rho|_{\Gal_{F_{\wp}}}$ (where $\Gal_{F_{\wp}}:=\Gal(\overline{F_{\wp}}/F_{\wp})$) is semi-stable non-crystalline and \emph{non-critical}, then one could find the Fontaine-Mazur $\cL$-invariants $(\cL_{\sigma})_{\sigma\in \Sigma_{\wp}}$ of $\rho_{\wp}$ (which are \emph{invisible} in classical local Langlands correspondence) in $\widehat{\Pi}(\rho)$, generalizing some of Breuil's results in \cite{Br10}.

However, when $F_{\wp}$ is different from $\Q_p$, a new phenomenon is that there exist $2$-dimensional semi-stable non-crystalline $\Gal_{F_{\wp}}$-representations which are \emph{critical} (or more precisely, critical for some embeddings in $\Sigma_{\wp}$). We consider this case in this paper. Denote by $S_c(\rho_{\wp})$ (resp. $S_n(\rho_{\wp})$) the set of embeddings where $\rho_{\wp}$ is critical (resp. non-critical), one can associate to $\rho_{\wp}$ the Fontaine-Mazur $\cL$-invariants $\cL_{\sigma}$ but only for embeddings $\sigma$ in $S_n(\rho_{\wp})$. In this paper, we prove these $\cL$-invariants can be found in $\widehat{\Pi}(\rho)$, meanwhile, we prove a partial result on Breuil's locally analytic socle conjecture \cite{Br13I} for embeddings in $S_c(\rho_{\wp})$.

One important ingredient in \cite{Ding3} is Zhang's generalization \cite[Thm.1.1]{Zhang} of Colmez-Greenberg-Stevens formula \cite{Colm10} (on $\cL$-invariants) in $F_{\wp}$-case. But the results in \cite{Zhang} are only for non-critical case. The following theorem generalizes such a formula in general case, which is of interest in its own right.
\begin{theorem}[$\text{cf. Cor.\ref{cor: lpl-rvs}}$]\label{thm: lpl-yee}
  Let $A$ be an affinoid $E$-algebra, $V$ be a locally free $A$-module of rank $2$ equipped with a continuous $A$-linear action of $\Gal_{F_{\wp}}$, let $z$ be an $E$-point of $A$, suppose
  \begin{enumerate}\item $V$ is trianguline with a triangulation given by
  \begin{equation*}
  0 \lra \cR_A(\delta_1) \lra D_{\rig}(V) \lra \cR_A(\delta_2) \lra 0,
\end{equation*}
where $\delta_i$ are continuous characters of $\Gal_{F_{\wp}}$ in $A^{\times}$, \item $V_z:=z^*V$ is semi-stable non-crystalline with $\cL_{\sigma}\in E$ for $\sigma\in S_n(V_z)$ the associated Fontaine-Mazur $\cL$-invariants (cf. \S \ref{sec: lpl-1.3}, where $S_n(V_z)$ denotes the set of embeddings where $V_z$ is non-critical),
  \item $V$ is $S_c(V_z)$-de Rham (cf. \S \ref{sec: lpl-2}, where $S_c(V_z)=\Sigma_{\wp}\setminus S_n(V_z)$); \end{enumerate}
then the differential form
\begin{equation*}
   d \log\big(\delta_1\delta_2^{-1}(p)\big) + \sum_{\sigma\in S_n(V_z)} \cL_{\sigma} d\big(\wt(\delta_1\delta_2^{-1})_{\sigma}\big)\in \Omega^1_{A/E}
\end{equation*}
vanishes at the point $z$.
\end{theorem}
Such formula was firstly established by Greenberg-Stevens \cite[Thm.3.14]{GS93} in the case of $2$-dimensional ordinary $G_{\Q_p}$-representations by Galois cohomology computations. In \cite{Colm10}, Colmez generalized \cite[Thm.3.14]{GS93} to $2$-dimensional trianguline $G_{\Q_p}$-representations case by Galois cohomology computations and computations in Fontaine's rings. The theorem \ref{thm: lpl-yee} in non-critical case (i.e. $S_c(V_z)=\emptyset$) was obtained by Zhang in \cite{Zhang}, by generalizing Colmez's method. In \cite{Pot}, Pottharst generalized \cite[Thm.3.14]{GS93} to rank $2$ triangulable $(\varphi,\Gamma)$-modules (in $\Q_p$ case) by studying cohomology of $(\varphi,\Gamma)$-modules.

The hypothesis (3) in Thm.\ref{thm: lpl-yee} is new but crucial. In fact, the statement does\emph{ not} hold (in general) if the condition (3) is replaced by (only) fixing the Hodge-Tate weights for $\sigma\in S_c(V_z)$ \big(namely, replacing the $S_c(V_z)$-de Rham family by $S_c(V_z)$-Hodge-Tate family\big).
%one replaces the condition $V$ being $S_c(V_z)$-de Rham by that the Hodge-Tate weights of $V$ at $\sigma\in S_c(V_z)$ are fixed integers \big(when the weights at $\sigma\in S_c(V_z)$ are distinct, the latter condition is equivalent to saying $V$ is $S_c(V_z)$-Hodge-Tate\big).
Partially de Rham families appear naturally in the study of $p$-adic automorphic forms, e.g.  one encounters such families when studying locally analytic vectors in completed cohomology of Shimura curves (e.g. see Prop.\ref{prop: lpl-uoc}), or certain families of overconvergent Hilbert modular forms (e.g. see App.\ref{sec: lpl-app}, in particular Conj.\ref{conj: lpl-fyl}).  Note Thm.\ref{thm: lpl-yee} also applies for families of \emph{$F_{\wp}$-analytic} $\Gal_{F_{\wp}}$-representations (cf. \cite{Ber14}, which in fact  can be viewed as special cases of partially de Rham families).
%This theorem might shed light on the study of $p$-adic $L$-functions in the critical (and $F_{\wp}\neq \Q_p$) case.
Indeed, this theorem also includes the case of parallel Hodge-Tate weights for some embeddings \big(and such embeddings would be contained in $S_c(V_z)$\big).

Return to the global setting before Thm.\ref{thm: lpl-yee}, and suppose moreover $\rho$ is absolutely irreducible modulo $\varpi_E$, and $\rho_{\wp}$ is of Hodge-Tate weights $\ul{h}_{\Sigma_{\wp}}:=(\frac{w-k_{\sigma}+2}{2}, \frac{w+k_{\sigma}}{2})_{\sigma\in \Sigma_{\wp}}$ with $w\in 2\Z$, $k_{\sigma}\in 2\Z_{\geq 1}$ (where we use the convention that the Hodge-Tate weight of the cyclotomic character is $-1$). Since $\rho_{\wp}$ is semi-stable non-crystalline, there exists $\alpha\in E^{\times}$, such that the eigenvalues of $\varphi^{d_0}$ \big(where $d_0$ is the degree of the maximal unramified extension in $F_{\wp}$ over $\Q_p$\big) on $D_{\st}(\rho_{\wp})$ are given by $\{\alpha,p^{d_0}\alpha\}$. Put $\alg(\ul{h}_{\Sigma_{\wp}}):=\otimes_{\sigma\in \Sigma_{\wp}} \big(\Sym^{k_{\sigma}-2} E^2 \otimes_E \dett^{-\frac{2-w-k_{\sigma}}{2}}\big)^{\sigma}$, which is an algebraic representation of $\Res_{F_{\wp}/\Q_p} \GL_2$ over $E$, and
\begin{equation*}
  \St(\alpha,\ul{h}_{\Sigma_{\wp}}):=\unr(\alpha)\circ \dett \otimes_{E} \St \otimes_E  \alg(\ul{h}_{\Sigma_{\wp}}),
\end{equation*}
which is in fact the locally algebraic representation of $\GL_2(F_{\wp})$ associated to $\rho_{\wp}$ via classical local Langlands correspondence, where $\unr(\alpha)$ denotes the unramified character of $F_{\wp}^{\times}$ sending uniformizers to $\alpha$, and $\St$ denotes the usual smooth Steinberg representation of $\GL_2(F_{\wp})$. Moreover it's known $\St(\alpha,\ul{h}_{\Sigma_{\wp}})\hookrightarrow \widehat{\Pi}(\rho)$. By Schraen's results (\cite{Sch10}) on  Breuil's $\cL$-invariants, one can associate to $\rho_{\wp}$ a locally $\Q_p$-analytic representation $\Sigma(\alpha,\ul{h}_{\Sigma_{\wp}}, \ul{\cL}_{S_n(\rho_{\wp})})$ of $\GL_2(F_{\wp})$ over $E$ \big(cf. \S \ref{sec: lpl-3}, as suggested by the notation, this representation is determined by $\alpha$, $\ul{h}_{\Sigma_{\wp}}$ and $\ul{\cL}_{S_n(\rho_{\wp})}$\big), whose socle is exactly $\St(\alpha,\ul{h}_{\Sigma_{\wp}})$.

% The following theorem generalizes \cite[Thm.0.1]{Ding3}.
\begin{theorem}[cf. Thm.\ref{thm: lpl-giao} (2), Cor.\ref{cor: lpl-luni}]\label{thm: lpl-esi}
 Keep the above notation, $\Sigma(\alpha, \ul{h}_{\Sigma_{\wp}}, \ul{\cL}_{S_n(\rho_{\wp})})$ is a subrepresentation of $\widehat{\Pi}(\rho)_{\Q_p-\an}$. Moreover, $\Sigma(\alpha,\ul{h}_{\Sigma_{\wp}}, \ul{\cL}'_{S_n(\rho_{\wp})}) \hookrightarrow \widehat{\Pi}(\rho)_{\Q_p-\an}$ if and only if $\ul{\cL}'_{S_n(\rho_{\wp})}=\ul{\cL}_{S_n(\rho_{\wp})}$.
\end{theorem}
This theorem shows the equality of Fontaine-Mazur $\cL$-invariants and Breuil's $\cL$-invariants.
As in \cite{Ding3}, we use $p$-adic family arguments on both Galois side and $\GL_2(F_{\wp})$ side. %We have a rough picture as follows:
%\begin{equation*}
%  \big\{\substack{\text{Trianguline}\\ \text{$G_{F_\wp}$-representations}}\big\} \longleftrightarrow \big\{\substack{\text{$T(F_{\wp})$-representations}}\big\} \longleftrightarrow \big\{\substack{\text{$\GL_2(F_{\wp})$-representations}}\big\}
%\end{equation*}
 The main objects are eigenvarieties, where live the locally analytic $T(F_{\wp})$-representations and $\Gal_F$-representations.
On one hand, we use the global triangulation theory to relate the $\Gal_{F_{\wp}}$-representations and $T(F_{\wp})$-representations; on the other hand, the locally analytic $T(F_{\wp})$-representations and locally analytic $\GL_2(F_{\wp})$-representations  are linked by the theory of  Jacquet-Emerton functor (\cite{Em11}, \cite{Em2}). Roughly speaking, we get a picture as follows:
\begin{equation*}
 \bigg\{\substack{\text{Trianguline}\\ \text{$\Gal_{F_\wp}$-representations}}\bigg\} \xleftrightarrow{\text{\tiny{Triangulation}}} \bigg\{\substack{\text{Locally analytic } \\ \text{$T(F_{\wp})$-representations}\\ \text{(Eigenvarieties)}}\bigg\} \autoleftrightharpoons{\tiny{Jacquet-Emerton functor}}{\tiny{Adjunction formula}} \bigg\{\substack{\text{Locally analytic }\\ \text{$\GL_2(F_{\wp})$-representations}}\bigg\}
\end{equation*}
All these relations are in family. The global triangulation theory and Thm.\ref{thm: lpl-yee} allow one to find the $\cL$-invariants in the related $T(F_{\wp})$-representations. Via the second arrow, one can thus find the $\cL$-invariants on $\GL_2$-side.
%which allows to transfer some infinitesimal information. As suggested by Colmez's formula, $\cL$-invariants can be found in the tangent space of eigenvarieties.
A key fact is that the family of Galois representations associated to locally $\tau$-analytic vectors of $\widehat{H}^1(K^p,E)$ is \emph{$\Sigma_{\wp}\setminus \{\tau\}$-de Rham } (cf. Prop.\ref{prop: lpl-ocj}), which ensures Thm.\ref{thm: lpl-yee} to apply (this observation, together with Schraen's results \cite{Sch10} on Breuil's $\cL$-invariants, in fact leads to the discovery of the hypothesis (3) in Thm.\ref{thm: lpl-yee}).

%Indeed, for $\tau\in S_n(\rho_{\wp})$,  one proves  $\cL_{\tau}$ could be found in the locally $\tau$-analytic subrepresentation of $\widehat{\Pi}(\rho)_{\tau-\an}$ in the case where $k_{\sigma}=2$ for all $\sigma\in \Sigma_{\wp}$ and $w=0$ (for general case one needs only to twist a locally algebraic representation).

For the critical embeddings, using global triangulation theory and Bergdall's method, we prove some results on Breuil's locally analytic socle conjecture (\cite{Br13I}). Namely, for each $\sigma\in S_c(\rho_{\wp})$, one can associate a locally $\Q_p$-analytic representation $I_{\sigma}^c(\alpha,\ul{h}_{\Sigma_{\wp}})$ of $\GL_2(F_{\wp})$ (see \S \ref{sec: lpl-3}), which can be viewed as a \emph{$\sigma$-companion} representation of $\St(\alpha,\ul{h}_{\Sigma_{\wp}})$.
\begin{theorem}[cf. Thm.\ref{thm: lpl-giao} (1)]\label{thm: lpl-comp}
  Keep the notation as in Thm.\ref{thm: lpl-esi}, $I_{\sigma}^c(\alpha,\ul{h}_{\Sigma_{\wp}})$ is a subrepresentation of $\widehat{\Pi}(\rho)$ if and only if $\sigma\in S_c(\rho_{\wp})$.
\end{theorem}
Thus from $\widehat{\Pi}(\rho)$, we can read out $S_c(\rho_{\wp})$ by Thm.\ref{thm: lpl-comp}, and then $\ul{\cL}_{S_n(\rho_{\wp})}$ by Thm.\ref{thm: lpl-esi}. Since $\rho_{\wp}$ is determined by $\{\alpha,\ul{h}_{\Sigma_{\wp}},S_n(\rho_{\wp}), \ul{\cL}_{S_n(\rho_{\wp})}\}$, we see:
\begin{corollary}
  The local Galois representation $\rho_{\wp}$ is determined by $\widehat{\Pi}(\rho)$.
\end{corollary}
We refer the body of the text for more detailed and more precise statements.

In \S \ref{sec: lpl-1}, we recall (and define) the Fontaine-Mazur $\cL$-invariants in terms of $B$-pairs, and develop some partially de Rham Galois cohomology theory for $B$-pairs. We prove Thm.\ref{thm: lpl-yee} in \S \ref{sec: lpl-2}. In \S \ref{sec: lpl-3}, we recall Schraen's theory on Breuil's $\cL$-invariants of locally $\Q_p$-analytic representations of $\GL_2(F_{\wp})$. These three sections are purely local. In the last section, we prove Thm.\ref{thm: lpl-esi} and Thm.\ref{thm: lpl-comp}. In Appx.\ref{sec: lpl-app}, we study some partially de Rham trianguline representations.

\addtocontents{toc}{\protect\setcounter{tocdepth}{1}}
\subsection*{Acknowledgements}
I would like to thank Christophe Breuil for suggesting to consider the critical case, thank Ruochuan Liu for answering my questions, thank Liang Xiao for drawing my attention to partially de Rham Galois representations, and thank Yuancao Zhang for pointing to me the formula in \cite{Zhang} does not hold in general in critical case. Part of this work was written when the author was visiting Beijing international center for mathematical research. I would like to thank Ruochuan Liu and B.I.C.M.R. for their hospitality. This work is partially supported by EPSRC grant EP/L025485/1.
\addtocontents{toc}{\protect\setcounter{tocdepth}{2}}

\section{Fontaine-Mazur $\cL$-invariants}\label{sec: lpl-1}
In this section, we recall (and define) Fontaine-Mazur $\cL$-invariants for $2$-dimensional $B$-pairs (see Def.\ref{def: lpl-eln} below). Let $F_{\wp}$ be a finite extension of $\Q_p$ of degree $d$ with $\co_{\wp}$ the ring of integers and $\varpi$ a uniformizer, $\Sigma_{\wp}:=\{\sigma: {F_{\wp}} \hookrightarrow \overline{\Q_p}\}$, $\Gal_{F_{\wp}}:=\Gal(\overline{\Q_p}/{F_{\wp}})$.
We fix an embedding $\iota: {F_{\wp}}\hookrightarrow B_{\dR}^+$ (and hence embeddings $\iota: {F_{\wp}} \hookrightarrow \bC_p$, $B_{\dR}$), and view $B_{\dR}^+$, $B_{\dR}$, $\bC_p$ as ${F_{\wp}}$-algebra via $\iota$. Let $E$ be a finite extension of $\Q_p$ sufficiently large containing all the embeddings of ${F_{\wp}}$ in $\overline{\Q_p}$. For an ${F_{\wp}}$-algebra $R$ and $\sigma\in \Sigma_{\wp}$, put $R_{\sigma}:=R\otimes_{{F_{\wp}},\sigma} E$ (e.g. we get $E$-algebras $B_{\dR,\sigma}^+$, $B_{\dR,\sigma}$, $\bC_{p,\sigma}$); for an $R$-module $M$, put $M_{\sigma}:=M \otimes_{R} R_{\sigma}$.
\subsection{Preliminaries on $B$-pairs}\label{sec: lpl-1.1}Let $B_e:=B_{\cris}^{\varphi=1}$, recall
\begin{definition}[$\text{cf. \cite[\S 2]{Ber08}}$]\label{def: lpl-sre}
(1) A $B$-pair of $\Gal_{F_{\wp}}$ is a couple $W=(W_e, W_{\dR}^+)$ where $W_e$ is a finite free $B_e$-module equipped with a semi-linear continuous action of $\Gal_{F_{\wp}}$, and $W_{\dR}^+$ is a $\Gal_{F_{\wp}}$-stable $B_{\dR}^+$-lattice of $W_{\dR}:= W_e \otimes_{B_e} B_{\dR}^+$. Let $r\in \Z_{>0}$, we say that $W$ is (a $B$-pair) of rank $r$ if $\rk_{B_e} W_e=r$.

(2) Let $W$, $W'$ be two $B$-pairs, a morphism $f: W \ra W'$ is defined to be a $B_e$-linear $\Gal_{F_{\wp}}$-invariant map $f_e: W_e \ra W_e'$ such that the induced $B_{\dR}$-linear map $f_{\dR}:=f_e \otimes \id: W_{\dR} \ra W'_{\dR}$ sends $W_{\dR}^+$ to $(W')_{\dR}^+$. Moreover, we say that $f$ is strict if the $B_{\dR}^+$-module $(W')_{\dR}^+/f_{\dR}^+(W_{\dR}^+)$ is torsion free, where $f_{\dR}^+:=f_{\dR} |_{W_{\dR}^+}$.
\end{definition}
By \cite[Thm. 2.2.7]{Ber08}, there exists an equivalence of categories between the category of $B$-pairs and that of $(\varphi,\Gamma)$-modules over the Robba ring $B_{\rig,{F_{\wp}}}^{\dagger}$ (e.g. see \cite[\S 1.1]{Ber08}).

Let $A$ be a local artinian $E$-algebra with residue field $E$.
\begin{definition}[$\text{cf. \cite[Def.2.11, Lem.2.12]{Na2}}$]
(1) An $A$-$B$-pair is a $B$-pair $W=(W_e, W_{\dR}^+)$ such that $W_e$ is a finite free $B_e \otimes_{\Q_p} A$-module, and $W_{\dR}^+$ is a $\Gal_{F_{\wp}}$-stable finite free $B_{\dR}^+ \otimes_{\Q_p} A$-submodule of $W_{\dR}:=W_e \otimes_{B_e} B_{\dR}$, which generates $W_{\dR}$. We say $W$ is (an $A$-$B$-pair) of rank $r$ if $\rk_{B_e \otimes_{\Q_p} A} W_e=r$.

(2) Let $W$, $W'$ be two $A$-$B$-pairs, a morphism $f: W\ra W'$ is defined to be a morphism of $B$-pairs such that $f_e: W_e\ra W'_e$ (cf. Def.\ref{def: lpl-sre} (2)) is moreover $B_e \otimes_{\Q_p} A$-linear.
\end{definition}
As in \cite[Thm.1.36]{Na}, one can deduce from \cite[Thm.2.2.7]{Ber08} an equivalence of categories between the category of $A$-$B$-pairs and that of $(\varphi,\Gamma)$-modules free over $\cR_A:=B_{\rig,{F_{\wp}}}^{\dagger}\widehat{\otimes}_{\Q_p} A$.

Let $W$ be an $A$-$B$-pair of rank $r$. By using the isomorphism
\begin{equation}\label{equ: lpl-pqm}F_{\wp} \otimes_{\Q_p} A \xrightarrow{\sim} \prod_{\sigma\in \Sigma_{\wp}} A, \ a\otimes b \mapsto (\sigma(a)b)_{\sigma\in \Sigma_{\wp}},\end{equation}
one gets $B_{\dR}^{*} \otimes_{\Q_p} A\xrightarrow{\sim} \oplus_{\sigma\in \Sigma_{\wp}} B_{\dR,\sigma}^{*}$ and  $W_{\dR}^*\xrightarrow{\sim} \oplus_{\sigma\in \Sigma_{\wp}} W_{\dR,\sigma}^*$ where $*\in \{\emptyset, +\}$. Put $D_e(W):=W_e^{\Gal_{F_{\wp}}}$ and $D_{\dR}(W):=W_{\dR}^{\Gal_{F_{\wp}}}$. The last one is thus a finite ${F_{\wp}}\otimes_{\Q_p} A$-module, and admits a decomposition (according to (\ref{equ: lpl-pqm})) $D_{\dR}(W)\xrightarrow{\sim} \prod_{\sigma\in \Sigma_{\wp}} D_{\dR}(W)_{\sigma}$. For $\sigma\in \Sigma_{\wp}$, one has in fact $D_{\dR}(W)_{\sigma}\cong W_{\dR,\sigma}^{\Gal_{F_{\wp}}}$.
\begin{definition}\label{def: lpl-pnt}Keep the above notation, let $\sigma\in \Sigma_{\wp}$, $W$ is called $\sigma$-de Rham if $D_{\dR}(W)_{\sigma}$ is a free $A$-module of rank $r$; for $S\subseteq \Sigma_{\wp}$, $W$ is called $S$-de Rham if $W$ is $\sigma$-de Rham for all $\sigma\in S$ (thus $W$ is de Rham if $W$ is $\Sigma_{\wp}$-de Rham).
\end{definition}
\begin{remark}
  Let $W$ be an $A$-$B$-pair, for $\sigma\in \Sigma_{\wp}$, $W$ is $\sigma$-de Rham if and only if $W$ is $\sigma$-de Rham as an $E$-$B$-pair. The ``only if" part is trivial. Suppose $W$ is $\sigma$-de Rham as an $E$-$B$-pair, denote by $\fm_A$ the maximal ideal of $A$, and $d_A:=\dim_E A$, thus $\dim_E D_{\dR}(W)_{\sigma}=rd_A$. Consider the exact sequence $0 \ra \fm_A D_{\dR}(W)_{\sigma} \ra D_{\dR}(W)_{\sigma} \ra D_{\dR}(W/\fm_A)_{\sigma}$, we deduce the last map is surjective and $\dim_E D_{\dR}(W/\fm_A)_{\sigma}=r$ by dimension calculation \big(since $\dim_E \fm_A D_{\dR}(W)_{\sigma}=\dim_E D_{\dR}(\fm_A W)_{\sigma}\leq (d_A-1)r$\big), from which we deduce $D_{\dR}(W)_{\sigma}$ is a free $A$-module.
% \big(fact: for any finite $A$-module $M$, if $\dim_E M=d_A(\dim_E M/\fm_A)$, then $M$ is free\big).
\end{remark}
\begin{definition}An $A$-$B$-pair $W$ of rank $r$ is called triangulable if it's an successive extension of $A$-$B$-pairs of rank $1$ , i.e. $W$ admits an increasing filtration of sub-$A$-$B$-pairs $W_i$ for $0 \leq i \leq r$ such that $W_0=0$, $W_r=W$, and $W_i/W_{i-1}$ is an $A$-$B$-pair of rank $1$.
\end{definition}
Denote by $B_A:=(B_e\otimes_{\Q_p} A, B_{\dR}^+ \otimes_{\Q_p} A)$ the trivial $A$-$B$-pair. Let $\chi$ be a continuous character of $F_{\wp}^{\times}$ in $A^{\times}$, following \cite[\S 2.1.2]{Na2}, one can associate to $\chi$ an $A$-$B$-pair of rank $1$, denoted by $B_A(\chi)$ (and we refer to \emph{loc. cit.} for details). By \cite[Prop.2.16]{Na2}, all the rank $1$ $A$-$B$-pairs can be obtained in this way: let $W$ be an $A$-$B$-pair of rank $1$, then there exists a unique continuous character $\chi: F_{\wp}^{\times} \ra A^{\times}$ such that $W\xrightarrow{\sim} B_A(\chi)$.
For a continuous representation $V$ of $\Gal_{F_{\wp}}$ over $A$, denote by $W(V):=(B_e \otimes_{\Q_p} V, B_{\dR}^+ \otimes_{\Q_p} V)$ the associated $A$-$B$-pair. The $\Gal_{F_{\wp}}$-representation $V$ is called \emph{trianguline} if $W(V)$ is triangulable.
%Denote by $\fm_A$ the maximal ideal of $A$. Let $\chi: L^{\times} \ra A^{\times}$ be a continuous character, $\chi_0: L^\times \ra E^{\times}$ the character $\chi$ modulus $\fm_A$. Thus there exists an additive character $\psi: L^{\times} \ra \fm_A$ such that $\chi=\chi_0 (1+\psi)$.

\subsection{Cohomology of $B$-pairs}Recall the cohomology of $E$-$B$-pairs (note that $A$-$B$-pairs can also be viewed as $E$-$B$-pairs). Let $W=(W_e, W_{\dR}^+)$ be an $E$-$B$-pair, following \cite[\S 2.1]{Na}, consider the following complex (of $\Gal_{F_{\wp}}$-modules)
\begin{equation*}
  C^{\bullet}(W):= W_e \oplus W_{\dR}^+ \xlongrightarrow{(x,y)\mapsto x-y} W_{\dR}.
\end{equation*}
Put $H^i(\Gal_{F_{\wp}},W):=H^i(\Gal_{F_{\wp}},C^{\bullet}(W))$ (cf. \cite[Def.2.1]{Na}). By definition, one has a long exact sequence
\begin{multline}\label{equ: lpl-ewr}
  0 \ra H^0(\Gal_{F_{\wp}},W) \ra H^0(\Gal_{F_{\wp}},W_e)\oplus H^0(\Gal_{F_{\wp}}, W_{\dR}^+) \ra H^0(\Gal_{F_{\wp}},W_{\dR})\\ \xrightarrow{\delta} H^1(\Gal_{F_{\wp}},W)\ra H^1(\Gal_{F_{\wp}},W_e) \oplus H^1(\Gal_{F_{\wp}},W_{\dR}^+) \ra H^1(\Gal_{F_{\wp}},W_{\dR}).
\end{multline}
For an $E$-$B$-pair $W$, denote by $W^{\vee}$ the dual of $W$:
\begin{equation*}W^{\vee}:=\Big(W_e^{\vee}:=\Hom_{B_e\otimes_{\Q_p} E}(W_e,B_e\otimes_{\Q_p} E), (W^{\vee})_{\dR}^+:=\Hom_{B_{\dR}^+\otimes_{\Q_p} E}(W_{\dR}^+, B_{dR}^+\otimes_{\Q_p} E)\Big)\end{equation*}
where $W^{\vee}_e$, $(W^{\vee})_{\dR}^+$ are equipped with a natural $\Gal_{F_{\wp}}$-action. One can check $W^{\vee}$ is also an $E$-$B$-pair.
\begin{remark}\label{rem: lpl-orl}
  As in  \cite[Def.1.9 (3)]{Na}, one can also consider the dual $W'$ of $W$ as $B$-pair with $W'_e:=\Hom_{B_e}(W_e,B_e)$ and $(W')_{\dR}^+:=\Hom_{B_{\dR}^+}(W_{\dR}^+, B_{\dR}^+)$ (equipped with a natural $\Gal_{F_{\wp}}$-action). Moreover, $W'_e$, $(W')_{\dR}^+$ can be equipped with a natural $E$-action: $(a\cdot f)(v):=f(av)$. One can check  this action realizes $W'$ as an $E$-$B$-pair. Moreover, the trace map $\tr_{E/\Q_p}: E\ra \Q_p$, induces bijections $W^{\vee}_e \xrightarrow{\sim} W'_e$ and  $(W^{\vee})_{\dR}^+ \xrightarrow{\sim} (W')_{\dR}^+$: $f\mapsto \tr_{E/\Q_p} \circ f$, and these bijections give an isomorphism $W^{\vee}\xrightarrow{\sim} W'$ as $E$-$B$-pairs.
\end{remark}
Denote by $W(1)$ the twist of $W$ by $W(\chi_{\cyc})$ where $\chi_{\cyc}$ is the cyclotomic character of $\Gal_{F_{\wp}}$ (base change to $E$):
\begin{equation*}
  W(1):=\Big(W(1)_e:=W_e \otimes_{B_e\otimes_{\Q_p}E} W(\chi_{\cyc})_e, W(1)_{\dR}^+:=W_{\dR}^+ \otimes_{B_{\dR}^+ \otimes_{\Q_p} E} W(\chi_{\cyc})_{\dR}^+\Big).
\end{equation*}
By \cite[\S 2]{Na} and \cite[\S 5]{Na2}, one has
\begin{proposition}\label{prop: lpl-pww} (1) $H^i(\Gal_{F_{\wp}},W)=0$ if $i\notin\{0,1,2\}$, and
$  \sum_{i=0}^2 (-1)^i \dim_E H^i(\Gal_{F_{\wp}},W)=-d(\rk W)$.

(2) There exists a natural isomorphism $H^1(\Gal_{F_{\wp}},W)\xrightarrow{\sim} \Ext^1(B_E,W)$, where $\Ext^1(B_E,W)$ denotes the group of extensions of $E$-$B$-pairs of $B_E$ by $W$.

  (3) Let $V$ be a finite dimensional  continuous $\Gal_{F_{\wp}}$-representation over $E$, then we have natural isomorphisms $H^i(\Gal_{F_{\wp}},W(V))\cong H^i(\Gal_{F_{\wp}}, V)$ for all $i\in \Z_{\geq 0}$.

  (4) The cup-product (see \cite[\S 5]{Na2} for details)
  \begin{equation}\label{equ: lpl-wgl}\cup: H^i(\Gal_{F_{\wp}},W)\times H^{2-i}(\Gal_{F_{\wp}},W^{\vee}(1))\lra H^2(\Gal_{F_{\wp}},B_E(1))\cong H^2(\Gal_{F_{\wp}}, \chi_{\cyc})\cong E\end{equation}
  is a perfect pairing for $i=0, 1, 2$.
\end{proposition}
\begin{remark}\label{rem: lpl-arw}
(1) In fact, in  \cite[\S 5]{Na2}, it's shown that the cup-product $H^i(\Gal_{F_{\wp}},W)\times H^{2-i}(\Gal_{F_{\wp}}, W'(1))\ra H^2(\Gal_{F_{\wp}}, \Q_p(1))\cong \Q_p$ is a perfect pairing (see Rem.\ref{rem: lpl-orl} for $W'$). By discussions in  Rem.\ref{rem: lpl-orl}, identifying $W^{\vee}$ and $W'$, this pairing then is equal to the composition of (\ref{equ: lpl-wgl}) with the trace map $\tr_{E/\Q_p}$, from which one deduces (\ref{equ: lpl-wgl}) is also perfect.

(2) Let $W$ be an $E$-$B$-pair, for an exact sequence of $E$-$B$-pairs
\begin{equation*}\label{equ: lpl-3ar}
  0 \lra W_1 \lra W_2 \lra W_3 \lra 0,
\end{equation*}
one has the following commutative diagram
\begin{equation*}
  \begin{CD}
    H^i(\Gal_{F_{\wp}},W_3) @. \times @. H^j(\Gal_{F_{\wp}},W) @> \cup >> H^{i+j}(\Gal_{F_{\wp}},W_3 \otimes W) \\
    @V \delta VV @. @| @V\delta VV \\
    H^{i+1}(\Gal_{F_{\wp}},W_1) @. \times @. H^j(\Gal_{F_{\wp}},W) @> \cup >> H^{i+j+1}(\Gal_{F_{\wp}},W_1 \otimes W)
  \end{CD},
\end{equation*}
where the $\delta$'s denote the connecting maps, $\cup$ the cup-products, and $W_i\otimes W$ is the $E$-$B$-pair given by  $(W_i\otimes W)_e:=(W_i)_e\otimes_{B_e\otimes_{\Q_p} E} W_e$, $(W_i\otimes W)_{\dR}^+:=(W_i)_{\dR}^+\otimes_{B_{\dR}^+ \otimes_{\Q_p} E} W_{\dR}^+$. %the right vertical arrow is obtained by tensoring (\ref{equ: lpl-3ar}) with $W$ and taking cohomology, and we refer to \cite{Na} for tensor products of $E$-$B$-pairs.

 (3) If $W$ is moreover an $A$-$B$-pair, by the same argument as in \cite[\S 2.1]{Na}, one can show there exists a natural isomorphism $H^1(\Gal_{F_{\wp}},W)\xrightarrow{\sim} \Ext^1(B_A,W)$ as $A$-modules, where $\Ext^1(B_A,W)$ denotes the group of extensions of $A$-$B$-pairs of $B_A$ by $W$.
\end{remark}
Put  (cf. \cite[Def.2.4]{Na})\begin{eqnarray*}
  H^1_g(\Gal_{F_{\wp}},W)&:=&\Ker [H^1(\Gal_{F_{\wp}},W) \ra H^1(\Gal_{F_{\wp}}, W_{\dR})], \\
  H^1_e(\Gal_{F_{\wp}},W)&:=&\Ker [H^1(\Gal_{F_{\wp}},W) \ra H^1(\Gal_{F_{\wp}}, W_e)],
  \end{eqnarray*}
  where the above maps are induced from the natural maps $C^{\bullet}(W) \ra [W_e\ra 0]\ra [W_{\dR}\ra 0]$. Note that by (\ref{equ: lpl-ewr}), the map $H^1(\Gal_{F_{\wp}},W) \ra H^1(\Gal_{F_{\wp}}, W_{\dR})$ factors through (up to $\pm 1$) the natural map $H^1(\Gal_{F_{\wp}}, W) \ra H^1(\Gal_{F_{\wp}}, W_{\dR}^+)$. If $W$ is a de Rham $A$-$B$-pair, let $[X]\in H^1(\Gal_{F_{\wp}},W)\cong \Ext^1(B_A,W)$, then $X$ is de Rham if and only if $[X]\in H^1_{g}(\Gal_{F_{\wp}},W)$. Moreover, in this case, by \cite[Lem.2.6]{Na}, the natural map $H^1(\Gal_{F_{\wp}},W_{\dR}^+) \ra H^1(\Gal_{F_{\wp}},W_{\dR})$ is injective, thus $H^1_g(\Gal_{F_{\wp}},W)=\Ker [H^1(\Gal_{F_{\wp}},W) \ra H^1(\Gal_{F_{\wp}}, W_{\dR}^+)]$.
 % It's known that if $W$ is de Rham, then  the natural map $H^1(\Gal_{F_{\wp}},W_{\dR}^+) \ra H^1(\Gal_{F_{\wp}},W_{\dR})$ is injective (cf. \cite[Lem.2.6]{Na}), thus one can deduce from (\ref{equ: lpl-ewr}) an exact sequence
%  \begin{equation}\label{equ: lpl-ewra}
%     0 \ra H^0(\Gal_{F_{\wp}},W) \ra H^0(\Gal_{F_{\wp}},W_e)\oplus H^0(\Gal_{F_{\wp}}, W_{\dR}^+) \ra H^0(\Gal_{F_{\wp}},W_{\dR})\ \xrightarrow{\delta} H^1_e(\Gal_{F_{\wp}},W) \ra 0.
%  \end{equation}
One has as in  \cite[Prop.2.10]{Na}
\begin{proposition}\label{prop: lpl-psc}Suppose $W$ is de Rham, the perfect pairing (\ref{equ: lpl-wgl}) induces an isomorphism \begin{equation*}H^1_g(\Gal_{F_{\wp}},W)\xlongrightarrow{\sim} H^1_e(\Gal_{F_{\wp}},W^{\vee}(1))^{\perp}.\end{equation*}
\end{proposition}
For $J\subseteq \Sigma_{\wp}$, $J\neq \emptyset$, put
\begin{equation*}
    H^1_{g,J}(\Gal_{F_{\wp}},W):=\Ker [H^1(\Gal_{F_{\wp}},W) \ra \oplus_{\sigma\in J}H^1(\Gal_{F_{\wp}}, W_{\dR,\sigma})],
\end{equation*}
where the map is induced by $C^{\bullet}(W)\ra [W_{\dR}\ra 0]\ra [\oplus_{\sigma\in J}W_{\dR,\sigma} \ra 0]$. Thus $H^1_{g,\Sigma_{\wp}}(\Gal_{F_{\wp}}, W)=H^1_g(\Gal_{F_{\wp}},W)$, $H^1_{g,J}(\Gal_{F_{\wp}},W)\cong \cap_{\sigma\in J} H^1_{g,\sigma}(\Gal_{F_{\wp}}, W)$, and $H^1(\Gal_{F_{\wp}}, W) \ra \oplus_{\sigma\in J} H^1(\Gal_{F_{\wp}}, W_{\dR,\sigma})$ factors through (up to $\pm 1$) the natural map $H^1(\Gal_{F_{\wp}},W)\ra \oplus_{\sigma\in J} H^1(\Gal_{F_{\wp}}, W_{\dR,\sigma}^+)$ (see the discussion above Prop.\ref{prop: lpl-psc}). Moreover, suppose $W$ is a $J$-de Rham $A$-$B$-pair, for $[X]\in H^1(\Gal_{F_{\wp}},W)\cong \Ext^1(B_A, W)$, $X$ is $J$-de Rham if and only if $[X]\in H^1_{g,J}(\Gal_{F_{\wp}},W)$.
  By the same argument as in \cite[Lem.2.6]{Na}, one has
\begin{lemma}\label{lem: lpl-jqt}
  Let $J\subseteq \Sigma_{\wp}$, $J\neq \emptyset$, suppose $W$ is $J$-de Rham, then the map $\oplus_{\sigma\in J}H^1(\Gal_{F_{\wp}}, W_{\dR,\sigma}^+) \ra \oplus_{\sigma\in J}H^1(\Gal_{F_{\wp}}, W_{\dR,\sigma})$ is injective.
\end{lemma}
Thus if $W$ is $J$-de Rham, then one has
\begin{equation}\label{equ: lpl-gno}
  H^1_{g,J}(\Gal_{F_{\wp}}, W) \cong \Ker[H^1(\Gal_{F_{\wp}}, W) \ra \oplus_{\sigma\in J} H^1(\Gal_{F_{\wp}}, W_{\dR,\sigma}^+)].
\end{equation}
\begin{lemma}\label{lem: lpl-tse}
  Let $J\subseteq \Sigma_{\wp}$, $J\neq \emptyset$, suppose $W$ is $J$-de Rham, if $H^0(\Gal_{F_{\wp}}, W_{\dR,\sigma}^+)=0$ for all $\sigma\in J$, then $H^1_{g,J}(\Gal_{F_{\wp}}, W)\xrightarrow{\sim} H^1(\Gal_{F_{\wp}},W)$.
\end{lemma}
\begin{proof}
  It's sufficient to prove $H^1_{g,\sigma}(\Gal_{F_{\wp}},W)\xrightarrow{\sim} H^1(\Gal_{F_{\wp}},W)$ for all $\sigma\in J$. Since $W$ is $\sigma$-de Rham and $H^0(\Gal_{F_{\wp}}, W_{\dR,\sigma}^+)=0$, we see $W_{\dR,\sigma}^+ \cong \oplus_{i\in \Z_{\geq 1}} (t^i B_{\dR,\sigma}^+)^{n_i}$ where $n_i=0$ for all but finite many $i$. However, for $i\in \Z_{\geq 1}$, $H^1(\Gal_{F_{\wp}}, t^i B_{\dR,\sigma}^+)=0$, thus $H^1(\Gal_{F_{\wp}}, W_{\dR,\sigma}^+)=0$, from which (and (\ref{equ: lpl-gno})) the lemma follows.
\end{proof}
For an $E$-$B$-pair $W$, $\delta:F_{\wp}^{\times}\ra E^{\times}$, put $W(\delta):=W\otimes B_E(\delta)$ \big(see Rem.\ref{rem: lpl-arw} (2) for tensor products of $E$-$B$-pairs, and \S \ref{sec: lpl-1.1} for $B_E(\delta)$\big). If there exist $k_{\sigma}\in \Z$ for all $\sigma\in\Sigma_{\wp}$ such that $\delta=\prod_{\sigma\in \Sigma_{\wp}}\sigma^{k_{\sigma}}$, then by \cite[Lem.2.12]{Na}, one has natural isomorphisms
\begin{equation}\label{equ: lpl-atl}
  W(\delta)_e \cong W_e,  \ W(\delta)_{\dR}^+ \cong \oplus_{\sigma\in \Sigma_{\wp}} W(\delta)_{\dR,\sigma}^+ \cong \oplus_{\sigma\in \Sigma_{\wp}} t^{k_{\sigma}} W_{\dR,\sigma}^+.
\end{equation}
Thus if $k_{\sigma}\in \Z_{\geq 0}$ for all $\sigma\in \Sigma_{\wp}$, one gets a natural morphism
\begin{equation}\label{equ: lpl-fjj} \fj: W(\delta) \lra W
\end{equation}
with $\fj_e=\id$ and $\fj_{\dR}^+$ the natural injection $\oplus_{\sigma\in \Sigma_{\wp}} t^{k_{\sigma}} W_{\dR,\sigma}^+\hookrightarrow \oplus_{\sigma\in \Sigma_{\wp}} W_{\dR,\sigma}^+$.

Let $J\subseteq \Sigma_{\wp}$, $J\neq \emptyset$, $W$ be a $J$-de Rham $E$-$B$-pair, let $k_\sigma\in \Z_{\geq 0}$, such that $(t^{k_{\sigma}} W_{\dR,\sigma}^+)^{\Gal_{F_{\wp}}}=0$ for $\sigma\in J$ \big(thus $t^{k_{\sigma}}W_{\dR,\sigma}^+\cong \oplus_{i\in \Z_{\geq 1}} (t^i B_{\dR,\sigma}^+)^{\oplus n_i}$ with $n_i=0$ for all but finite many $i$ for $\sigma\in J$\big), let $\delta:=\prod_{\sigma\in J}\sigma^{k_{\sigma}}$. The morphism (\ref{equ: lpl-fjj}) induces an exact sequence of $\Gal_{F_{\wp}}$-complexes
\begin{equation*}
  0 \ra [W(\delta)_e\oplus W(\delta)_{\dR}^+ \ra W(\delta)_{\dR}]\ra [W_e \oplus W_{\dR}^+ \ra W_{\dR}]\ra [\oplus_{\sigma\in J} (W_{\dR,\sigma}^+/t^{k_{\sigma}} W_{\dR,\sigma}^+) \ra 0].
\end{equation*}
Taking $\Gal_{F_{\wp}}$-cohomology, one gets
\begin{multline*}
  0 \ra H^0(\Gal_{F_{\wp}}, W(\delta))\ra H^0(\Gal_{F_{\wp}}, W) \ra \oplus_{\sigma\in J} H^0(\Gal_{F_{\wp}},W_{\dR,\sigma}^+/t^{k_{\sigma}})\\  \ra H^1(\Gal_{F_{\wp}}, W(\delta)) \ra H^1(\Gal_{F_{\wp}}, W) \ra \oplus_{\sigma\in J} H^1(\Gal_{F_{\wp}},W_{\dR,\sigma}^+/t^{k_{\sigma}}).
\end{multline*}
By our assumption on $k_{\sigma}$, $H^0(\Gal_{F_{\wp}}, W(\delta))=0$ and $H^i(\Gal_{F_{\wp}}, W_{\dR,\sigma}^+)\xrightarrow{\sim} H^i(\Gal_{F_{\wp}}, W_{\dR,\sigma}^+/t^{k_{\sigma}})$ for $i=0, 1$,  from which and Lem.\ref{lem: lpl-jqt} (and the discussion above it), one gets
\begin{equation}\label{equ: lpl-jgw}
0 \ra  H^0(\Gal_{F_{\wp}}, W) \ra \oplus_{\sigma\in J} H^0(\Gal_{F_{\wp}}, W_{\dR,\sigma}^+) \ra H^1(\Gal_{F_{\wp}}, W(\delta)) \ra H^1_{g,J}(\Gal_{F_{\wp}}, W)\ra 0,
\end{equation}
which would be useful to calculate $H^1_{g,J}(\Gal_{F_{\wp}},W)$.
%one has thus (by the same argument as in Prop.\ref{prop: lpl-psc})
%\begin{proposition}\label{prop: lpl-sei}
 % Let $\sigma\in \Sigma_{\wp}$, suppose $W$ is $\sigma$-de Rham, then the perfect pairing (\ref{equ: lpl-wgl}) induces an isomorphism
 % \begin{equation*}
 %   H^1_{g,\sigma}(\Gal_{F_{\wp}},W)\xlongrightarrow{\sim} \big(\delta\big(H^0(\Gal_{F_{\wp}},W^{\vee}(1)_{\dR,\sigma})\big)\big)^{\perp}.
%  \end{equation*}
%\end{proposition}
%For $\emptyset \neq J\subseteq \Sigma_{\wp}$, put
%\begin{equation*}H^1_{g,J}(\Gal_{F_{\wp}},W):=\Ker[H^1(\Gal_{F_{\wp}}, W)\ra \oplus_{\sigma\in J} H^1(\Gal_{F_{\wp}},W_{\dR,\sigma})]\cong \cap_{\sigma\in J} H^1_{g,\sigma}(\Gal_{F_{\wp}},W).\end{equation*}
%% Suppose $W$ is $J$-de Rham, for $[X]\in H^1(\Gal_{F_{\wp}},W)$, then $X\in H^1_{g,J}(\Gal_{F_{\wp}},W)$ if and only if $X$ is $J$-de Rham. By Prop.\ref{prop: lpl-sei}, one has
%\begin{corollary}\label{cor: lpl-mnt}
%  Suppose $W$ is $J$-de Rham, then the perfect pairing (\ref{equ: lpl-wgl}) induces an isomorphism
%   \begin{equation*}
%    H^1_{g,J}(\Gal_{F_{\wp}},W)\xlongrightarrow{\sim} \big(\delta\big(\oplus_{\sigma\in J} D_{\dR}(W^{\vee}(1))_{\sigma}\big)\big)^{\perp}
%  \end{equation*}
%\end{corollary}
At last, note that a morphism of $E$-$B$-pairs $W_1\ra W_2$ induces a map $H^1(\Gal_{F_{\wp}},W_1) \ra H^1(\Gal_{F_{\wp}},W_2)$ which restricts to maps $H^1_{*}(\Gal_{F_{\wp}},W_1)\ra H^1_{*}(\Gal_{F_{\wp}},W_2)$ with $*\in \{e,g, \{g,J\}\}$.

\subsection{Fontaine-Mazur $\cL$-invariants}\label{sec: lpl-1.3}%Firstly introduce some notation, for an $E$-$B$-pair $W$, $\delta:F_{\wp}^{\times}\ra E^{\times}$, put $W(\delta):=W\otimes B_E(\delta)$ \big(see Rem.\ref{rem: lpl-arw} (2) for tensor products of $E$-$B$-pairs, and \S \ref{sec: lpl-1.1} for $B_E(\delta)$\big). If there exist $k_{\sigma}\in \Z$ for all $\sigma\in\Sigma_{\wp}$ such that $\delta=\prod_{\sigma\in \Sigma_{\wp}}\sigma^{k_{\sigma}}$, then by \cite[Lem.2.12]{Na}, one has
%\begin{equation*}
%  W(\delta)_e \cong W_e,  \ W(\delta)_{\dR}^+ \cong \oplus_{\sigma\in \Sigma_{\wp}} W(\delta)_{\dR,\sigma}^+ \cong \oplus_{\sigma\in \Sigma_{\wp}} t^{k_{\sigma}} W_{\dR,\sigma}^+.
%\end{equation*}
%If moreover $k_{\sigma}\geq 1$ for all $\sigma\in \Sigma_{\wp}$, then one has a natural morphism $\fj: W(\delta) \ra W$ with $\fj_e$ the identity and $\fj_{\dR}^+$ the natural inclusion. In particular, let $\chi_1$, $\chi_2$ be two characters of $F_{\wp}^{\times}\ra E^{\times}$, if $\chi_1=\chi_2\prod_{\sigma\in \Sigma_{\wp}}\sigma^{k_{\sigma}}$ with $k_\sigma\in \Z_{\geq 0}$ for all $\sigma$, then one has a natural morphism $\fj: B_E(\chi_1) \ra B_E(\chi_2)$. We would frequently use such morphisms in this section.

Let $\chi$ be a continuous character of $F_{\wp}^{\times}$ in $E^{\times}$, $\chi$ is called \emph{special} if there exist $k_{\sigma}\in \Z$ for all $\sigma\in \Sigma_{\wp}$ such that $\chi=\unr(q^{-1}) \prod_{\sigma\in \Sigma_{\wp}} \sigma^{k_{\sigma}}=\chi_{\cyc} \prod_{\sigma\in \Sigma_{\wp}} \sigma^{k_{\sigma}-1}$ where $\unr(z)$ denotes the unramified character of $F_{\wp}^{\times}$ sending uniformizers to $z$. In this section, we associate to  $[X]\in H^1_g(\Gal_{F_{\wp}},B_E(\chi))$ the so-called \emph{Fontaine-Mazur $\cL$-invariants} for special characters $\chi$.

Let $\chi=\chi_{\cyc} \prod_{\sigma\in \Sigma_{\wp}} \sigma^{k_{\sigma}-1}$,  by \cite[Lem.2.12]{Na}, $B_E(\chi)_e\cong (B_e\otimes_{\Q_p} E)t$, $B_E(\chi)_{\dR}^+ \cong \oplus_{\sigma\in \Sigma_{\wp}} t^{k_{\sigma}} B_{\dR,\sigma}^+$. Put $\eta:=\prod_{\sigma\in \Sigma_{\wp}} \sigma^{1-k_{\sigma}}$, thus $B_E(\eta)\cong B_E(\chi)^{\vee}(1)$, $B_E(\eta)_e \cong B_e\otimes_{\Q_p} E$ and $B_E(\eta)_{\dR}^+ \cong \oplus_{\sigma\in \Sigma_{\wp}} t^{1-k_{\sigma}} B_{\dR,\sigma}^+$.  Put
\begin{equation}\label{equ: lpl-scia}\begin{cases}S_c(\chi):=\{\sigma\in \Sigma_{\wp}\ |\ k_{\sigma}\in \Z_{\leq 0}\},\\ S_n(\chi):=\{\sigma\in \Sigma_{\wp}\ |\ k_{\sigma}\in \Z_{\geq 1}\},\end{cases}\end{equation} thus $B_E(\chi)$ is \emph{non-$S_n(\chi)$-critical} (cf. Def.\ref{def: pdeR-lir} below). By \cite[Prop.2.15, Lem.4.2 and Lem.4.3]{Na}, one has
\begin{lemma}\label{lem: lpl-sit}Keep the above notation.

(1) If $S_c(\chi)=\emptyset$, then we have $\dim_E H^0(\Gal_{F_{\wp}},B_E(\eta))=1$, $\dim_E H^2(\Gal_{F_{\wp}},B_E(\chi))=1$, and $\dim_E H^1(\Gal_{F_{\wp}},B_E(\chi))=\dim_E H^1(\Gal_{F_{\wp}},B_E(\eta))=d+1$;  if $S_c(\chi)\neq\emptyset$, then $\dim_E H^i(\Gal_{F_{\wp}},B_E(\chi))=\dim_E H^i(\Gal_{F_{\wp}},B_E(\eta))=0$ for $i=0$, $2$, and $\dim_E H^1(\Gal_{F_{\wp}},B_E(\chi))=\dim_E H^1(\Gal_{F_{\wp}},B_E(\eta))=d$ ($=[F_{\wp}:\Q_p]$).

(2)  $\dim_E H^1_e(\Gal_{F_{\wp}},B_E(\chi))=d-|S_c(\chi)|$, and $\dim_E H^1_g(\Gal_{F_{\wp}},B_E(\chi))=d+1-|S_c(\chi)|$. If $S_c(\chi)=\emptyset$, $\dim_E H^1_e(\Gal_{F_{\wp}},B_E(\eta))=0$; if $S_c(\chi)\neq \emptyset$, $\dim_E H^1_e(\Gal_{F_{\wp}},B_E(\eta))=|S_c(\chi)|-1$.
\end{lemma}
Suppose first $S_c(\chi)=\emptyset$ \big(thus $H^1_g(\Gal_{F_{\wp}},B_E(\chi))\xrightarrow{\sim} H^1(\Gal_{F_{\wp}},B_E(\chi))$\big), we would use the cup-product
\begin{equation}\label{equ: lpl-ice}
\langle\cdot, \cdot\rangle: H^1(\Gal_{F_{\wp}},B_E(\chi))\times H^1(\Gal_{F_{\wp}},B_E) \longrightarrow H^2(\Gal_{F_{\wp}},B_E(\chi))\cong E
\end{equation}
to define $\cL$-invariants for elements in $H^1(\Gal_{F_{\wp}},B_E(\chi))$.
\begin{lemma}\label{lem: lpl-tng}
  The cup-product (\ref{equ: lpl-ice}) is a perfect pairing.
\end{lemma}
\begin{proof}The natural morphism $\fj: B_E(\chi) \ra B_E(1)$ (cf. (\ref{equ: lpl-fjj})) induces an exact sequence of $\Gal_{F_{\wp}}$-complexes
 \begin{multline}
  0 \lra [B_E(\chi)_e \oplus B_E(\chi)_{\dR}^+ \ra B_E(\chi)_{\dR}] \lra [B_E(1)_e \oplus B_E(1)_{\dR}^+\ra B_E(1)_{\dR}]\\ \lra [\oplus_{\sigma\in \Sigma_{\wp}} tB_{\dR,\sigma}^+/t^{k_{\sigma}}B_{\dR,\sigma}^+ \ra 0] \lra 0.
 \end{multline}
 Since $H^i\big(\Gal_{F_{\wp}},tB_{\dR,\sigma}^+/t^{k_{\sigma}} B_{\dR,\sigma}^+\big)=0$ for any $i\in \Z_{\geq 0}$, we see $\fj$ induces isomorphisms $H^i(\Gal_{F_{\wp}}, B_E(\chi)) \xrightarrow{\sim} H^i(\Gal_{F_{\wp}},B_E(1))$ for $i\in \Z_{\geq 0}$
%\begin{eqnarray*}\fj_e&: &B_E(\chi)_e\cong (B_e \otimes_{\Q_p} E)t \xlongrightarrow{\id} (B_e\otimes_{\Q_p} E)t \cong B_E(1)_e,\\
 %\fj_{\dR}^+&: &B_E(\chi)_{\dR}^+ \cong \oplus_{\sigma\in \Sigma_{\wp}} t^{k_{\sigma}} B_{\dR,\sigma}^+ \hooklongrightarrow \oplus_{\sigma \in \Sigma_{\wp}} t B_{\dR,\sigma}^+\cong B_E(1)_{\dR}^+.
% \end{eqnarray*}
% We claim $\fj$ induces an isomorphism $H^i(\Gal_{F_{\wp}}, B_E(\chi)) \xrightarrow{\sim} H^i(\Gal_{F_{\wp}},B_E(1))$ for $i\in \{0,1,2\}$. In fact, the morphism $\fj$ induces a
% and it's known that $H^i\big(\Gal_{F_{\wp}},tB_{\dR,\sigma}^+/t^{k_{\sigma}} B_{\dR,\sigma}^+\big)=0$ for any $i\in \Z_{\geq 0}$, from which the claim follows.
 Moreover, the following diagram commutes
\begin{equation*}
  \begin{CD}H^1(\Gal_{F_{\wp}},B_E(\chi)) @. \times @. H^1(\Gal_{F_{\wp}},B_E) @> \cup >> H^2(\Gal_{F_{\wp}},B_E(\chi))\\
  @V \sim VV @. @| @V \sim VV \\
    H^1(\Gal_{F_{\wp}},B_E(1)) @. \times @. H^1(\Gal_{F_{\wp}},B_E) @> \cup >> H^2(\Gal_{F_{\wp}},B_E(1))\\
  \end{CD}.
\end{equation*}
Since the cup-product below is perfect by Prop.\ref{prop: lpl-pww}(4), so is the above one.
\end{proof}
Recall that $H^1(\Gal_{F_{\wp}},B_E)\cong H^1(\Gal_{F_{\wp}},E) \cong \Hom(\Gal_{F_{\wp}},E)$, where the last denotes the $E$-vector space of continuous additive characters of $\Gal_{F_{\wp}}$ in $E$. Before going any further, we recall some facts on additive characters of $\Gal_{F_{\wp}}$.
\subsubsection{A digression: additive characters of $\Gal_{F_{\wp}}$} \label{sec: lpl-1.3.1}Let $W_{F_{\wp}}$ denote the Weil group of ${F_{\wp}}$.
We fix a local Artin map $\Art_{F_{\wp}}: F_{\wp}^{\times} \xrightarrow{\sim} W_{F_{\wp}}^{\ab}$ sending uniformizers to geometric Frobenius. One has thus
\begin{multline}\label{equ: lpl-lbe}
H^1(\Gal_{F_{\wp}},B_E) \cong H^1(\Gal_{F_{\wp}},E) \cong \Hom(\Gal_{F_{\wp}},E) \cong \Hom(\Gal_{F_{\wp}}^{\ab},E) \\ \cong \Hom(W_{F_{\wp}}^{\ab},E) \xrightarrow[\sim]{\Art_{F_{\wp}}}\Hom(F_{\wp}^{\times},E)
\end{multline}
where the fourth isomorphism follows from the fact that  any character of $\Z$ in $E$ gives rise to a continuous character of $\widehat{\Z}:=\varprojlim_{n} \Z/n\Z$ in $E$. We would identify these $E$-vector spaces via (\ref{equ: lpl-lbe}) with no mention.

For a uniformiser $\varpi\in F_{\wp}^{\times}$ one gets a character $\varepsilon_{\varpi}: F_{\wp}^{\times} \lra \co_{\wp}^{\times}$ which is identity on $\co_{\wp}^{\times}$ and sends $\varpi$ to $1$ . Let $\psi_{\sigma,\varpi}:=\sigma\circ \log \circ \varepsilon_{\varpi}: F_{\wp}^{\times} \ra E$ for $\sigma\in \Sigma_{\wp}$, and $\psi_{\ur}: F_{\wp}^{\times} \ra \Z$ be the unramified character sending $p$ to $1$ (thus sending $\varpi$ to $e^{-1}$).
 \begin{lemma}$\{\psi_{\sigma,\varpi}\}_{\sigma\in \Sigma_{\wp}}$ and $\psi_{\ur}$ form a basis of $\Hom(F_{\wp}^{\times},E)$.
 \end{lemma}
 \begin{proof}
   One has isomorphisms \begin{equation}\label{equ: lpl-ols}\Hom(\co_{\wp}^{\times},E) \cong \Hom_{\Q_p}(F_{\wp},E) \cong \Hom_E(F_{\wp}\otimes_{\Q_p} E, E)\cong \Hom_E\big(\prod_{\sigma\in \Sigma_{\wp}}E,E\big), \end{equation} where the first isomorphism is induced by the log map. For $\tau\in \Sigma_{\wp}$, one sees $\tau \circ \log: \co_{\wp}^{\times} \ra E$ corresponds to the map $\prod_{\sigma\in \Sigma_{\wp}} E, \ra E$, $(a_{\sigma})_{\sigma\in \Sigma_{\wp}} \mapsto a_{\tau}$. So $\{\sigma\circ \log\}_{\sigma\in \Sigma_{\wp}}$ form a basis of the $E$-vector space $\Hom(\co_{\wp}^{\times},E)$, and hence $\{\psi_{\sigma,\varpi}\}_{\sigma\in \Sigma_{\wp}}$ form a basis of the $E$-vector subspace of $\Hom(F_{\wp}^{\times},E)$ generated by characters sending $\varpi$ to $0$. The lemma follows.
 \end{proof}
 The cyclotomic character $\chi_{\cyc}$ of $\Gal_{F_{\wp}}$ corresponds (via $\Art_{F_{\wp}}$) to the character $F_{\wp}^{\times} \xrightarrow{\cN_{F_{\wp}/\Q_p}} \Q_p^{\times}\ra \Z_p^{\times}$ with the last map being identity on $\Z_p^{\times}$ and sending $p$ to $1$.
% In fact, this follows easily from the fact that $\{\sigma\circ \log\}_{\sigma}$ form a basis of $\Hom(\co_{\wp}^{\times},E)$:
%\begin{equation}\label{equ: lpl-lse}\Hom(\co_{\wp}^{\times},E) \xrightarrow[\log]{\sim} \Hom_{\Q_p}(F_{\wp},E) \cong \Hom_E\big(\prod_{\sigma\in \Sigma_{\wp}} E, E\big).\end{equation}
%Consider the cyclotomic character $\chi_{\cyc}: \Gal_{F_{\wp}}\ra \Z_p^{\times}$, and denote by $\psi_{\cyc}:=\log \circ \chi_{\cyc}: \Gal_{F_{\wp}} \ra \Z_p$. %We see $\chi_{\cyc}$ equals the following composition
%\begin{equation*}
%  W_{F_{\wp}}^{\ab}\xlongrightarrow{\Art_L^{-1}} F_{\wp}^{\times} \xlongrightarrow{N_{L/\Q_p}} \Q_p^{\times} \lra \Z_p^{\times}
%\end{equation*}
%with the last map sending $p$ to $1$.
Consider the restriction of $\cN_{F_{\wp}/\Q_p}$ to $\co_{\wp}^{\times}$, which corresponds (via (\ref{equ: lpl-ols})) to the map $\tr\in \Hom_E\big(\prod_{\sigma\in \Sigma_{\wp}} E, E\big): (a_{\sigma})_{\sigma}\mapsto \sum_{\sigma\in \Sigma_{F_{\wp}}}\sigma(a_{\sigma})$. This map is in fact a generator of  $\Hom_E\big(\prod_{\sigma\in \Sigma_{\wp}} E, E\big)$ over $F_{\wp}\otimes_{\Q_p} E$. For any $f\in F_{\wp} \otimes_{\Q_p} E$, denote by $\psi_{f,p}$ the character $F_{\wp}^{\times} \ra E$ such that $\psi_{f,p}|_{\co_{\wp}^{\times}}$ coincides with the preimage of $f\cdot \tr \in \Hom_E\big(\prod_{\sigma\in \Sigma_{\wp}} E, E\big)$ in $\Hom(\co_{\wp}^{\times},E)$ via (\ref{equ: lpl-ols}) and that $\psi_{f,p}(p)$=1. For $\tau\in \Sigma_{\wp}$, denote by $1_\tau\in F_{\wp}\otimes_{\Q_p} E\cong \prod_{\sigma\in \Sigma_{\wp}} E$ with $(1_{\tau})_{\tau}=1$ and $(1_{\tau})_{\sigma}=0$ for $\sigma\neq \tau$. Let $\psi_{\tau,p}:=\psi_{1_{\tau},p}$ to simplify, we see $\psi_{\tau,\varpi}=\psi_{\tau,p}+ \tau(\log(p/\varpi^e)) \psi_{\ur}$ (by comparing their values of $p$ and $\co_{\wp}^{\times}$). In particular, $\{\psi_{\sigma,p}\}_{\sigma\in \Sigma_{\wp}}$ and $\psi_{\ur}$ also form a basis of $\Hom_{\Q_p}(F_{\wp}^{\times}, E)$.

For $\tau\in \Sigma_{\wp}$, the embedding $\iota: F_{\wp} \hookrightarrow B_{\dR}^+$ induces $\iota_{\tau}: E\hookrightarrow B_{\dR,\tau}^+\twoheadrightarrow \bC_{p,\tau}$. One gets \begin{equation*}\iota_{\tau}: H^1(\Gal_{F_{\wp}},E) \lra H^1(\Gal_{F_{\wp}},B_{\dR,\tau}^+)\xlongrightarrow{\sim} H^1(\Gal_{F_{\wp}},\bC_{p,\tau}).\end{equation*} For $\psi \in H^1(\Gal_{F_{\wp}},E)$,  $\psi$ is mapped to zero if and only if there exists $x\in \bC_{p,\tau}$ such that $\chi(g)=g(x)-x$.
%For $\tau\in \Sigma_{\wp}$, denote by $\chi_{\LT,\varpi,\tau}$ the Lubin-Tate character corresponding to the character $\tau\circ \varepsilon_{\varpi}$ of $F_{\wp}^{\times}$ via $\Art_{F_{\wp}}$.
It's known that for any $\iota' \neq \iota: F_{\wp} \hookrightarrow \bC_p$, there exists $u_{\iota'}\in \bC_p^{\times}$, such that $g(u_{\iota'})=(\iota' \circ \varepsilon_{\varpi}(g)) \cdot u_{\iota'}$ (where $\varepsilon_{\varpi}$ is viewed as a character of $\Gal_{F_{\wp}}$ via $\Art_{F_{\wp}}$), put $x_{\iota'}:=\log(u_{\iota'})$, we have $g(x_{\iota'})-x_{\iota'}=\log\circ \varepsilon_{\varpi}(g)$. From which we deduce that  for any $\tau'\neq \tau$, $\iota_{\tau}(\psi_{\tau',\varpi})=0$. Similarly, we have $\iota_{\tau}(\psi_{\ur})=0$. So $\psi_{\ur}\in H^1_g(\Gal_{F_{\wp}},B_E)=H^1_g(\Gal_{F_{\wp}},E)=\ker[H^1(\Gal_{F_{\wp}},E)\ra H^1(\Gal_{F_{\wp}},\oplus_{\sigma\in \Sigma_{\wp}} B_{\dR,\sigma})]$, and is a generator of $H^1_g(\Gal_{F_{\wp}},B_E)$ (which is $1$-dimensional over $E$). For $S\subseteq \Sigma_{\wp}$, recall $H^1_{g,S}(\Gal_{F_{\wp}},E)=\Ker[H^1(\Gal_{F_{\wp}},E) \ra \oplus_{\sigma\in S} H^1(\Gal_{F_{\wp}}, B_{\dR,\sigma})]$, by the above discussion, one has
\begin{lemma}\label{lem: lpl-dyp}
  The $E$-vector space $H^1_{g,S}(\Gal_{F_{\wp}},E)$ is of dimension $|\Sigma_{\wp}\setminus S|+1$, and is generated by $\{\psi_{\sigma,\varpi}\}_{\sigma\in \Sigma_{\wp} \setminus S}$ and $\psi_{\ur}$ \big(thus can also be generated by $\{\psi_{\sigma,p}\}_{\sigma \in \Sigma_{\wp} \setminus S}$ and $\psi_{\ur}$\big).
\end{lemma}

\subsubsection{$\cL$-invariants}\label{sec: lpl-1.3.2}Return to the situation before \S \ref{sec: lpl-1.3.1} \big(thus $\chi$ is a special character with $S_c(\chi)=\emptyset$\big). Let $[X]\in H^1(\Gal_{F_{\wp}},B_E(\chi))=H^1_g(\Gal_{F_{\wp}},B_E(\chi))$,  by Prop.\ref{prop: lpl-psc} and Lem.\ref{lem: lpl-dyp}, $[X]\in H^1_e(\Gal_{F_{\wp}},B_E(\chi))$ if and only if $\langle[X],\psi_{\ur}\rangle=0$ (cf. (\ref{equ: lpl-ice})).
\begin{definition}[non-critical case]\label{equ: lpl-eie}
  Keep the above notation, if $[X]\notin H^1_e(\Gal_{F_{\wp}},B_E(\chi))$, for $\sigma\in \Sigma_{\wp}$, put  $\cL(X)_{\sigma}:=\langle[X], \psi_{\sigma,p}\rangle/\langle[X],\psi_{\ur}\rangle\in E$ (cf. (\ref{equ: lpl-ice})), and the $\{\cL(X)_{\sigma}\}_{\sigma\in \Sigma_{\wp}}$ are called the Fontaine-Mazur $\cL$-invariants of $X$; if $[X]\in H^1_e(\Gal_{F_{\wp}},B_E(\chi))$, we define the Fontaine-Mazur $\cL$-invariants of $X$ to be $(\cL(X)_{\sigma})_{\sigma\in \Sigma_{\wp}}:=(\langle[X],\psi_{\sigma,p}\rangle)_{\sigma\in \Sigma_{\wp}} \in \bP^{d}(E)$.
\end{definition}
\begin{remark}
  Let $\chi'=\unr(q^{-1})\prod_{\sigma\in \Sigma_{\wp}}\sigma^{k_{\sigma}'}$ with $1\leq k_{\sigma}'\leq k_{\sigma}$ for all $\sigma\in \Sigma_{\wp}$. The natural morphism $\fj: B_E(\chi)\ra B_E(\chi')$ induces isomorphisms $\fj: H^i(\Gal_{F_{\wp}},B_E(\chi)) \xrightarrow{\sim} H^i(\Gal_{F_{\wp}},B_E(\chi'))$ for $i\in \Z_{\geq 0}$ (by the same argument as in the proof of Lem.\ref{lem: lpl-tng}), moreover, the following diagram commutes
  \begin{equation*}
  \begin{CD}H^1(\Gal_{F_{\wp}},B_E(\chi)) @. \times @. H^1(\Gal_{F_{\wp}},B_E) @> \cup >> H^2(\Gal_{F_{\wp}},B_E(\chi))\\
  @V \sim VV @. @| @V \sim VV \\
    H^1(\Gal_{F_{\wp}},B_E(\chi')) @. \times @. H^1(\Gal_{F_{\wp}},B_E) @> \cup >> H^2(\Gal_{F_{\wp}},B_E(\chi'))\\
  \end{CD}.
\end{equation*}
We see by definition $(\cL(X')_{\sigma})_{\sigma\in \Sigma_{\wp}}=(\cL(X)_{\sigma})_{\sigma\in \Sigma_{\wp}}$ if $[X']=\fj([X])$ (up to scalars).
\end{remark}

Consider now the case $S_c(\chi)\neq \emptyset$ (i.e. the \emph{critical} case). Let $$\chi^{\sharp}:= \chi \prod_{\sigma\in S_c(\chi)} \sigma^{1-k_{\sigma}}=\unr(q^{-1})\prod_{\sigma\in S_n(\chi)}\sigma^{k_{\sigma}} \prod_{\sigma\in S_{c}(\chi)}\sigma,$$ thus $\chi^{\sharp}$ is also special and $S_c(\chi^{\sharp})=\emptyset$. The natural morphism of $\fj: B_E(\chi^{\sharp}) \ra B_E(\chi)$ (cf. (\ref{equ: lpl-fjj})) induces a map $\fj: H^1(\Gal_{F_{\wp}}, B_E(\chi^{\sharp})) \ra H^1(\Gal_{F_{\wp}}, B_E(\chi))$.
 \begin{lemma}\label{lem: lpl-fpi}
   $\Ima(\fj)=H^1_{g,S_c(\chi)}(\Gal_{F_{\wp}},B_E(\chi))= H^1_g(\Gal_{F_{\wp}},B_E(\chi))$.
 \end{lemma}
 \begin{proof}By (\ref{equ: lpl-jgw}), $\Ima(\fj)=H^1_{g,S_c(\chi)}(\Gal_{F_{\wp}},B_E(\chi))$. By Lem.\ref{lem: lpl-tse}, $H^1_{g,\sigma}(\Gal_{F_{\wp}}, B_E(\chi))=H^1(\Gal_{F_{\wp}}, B_E(\chi))$ for $\sigma\in S_n(\chi)$, and hence $H^1_{g,S_c(\chi)}(\Gal_{F_{\wp}},B_E(\chi))= H^1_g(\Gal_{F_{\wp}},B_E(\chi))$. The lemma follows.
\end{proof}

Let $\eta':=\prod_{\sigma\in S_c(\chi)}\sigma^{1-k_{\sigma}}$, so $\chi^{\sharp}=\chi \eta'$. We claim the cup-product
\begin{equation}\label{equ: lpl-'al}
  H^1(\Gal_{F_{\wp}},B_E(\chi)) \times H^1(\Gal_{F_{\wp}},B_E(\eta')) \lra H^2(\Gal_{F_{\wp}}, B_E(\chi^{\sharp}))
\end{equation}
is a perfect pairing. Indeed, similarly as in the proof of Lem.\ref{lem: lpl-tng}, this follows from the commutative diagram (recall $\eta=\chi^{-1}\chi_{\cyc}$):
\begin{equation*}
  \begin{CD}H^1(\Gal_{F_{\wp}},B_E(\chi)) @. \times @. H^1(\Gal_{F_{\wp}},B_E(\eta')) @> \cup >> H^2(\Gal_{F_{\wp}},B_E(\chi^{\sharp}))\\
  @| @. @V \sim VV @V \sim VV \\
    H^1(\Gal_{F_{\wp}},B_E(\chi)) @. \times @. H^1(\Gal_{F_{\wp}},B_E(\eta)) @> \cup >> H^2(\Gal_{F_{\wp}},B_E(1))\\
  \end{CD}.
\end{equation*}
One deduces from Prop.\ref{prop: lpl-psc} that this pairing induces  isomorphisms
\begin{equation}\label{equ: lpl-eig}H^1_g(\Gal_{F_{\wp}},B_E(\chi))\cong H^1_e(\Gal_{F_{\wp}},B_E(\eta'))^{\perp}, \ H^1_e(\Gal_{F_{\wp}},B_E(\chi))\cong H^1_g(\Gal_{F_{\wp}},B_E(\eta'))^{\perp}.\end{equation}
Let $\fj'$ denote the natural morphism $B_E(\eta') \ra B_E$, the following diagram commutes
  \begin{equation}\label{equ: lpl-asm}
  \begin{CD}H^1(\Gal_{F_{\wp}},B_E(\eta'))@.  \times @.  H^1(\Gal_{F_{\wp}},B_E(\chi))@>>>  H^2(\Gal_{F_{\wp}},B_E(\chi^{\sharp}))\\
  @V \fj' VV @. @A \fj AA @|\\
  H^1(\Gal_{F_{\wp}},B_E) @. \times @. H^1(\Gal_{F_{\wp}},B_E(\chi^{\sharp})) @>>> H^2(\Gal_{F_{\wp}},B_E(\chi^{\sharp}))
  \end{CD}.
\end{equation}
By (\ref{equ: lpl-jgw}), $\Ima(\fj')=H^1_{g,S_c(\chi)}(\Gal_{F_{\wp}},B_E)$.
\begin{lemma}Denote by $\langle \cdot, \cdot \rangle_n$ the bottom (perfect) pairing in (\ref{equ: lpl-asm}), then
  the pairing \begin{equation}\label{equ: lpl-e1c}
  \langle\cdot, \cdot\rangle: H^1_{g,S_c(\chi)}(\Gal_{F_{\wp}},B_E) \times H^1_g(\Gal_{F_{\wp}},B_E(\chi)) \ra H^2(\Gal_{F_{\wp}},B_E(\chi^{\sharp}))\cong E, \ \langle x, y\rangle:=\langle x, y^{\sharp} \rangle_n,
\end{equation}
where $y^{\sharp}$ denotes a preimage of $y$ in $H^1(\Gal_{F_{\wp}}, B_E(\chi^{\sharp}))$, is independent of the choice of $y^{\sharp}$ and is a perfect pairing. Moreover, this pairing induces an isomorphism $H^1_g(\Gal_{F_{\wp}},B_E)^{\perp}\xrightarrow{\sim} H^1_e(\Gal_{F_{\wp}},B_E(\chi))$.
\end{lemma}
\begin{proof}The independence of the choice of $y^{\sharp}$ follows from the commutativity of  (\ref{equ: lpl-asm}) and the fact $\Ima(\fj')=H^1_{g,S_c(\chi)}(\Gal_{F_{\wp}},B_E)$. Indeed, for $y'\in H^1(\Gal_{F_{\wp}}, B_E(\chi^{\sharp}))$, if $\fj(y')=0$, by  (\ref{equ: lpl-asm}), $\Ima(\fj')\subseteq (E\cdot y')^{\perp}$.

By (\ref{equ: lpl-eig}), the top pairing in (\ref{equ: lpl-asm}) induces a perfect pairing $$H^1(\Gal_{F_{\wp}},B_E(\eta'))/H^1_e(\Gal_{F_{\wp}},B_E(\eta')) \times H^1_g(\Gal_{F_{\wp}}, B_E(\chi))\lra E.$$
We claim $\fj'$ induces an isomorphism $H^1(\Gal_{F_{\wp}},B_E(\eta'))/H^1_e(\Gal_{F_{\wp}},B_E(\eta')) \xrightarrow{\sim} H^1_{g,S_c(\chi)}(\Gal_{F_{\wp}},B_E)$, from which one can easily deduce (\ref{equ: lpl-e1c}) is perfect. Since $H^1_e(\Gal_{F_{\wp}},B_E)=\{0\}$, $H^1_e(\Gal_{F_{\wp}},B_E(\eta'))\subseteq \Ker (\fj')$ \big(note $\fj'$ sends $H^1_e(\Gal_{F_{\wp}},B_E(\eta'))$ to $H^1_e(\Gal_{F_{\wp}}, B_E)=0$\big). By Lem.\ref{lem: lpl-dyp}, \begin{equation*}\dim_E \Ima(\fj')=\dim_EH^1_{g,S_c(\chi)}(\Gal_{F_{\wp}},B_E)= |S_n(\chi)|+1;\end{equation*} by Lem.\ref{lem: lpl-sit},  $\dim_E H^1(\Gal_{F_{\wp}},B_E(\eta'))=d$ and $\dim_E H^1_e(\Gal_{F_{\wp}},B_E(\eta'))=|S_c(\chi)|-1$.
By dimension calculation, the claim follows.

The second part follows from (\ref{equ: lpl-eig}) the fact $\fj'$ sends $H^1_g(\Gal_{F_{\wp}},B_E(\eta'))$ to $H^1_g(\Gal_{F_{\wp}},B_E)$.
\end{proof}
%\begin{lemma}
%  Keep the above notation, then $\Ima(\fj')=H^1_{g,S_c(\chi)}(\Gal_{F_{\wp}},B_E)$.
%\end{lemma}
%\begin{proof}
%  The morphism $\fj'$ induces an exact sequence of $\Gal_{F_{\wp}}$-complexes
%  \begin{multline}
%  0 \lra [B_E(\eta')_e \oplus B_E(\eta')_{\dR}^+ \ra B_E(\eta')_{\dR}] \lra [(B_E)_e \oplus (B_E)_{\dR}^+\ra (B_E)_{\dR}]\\ \lra [\oplus_{\sigma\in S_c(\chi)} B_{\dR,\sigma}^+/t^{1-k_{\sigma}}B_{\dR,\sigma}^+ \ra 0] \lra 0,
%\end{multline}
%from which one deduces
%\begin{equation*}
%  H^1(\Gal_{F_{\wp}},B_E(\eta'))\xrightarrow{\fj'} H^1(\Gal_{F_{\wp}},B_E) \ra H^1(\Gal_{F_{\wp}}, \oplus_{\sigma\in S_c(\chi)} B_{\dR,\sigma}^+/t^{1-k_{\sigma}}) \ra 0.
%\end{equation*}
%The lemma follows.
%\end{proof}
%By Lem.\ref{lem: lpl-dyp}, $\dim_E \Ima(\fj')=|S_n(\chi)|+1$; by Lem.\ref{lem: lpl-sit} (2),  $\dim_E H^1_e(\Gal_{F_{\wp}},B_E(\eta'))=|S_c(\chi)|-1$. Moreover, since $H^1_e(\Gal_{F_{\wp}},B_E)=\{0\}$, hence $H^1_e(\Gal_{F_{\wp}},B_E(\eta'))\subseteq \Ker (\fj')$ \big(note $\fj'$ restricts to a map $H^1_e(\Gal_{F_{\wp}},B_E(\eta')) \ra H^1_e(\Gal_{F_{\wp}}, B_E)$\big). By dimension calculation, one sees $\fj'$ induces an isomorphism $\fj': H^1(\Gal_{F_{\wp}},B_E(\eta'))/H^1_e(\Gal_{F_{\wp}},B_E(\eta')) \xrightarrow{\sim} H^1_{g,S_c(\chi)}(\Gal_{F_{\wp}},B_E)$. Thus one obtains a perfect pairing
%\begin{equation}\label{equ: lpl-e1c}
%  \langle\cdot, \cdot\rangle: H^1_{g,S_c(\chi)}(\Gal_{F_{\wp}},B_E) \times H^1_g(\Gal_{F_{\wp}},B_E(\chi)) \lra E, \ \langle x, y\rangle:=\langle x, y^{\sharp} \rangle .
%\end{equation}
Using this pairing and Lem.\ref{lem: lpl-dyp}, one can now define Fontaine-Mazur $\cL$-invariants in general case:
\begin{definition}[general case]\label{def: lpl-eln}Let $\chi$ be a special character of $F_{\wp}^{\times}$, $[X]\in H^1_g(\Gal_{F_{\wp}},B_E(\chi))$, if $[X]\notin H^1_e(\Gal_{F_{\wp}},B_E(\chi))$, for $\sigma\in S_n(\chi)$, put $\cL(X)_{\sigma}:=\langle[X], \psi_{\sigma,p}\rangle/\langle[X],\psi_{\ur}\rangle\in E$ (cf. (\ref{equ: lpl-e1c})), and $\{\cL(X)_{\sigma}\}_{\sigma\in S_n(\chi)}$ are called the Fontaine-Mazur $\cL$-invariants of $X$; if $[X]\in H^1_e(\Gal_{F_{\wp}},B_E(\chi))$, we define the Fontaine-Mazur $\cL$-invariants of $X$ to be $(\cL(X)_{\sigma})_{\sigma\in S_n(\chi)}:=(\langle[X],\psi_{\sigma,p}\rangle)_{\sigma\in S_n(\chi)} \in \bP^{|S_n(\chi)|}(E)$ (cf. (\ref{equ: lpl-e1c})).
\end{definition}
\begin{remark}\label{rem: lpl-tng}
  Keep the above notation, and let $[X]\in H^1_g(\Gal_{F_{\wp}}, B_E(\chi))$, $[X^{\sharp}]\in H^1(\Gal_{F_{\wp}}, B_E(\chi^{\sharp}))$ with $\fj([X^{\sharp}])=[X]$, thus  $(\cL(X))_{\sigma\in S_n(\chi)}=(\cL(X^{\sharp}))_{\sigma\in S_n(\chi)}$.
\end{remark}
Let $X$ be a $2$-dimensional triangulable $E$-$B$-pair with a triangulation given by
\begin{equation*}
  0 \ra B_E(\chi_1) \ra X \ra B_E(\chi_2) \ra 0.
\end{equation*}
We denote by $(X,\chi_1,\chi_2)$ a such triangulation. The $E$-$B$-pair $X$ is called \emph{special} if $\chi_1\chi_2^{-1}$ is special.
%, this notion is in fact independent of the choice of triangulations of $X$.
 Suppose $X$ is special, let $[X_0]\in H^1(\Gal_{F_{\wp}},B_E(\chi_1\chi_2^{-1}))$ be the image of $[X]$ via the isomorphism $\Ext^1(B_E(\chi_1),B_E(\chi_2)) \xrightarrow{\sim} H^1(\Gal_{F_{\wp}},B_E(\chi_1\chi_2^{-1}))$. If $[X_0]\in H_g^1(\Gal_{F_{\wp}},B_E(\chi_1\chi_2^{-1}))$, we define the $\cL$-invariants of $(X,\chi_1,\chi_2)$ to be the $\cL$-invariants of $[X_0]$; if moreover  $[X_0]\notin  H_e^1(\Gal_{F_{\wp}},B_E(\chi_1\chi_2^{-1}))$, these are called $\cL$-invariants of $X$ (since in this case, $X$ admits a unique triangulation, cf. \cite[Thm.3.7]{Na}).

 Let $V$ be a $2$-dimensional semi-stable representation of $\Gal_{F_{\wp}}$ over $E$, and
 \begin{equation*}
   0 \ra B_E(\chi_1) \ra W(V) \ra B_E(\chi_2) \ra 0
 \end{equation*}
 a triangulation of $W(V)$. Suppose $\chi_1\chi_2^{-1}$ is special, which is equivalent to that the eigenvalues $\alpha_1$, $\alpha_2$ of the $E$-linear operator $\varphi^{d_0}$ on $D_{\st}(V)$ satisfy $\alpha_1\alpha_2^{-1}=q$ or $q^{-1}$. One defines the $\cL$-invariants of $(V,\chi_1,\chi_2)$ to be the $\cL$-invariants of $(W(V),\chi_1,\chi_2)$, which are called the\emph{ Fontaine-Mazur $\cL$-invariants} of $V$ if $V$ is moreover non-crystalline.
%\subsection{$\cL$-invariants in non-critical case}
\section{$\cL$-invariants and partially de Rham families}\label{sec: lpl-2}
Let $\chi$ be a special character of $F_{\wp}^{\times}$ in $E^{\times}$, $\widetilde{\chi}$ be a character of $F_{\wp}^{\times}$ in $(E[\epsilon]/\epsilon^2)^{\times}$ such that $\widetilde{\chi}\equiv \chi \pmod{\epsilon}$. So there exists an additive character $\psi$ of $F_{\wp}^{\times}$ in $E$ such that $\widetilde{\chi}=\chi(1+\epsilon \psi)$. By results in \S \ref{sec: lpl-1.3.1}, there exist $a_\sigma\in E$ for all $\sigma\in \Sigma_{\wp}$ and $a_{\ur}\in E$ such that $\psi=a_{\ur} \psi_{\ur}+\sum_{\sigma\in \Sigma_{\wp}} a_{\sigma}\psi_{\sigma,p}$.

 Let $X$ be  an $E[\epsilon]/\epsilon^2$-$B$-pair of rank $2$ such that
 $$[X]\in H^1\big(\Gal_{F_{\wp}},B_{E[\epsilon]/\epsilon^2}\big(\widetilde{\chi}\big)\big)\cong \Ext^1\big(B_{E[\epsilon]/\epsilon^2}, B_{E[\epsilon]/\epsilon^2}\big(\widetilde{\chi}\big)\big).$$
  Denote by $X_0:=X\pmod{\epsilon}$, which is a triangulable $E$-$B$-pair and lies in $H^1(\Gal_{F_{\wp}},B_E(\chi))$. Suppose $X_0$ is de Rham \big(i.e. $[X_0]\in H^1_g(\Gal_{F_{\wp}},B_E(\chi))$\big), and denote by $\ul{\cL}_{S_n(\chi)}=(\cL_{\sigma})_{\sigma\in S_n(\chi)}$ the associated $\cL$-invariants (cf. Def.\ref{def: lpl-eln}). This section is devoted to prove the following theorem.
\begin{theorem}\label{thm: lpl-fee}
  Keep the above notation, and suppose $X$ is $S_c(\chi)$-de Rham (cf. Def.\ref{def: lpl-pnt}), then
  \begin{equation}\label{equ: lpl-lre}
    \begin{cases}
   a_{\ur}+  \sum_{\sigma\in S_n(\chi)}a_{\sigma} \cL_{\sigma}=0& \text{if $X_0$ is non-crystalline,} \\
      \sum_{\sigma\in S_n(\chi)}a_{\sigma} \cL_{\sigma}=0& \text{if $X_0$ is crystalline.}
    \end{cases}
  \end{equation}
\end{theorem}
\begin{remark}\label{rem: lpl-mle}(1) Such formula was firstly established by Greenberg-Stevens \cite[Thm.3.14]{GS93} in the case of $2$-dimensional ordinary $\Gal_{\Q_p}$-representations by Galois cohomology computations. In \cite{Colm10}, Colmez generalized \cite[Thm.3.14]{GS93} to $2$-dimensional trianguline $\Gal_{\Q_p}$-representations case by Galois cohomology computations and computations in Fontaine's rings. The theorem \ref{thm: lpl-fee} in non-critical case (i.e. $S_c(\chi)=\emptyset$) was obtained by Zhang in \cite{Zhang}, by generalizing Colmez's method. In \cite{Pot}, Pottharst generalized \cite[Thm.3.14]{GS93} to rank $2$ triangulable $(\varphi,\Gamma)$-modules (in $\Q_p$ case) by studying cohomology of $(\varphi,\Gamma)$-modules.

%s.

(2) The hypothesis $X$ being $S_c(\chi)$-de Rham would imply that $a_{\sigma}=0$ for all $\sigma\in S_c(\chi)$. In fact, $X$ being $S_c(\chi)$-de Rham implies $B_{E[\epsilon]/\epsilon^2}(\widetilde{\chi})$ being $S_c(\chi)$-de Rham. We claim that $B_{E[\epsilon]/\epsilon^2}(\widetilde{\chi})$ is $S_c(\chi)$-de Rham if an only if $a_{\sigma}=0$ for all $\sigma\in S_c(\chi)$. Indeed, it's easy to see $B_{E[\epsilon]/\epsilon^2}(\widetilde{\chi})$ is $S_c(\chi)$-de Rham if and only if $B_{E[\epsilon]/\epsilon^2}(1+\psi \epsilon)$ is $S_c(\chi)$-de Rham. Viewing $B_{E[\epsilon]/\epsilon^2}(1+\psi\epsilon)$ as an extension of $B_E$ by $B_E$ defined by $\psi \in H^1(\Gal_{F_{\wp}},B_E)$ (cf. \S \ref{sec: lpl-1.3.1}), we see $B_{E[\epsilon]/\epsilon^2}(1+\psi\epsilon)$ is $S_c(\chi)$-de Rham if and only $\psi\in H^1_{g,S_c(\chi)}(\Gal_{F_{\wp}},B_E)$,  which is equivalent to that $a_{\sigma}=0$ for all $\sigma\in S_c(\chi)$ by Lem.\ref{lem: lpl-dyp}. However, the converse is not true. This is a new subtlety: the formulas in (\ref{equ: lpl-lre}) do \emph{not} hold (in general) if one only assumes $a_{\sigma}=0$ for all $\sigma\in S_c(\chi)$ (e.g. see the discussion before Lem.\ref{lem: lpl-1gb} below).
%This is a subtlety in $F_{\wp}\neq \Q_p$-case. %The theorem might shed light on the study of $p$-adic $F_{\wp}$-functions in $L\neq \Q_p$-case.
 %$M_A$ being $S_{\crit}(\delta_z)$-de Rham up to twists of characters implies that  $d \fa_{\sigma}=0$ for all $\sigma\in S_{\crit}(\delta_z)$; in other words, being $S_{\crit}(\delta_z)$-de Rham  implies being $S_{\crit}(\delta_z)$-Hodge-Tate. The converse is not true. In fact, these formulas do\emph{ not} hold in general if one only assumes $d \fa_{\sigma}=0$ for all $\sigma\in S_{\crit}(\delta_z)$.

 %This theorem might also shed light on the study of $p$-adic $L$-functions in the critical case.
\end{remark}
We translate this theorem in terms of families of Galois representations. Let $A$ be an affinoid $E$-algebra, $V$ be a locally free $A$-module of rank $2$ equipped with a continuous $\Gal_{F_{\wp}}$-action. Thus $D_{\dR}(V):=(B_{\dR} \widehat{\otimes}_{\Q_p} V)^{\Gal_{F_{\wp}}}$ is an $A \otimes_{\Q_p} F_{\wp}$-module. Using $A\otimes_{\Q_p} F_{\wp} \xrightarrow{\sim} \prod_{\sigma\in \Sigma_{\wp}} A$, $a \otimes b \mapsto (a\sigma(b))_{\sigma}$, one can decompose $D_{\dR}(V)\xrightarrow{\sim} \oplus_{\sigma\in \Sigma_{\wp}} D_{\dR}(V)_{\sigma}$. For $\sigma\in \Sigma_{\wp}$, we say $V$ is \emph{$\sigma$-de Rham} if $D_{\dR}(V)_{\sigma}$ is locally free of rank $2$ over $A$. Let $\cR_A:=\cR_E  \widehat{\otimes}_E  A$, one can associate to $V$ a $(\varphi,\Gamma)$-module $D_{\rig}(V)$ (cf. \cite[Thm.2.2.17]{KPX}) over $\cR_A$. Suppose $D_{\rig}(V)$ sits in an exact sequence of $(\varphi,\Gamma)$-modules over $\cR_A$ as follows:
\begin{equation*}
  0 \lra \cR_A(\delta_1) \lra D_{\rig}(V) \lra \cR_A(\delta_2) \lra 0,
\end{equation*}
where $\delta_i: F_{\wp}\ra A^{\times}$ are continuous characters, and we refer to \cite[Const.6.2.4]{KPX} for rank $1$ $(\varphi,\Gamma)$-modules associated to characters. For a continuous character $\chi$ of $F_{\wp}^{\times}$ in $A^{\times}$, $\chi$ induces a $\Q_p$-linear map
\begin{equation*}
 d\chi: F_{\wp} \lra A, \ a\mapsto \frac{d}{dx} \chi(\exp(ax))|_{x=0},
\end{equation*}
and thus an $E$-linear map $d\chi: F_{\wp}\otimes_{\Q_p} E \cong \prod_{\sigma\in \Sigma_{\wp}} E \ra A$. So there exists $(\wt(\chi)_{\sigma})_{\sigma\in \Sigma_{\wp}}\in A^{|\Sigma_{\wp}|}$, called the weight of $\chi$,  such that $d\chi \big((a_{\sigma})_{\sigma\in \Sigma_{\wp}}\big)=\sum_{\sigma\in \Sigma_{\wp}} a_{\sigma} \wt(\chi)_{\sigma}$. Let $z$ be an $E$-point of $A$, and $\delta_{i,z}:=z^* \delta_i$, suppose
\begin{itemize}
  \item $V_z:=z^* V$ is semi-stable;
  \item $\delta_{1,z}\delta_{2,z}^{-1}$ is special.
\end{itemize}
Put $S_n(V_z):=S_n(\delta_{1,z}\delta_{2,z}^{-1})$, $S_c(V_z):=S_c(\delta_{1,z}\delta_{2,z}^{-1})$ (cf. (\ref{equ: lpl-scia})),  and $\ul{\cL}_{S_n(V_z)}$ the Fontaine-Mazur $\cL$-invariants of $V_z$. By Thm.\ref{thm: lpl-fee}, one has
\begin{corollary}\label{cor: lpl-rvs}
  Keep the above notation, suppose moreover $V$ is $S_c(V_z)$-de Rham, then the differential form in $\Omega_{A/E}^1$
  \begin{equation*}
    \begin{cases}
      d\log \big(\delta_1\delta_2^{-1}(p)\big) + \sum_{\sigma\in S_n(V_z)} \cL_{\sigma} d\big(\wt(\delta_1\delta_2^{-1})_{\sigma}\big) & \text{if $V_z$ is non-crystalline} \\
      \sum_{\sigma\in S_n(V_z)} \cL_{\sigma} d\big(\wt(\delta_1\delta_2^{-1})_{\sigma}\big) & \text{if $V_z$ is crystalline}
    \end{cases}
  \end{equation*}
  vanishes at $z$.
\end{corollary}
\begin{remark}
Partially de Rham families would appear naturally in the study of $p$-adic automorphic forms, e.g.  one encounters such families when studying locally analytic vectors in completed cohomology of Shimura curves (see Prop.\ref{prop: lpl-uoc} below), or certain families of overconvergent Hilbert modular forms (see App.\ref{sec: lpl-app} below).  Note this formula also applies for families of $F_{\wp}$-analytic $\Gal_{F_{\wp}}$-representations (cf. \cite{Ber14}, which can be viewed as special cases of partially de Rham families).
\end{remark}
The rest of this section is devoted to the proof of Thm.\ref{thm: lpl-fee}. We use Pottharst's method \cite{Pot} (but in terms of $B$-pairs). It's clear that  $X$ being $S_c(\chi)$-de Rham is equivalent to saying that $B_{E[\epsilon]/\epsilon^2}(\widetilde{\chi})$ is $S_c(\chi)$-de Rham and $[X]\in H^1_{g,S_c(\chi)}(\Gal_{F_{\wp}}, B_{E[\epsilon]/\epsilon^2}(\widetilde{\chi}))$. As discussed in Rem.\ref{rem: lpl-mle}(2), $B_{E[\epsilon]/\epsilon^2}(\widetilde{\chi})$ being $S_c(\chi)$-de Rham is equivalent to that $a_{\sigma}=0$ for all $\sigma\in S_c(\chi)$. Thus it's sufficient to prove
\begin{proposition}\label{prop: lpl-act}
Suppose $a_{\sigma}=0$ for all $\sigma\in S_c(\chi)$ and $[X]\in H^1_{g,S_c(\chi)}(\Gal_{F_{\wp}},B_{E[\epsilon]/\epsilon^2}(\widetilde{\chi}))$, then the formulas  in (\ref{equ: lpl-lre}) hold.
\end{proposition}
 Let $k_{\sigma}\in \Z$ for all $\sigma\in \Sigma_{\wp}$ such that $\chi=\unr(q^{-1})\prod_{\sigma\in \Sigma_{\wp}} \sigma^{k_{\sigma}}$. One has a natural exact sequence of $E$-$B$-pairs
\begin{equation}\label{equ: lpl-eaa}
  0 \ra B_E(\chi) \ra B_{E[\epsilon]/\epsilon^2}(\widetilde{\chi}) \ra B_E(\chi)\ra 0,
\end{equation}
by taking cohomology, one gets an exact sequence
\begin{equation}\label{equ: lpl-ectn}
  0 \ra H^1(\Gal_{F_{\wp}},B_E(\chi)) \ra H^1(\Gal_{F_{\wp}},B_{E[\epsilon]/\epsilon^2}(\widetilde{\chi})) \xrightarrow{\kappa} H^1(\Gal_{F_{\wp}},B_E(\chi)) \xrightarrow{\delta} H^2(\Gal_{F_{\wp}},B_E(\chi)).
\end{equation}
Note $\kappa([X])=[X_0]$. We suppose $\psi\neq 0$ (since the case $\psi=0$ is trivial).

First consider non-critical case (i.e. $S_c(\chi)=\emptyset$), thus $k_{\sigma}\in \Z_{\geq 1}$ for all $\sigma\in \Sigma_{\wp}$. In this case, one has $H^1(\Gal_{F_{\wp}},B_E(\chi))=H^1_g(\Gal_{F_{\wp}},B_E(\chi))$, which is of dimension $d+1$ over $E$. One also has
\begin{lemma}\label{equ: lpl-ebn}$\dim_E H^1(\Gal_{F_{\wp}},B_{E[\epsilon]/\epsilon^2}(\widetilde{\chi}))=2d+1$.
\end{lemma}
\begin{proof}
  Let $W:=B_{E[\epsilon]/\epsilon^2}(\widetilde{\chi})$ for simplicity, one has $\sum_{i=0}^{2} (-1)^i H^1(\Gal_{F_{\wp}},W)=-2d$. It's easy to see $H^0(\Gal_{F_{\wp}},W)=0$ (cf. \cite[Prop.2.14]{Na}); moreover one has $\dim_E H^0(\Gal_{F_{\wp}},W^{\vee}(1))=1$: let $\eta:=\prod_{\sigma\in \Sigma_{\wp}} \sigma^{1-k_{\sigma}}$, then $W^{\vee}(1)$ is an extension of $B_E(\eta)$ by $B_E(\eta)$ (defined by $\psi$), by \cite[Prop.2.14]{Na}, $\dim_E H^0(\Gal_{F_{\wp}}, B_E(\eta))=1$, which together with the fact $\psi \neq 0$ deduces then $\dim_E H^0(\Gal_{F_{\wp}},W^{\vee}(1))=1$. By the duality between $H^0(\Gal_{F_{\wp}},W^{\vee}(1))$ and $H^2(\Gal_{F_{\wp}},W)$, one sees $\dim_E H^2(\Gal_{F_{\wp}},W)=1$, so $H^1(\Gal_{F_{\wp}},W)=2d+1$.
\end{proof}
In particular, the map $\kappa$ is not surjective. On the other hand, by Rem.\ref{rem: lpl-arw} (1) \big(applied to $W_1=W_3=B_E$, $W_2=B_{E[\epsilon]/\epsilon^2}(1+\epsilon \psi)$, $W=B_E(\chi)$\big), we see the map $\delta$ is given (up to scalars) by $x\mapsto \langle x,\psi \rangle$, where $\langle \cdot, \cdot\rangle$ denotes the cup-product $H^1(\Gal_{F_{\wp}},B_E(\chi)) \times H^1(\Gal_{F_{\wp}},B_E) \ra H^2(\Gal_{F_{\wp}},B_E(\chi))$ (cf. (\ref{equ: lpl-ice})). So one has (since $[X_0]=\kappa([X])$)
\begin{equation}\label{equ: lpl-ana}\delta([X_0])= \sum_{\sigma\in \Sigma_{\wp}} a_{\sigma}\langle [X_0], \psi_{\sigma,p}\rangle +a_{\ur} \langle [X_0], \psi_{\ur}\rangle=0,\end{equation}
from which we deduces Prop.\ref{prop: lpl-act} in non-critical case by the definition of $\cL$-invariants of $X_0$ (cf. Def.\ref{equ: lpl-eie}).

Suppose now $S_c(\chi)\neq \emptyset$, in this case $\dim_E H^1(\Gal_{F_{\wp}},B_E(\chi))=d$. One can show as in the proof of  Lem.\ref{equ: lpl-ebn} that  $\dim_E H^1(\Gal_{F_{\wp}},B_{E[\epsilon]/\epsilon^2}(\widetilde{\chi}))=2d$, so $\kappa$ is surjective (cf. (\ref{equ: lpl-ectn})). Consequently, one can \emph{not} expect any formula without further condition on $X$. %However, one can also deduce from (\ref{equ: lpl-eaa}) a sequence
%\begin{equation*}
% H^1_{g,S_c(\chi)}(\Gal_{F_{\wp}},B_E(\chi)) \ra H^1_{g,S_c(\chi)}(\Gal_{F_{\wp}},B_{E[\epsilon]/\epsilon^2}(\widetilde{\chi})) \xrightarrow{\kappa_g} H^1_{g,S_c(\chi)}(\Gal_{F_{\wp}},B_E(\chi))
%\end{equation*}
%with the first map injective.

 As in \S \ref{sec: lpl-1.3.2},  put $\chi^{\sharp}:=\chi \prod_{\sigma\in S_c(\chi)} \sigma^{1-k_{\sigma}}$, and $\widetilde{\chi}^{\sharp}:=\chi^{\sharp} (1+\epsilon \psi)=\widetilde{\chi}\prod_{\sigma\in S_c(\chi)} \sigma^{1-k_{\sigma}}$. Note $B_{E[\epsilon]/\epsilon^2}\big(\widetilde{\chi}\big)$ and $B_{E[\epsilon]/\epsilon^2}\big(\widetilde{\chi}^{\sharp}\big)$ are both $S_c(\chi)$-de Rham (see Rem. \ref{rem: lpl-mle} (2)). By (\ref{equ: lpl-jgw}), one has an exact sequence
 \begin{multline}
   0 \ra H^0\big(\Gal_{F_{\wp}}, B_{E[\epsilon]/\epsilon^2}\big(\widetilde{\chi}\big)\big) (=0) \ra \oplus_{\sigma\in S_c(\chi)} H^0\big(\Gal_{F_{\wp}}, B_{E[\epsilon]/\epsilon^2}\big(\widetilde{\chi}\big)_{\dR,\sigma}^+/t^{1-k_{\sigma}}\big) \\
   \ra H^1\big(\Gal_{F_{\wp}}, B_{E[\epsilon]/\epsilon^2}\big(\widetilde{\chi}^{\sharp}\big)\big) \ra H^1_{g,S_c(\chi)}\big(\Gal_{F_{\wp}}, B_{E[\epsilon]/\epsilon^2}\big(\widetilde{\chi}\big)\big)\ra 0,
 \end{multline}
 from which (and Lem.\ref{equ: lpl-ebn}) one calculates:
 \begin{lemma}\label{lem: lpl-1gb}
  $\dim_E H^1_{g,S_c(\chi)}(\Gal_{F_{\wp}},B_{E[\epsilon]/\epsilon^2}(\widetilde{\chi}))=2d+1-2|S_c(\chi)|$.
 \end{lemma}
The commutative diagram of $E$-$B$-pairs
 \begin{equation*}
   \begin{CD}
     0 @>>> B_E(\chi^{\sharp}) @>>> B_E(\widetilde{\chi}^{\sharp}) @>>> B_E(\chi^{\sharp}) @>>> 0 \\
     @. @V \fj VV @V \fj VV @V\fj VV @. \\
     0 @>>> B_E(\chi) @>>> B_E(\widetilde{\chi}) @>>> B_E(\chi) @>>> 0
   \end{CD}
 \end{equation*}
 induces a commutative diagram (by (\ref{equ: lpl-jgw}))
\begin{equation*}
  \begin{CD}
H^1(\Gal_{F_{\wp}},B_E(\chi^{\sharp})) @>>> H^1(\Gal_{F_{\wp}},B_{E[\epsilon]/\epsilon^2}(\widetilde{\chi}^{\sharp})) @> \kappa >> H^1(\Gal_{F_{\wp}},B_E(\chi^{\sharp})) %@> \delta >> H^2(\Gal_{F_{\wp}},B_E(\chi^\sharp ))
\\
 @V \fj VV @V \fj VV @V \fj VV  \\
 H^1_{g,S_c(\chi)}(\Gal_{F_{\wp}},B_E(\chi)) @>>> H^1_{g,S_c(\chi)}(\Gal_{F_{\wp}},B_{E[\epsilon]/\epsilon^2}(\widetilde{\chi})) @> \kappa_g >> H^1_{g,S_c(\chi)}(\Gal_{F_{\wp}},B_E(\chi))
  \end{CD},
\end{equation*}
where all the vertical arrows are surjective, the two horizontal maps on the left are injective, and the top sequence is exact. Note by Lem.\ref{lem: lpl-fpi} and Lem.\ref{lem: lpl-sit}, $H^1_{g,S_c(\chi)}(\Gal_{F_{\wp}},B_E(\chi))=H^1_g(\Gal_{F_{\wp}},B_E(\chi))$ is of dimension $d+1-|S_c(\chi)|$. Since $\dim_E \Ima(\kappa)=d=\dim_E H^1(\Gal_{F_{\wp}},B_E(\chi^{\sharp}))-1$, one has $\dim_E \Ima(\kappa_g)\geq \dim_E H^1_{g,S_c(\chi)}(\Gal_{F_{\wp}},B_E(\chi))-1=d-|S_c(\chi)|$, which together with Lem.\ref{lem: lpl-1gb} shows that the bottom sequence is also exact and that $\dim_E \Ima(\kappa_g)= \dim_E H^1_{g,S_c(\chi)}(\Gal_{F_{\wp}},B_E(\chi))-1$ (in particular, $\kappa_g$ is not surjective).

Denote by $[X^{\sharp}]\in H^1\big(\Gal_{F_{\wp}},B_{E[\epsilon]/\epsilon^2}\big(\widetilde{\chi}^{\sharp}\big)\big)$ the preimage of $[X]\in H^1_{g,S_c(\chi)}\big(\Gal_{F_{\wp}},B_{E[\epsilon]/\epsilon^2}\big(\widetilde{\chi}\big)\big)$ via $\fj$, $[X_0^{\sharp}]:=\kappa([X^{\sharp}])$, thus $\fj([X_0^{\sharp}])=[X_0]$.  By (\ref{equ: lpl-ana}) applied to $[X_0^{\sharp}]$, one has (note that $a_{\sigma}=0$ for $\sigma\in S_c(\chi)$)
\begin{equation*}
  \sum_{\sigma\in \Sigma_{\wp}} a_{\sigma}\langle [X_0^{\sharp}], \psi_{\sigma,p}\rangle +a_{\ur} \langle [X_0^{\sharp}], \psi_{\ur}\rangle=0,
\end{equation*}
from which, together with the  definition of $\cL$-invariants for $[X_0]$ (Def.\ref{def: lpl-eln}, see in particular Rem.\ref{rem: lpl-tng}), Prop.\ref{prop: lpl-act} follows.
\section{Breuil's $\cL$-invariants}\label{sec: lpl-3}
In \cite{Br04}, to a $2$-dimensional semi-stable non-crystalline representation $V$ of $\Gal_{\Q_p}$, Breuil associated a locally analytic representation $\Pi(V)$ of $\GL_2(\Q_p)$ (Breuil also considered Banach representations, but we only focus on locally analytic representations in this paper), which can determine $V$ and in particular contains the information on the  Fontaine-Mazur $\cL$-invariant of $V$. Roughly speaking, Breuil found that certain extensions of locally analytic representations \big(of $\GL_2(\Q_p)$\big) can be parameterized by some invariants (which are referred as to \emph{Breuil's $\cL$-invariants}), and by matching these invariants with Fontaine-Mazur $\cL$-invariants, one could get a one-to-one correspondence \big(in $p$-adic Langlands for $\GL_2(\Q_p)$\big) in semi-stable non-crystalline case. In \cite{Sch10}, generalizing Breuil's theory, Schraen associated a locally $\Q_p$-analytic representation of $\GL_2(F_{\wp})$ to a $2$-dimensional semi-stable non-crystalline representation of $\Gal_{F_{\wp}}$ (although only the non-critical case was considered in \emph{loc. cit.}, Schraen's construction can easily generalize to critical case). We recall some results of \emph{loc. cit.} in this section.

Let $V$ be a $2$-dimensional semi-stable non-crystalline representation of $\Gal_{F_{\wp}}$ over $E$ of distinct Hodge-Tate weights $\ul{h}_{\Sigma_p}:=(k_{1,\sigma},k_{2,\sigma})_{\sigma\in \Sigma_{\wp}}$ ($k_{1,\sigma}<k_{2,\sigma}$, we use the convention that the Hodge-Tate weight of the cyclotomic character is $-1$), denote by $\alpha$, $q\alpha$ the eigenvalues of $\varphi^{d_0}$ on $D_{\st}(V):=\big(B_{\st} \otimes_{\Q_p}V\big)^{\Gal_{F_{\wp}}}$.
By \cite[\S 4.3]{Na}, the $E$-$B$-pair $W(V)$ admits a unique triangulation:
\begin{equation*}
  0 \ra W(\delta_1) \ra W(V) \ra W(\delta_2)\ra 0
\end{equation*}
where $\delta_1=\unr(\alpha)\prod_{\sigma\in S_n} \sigma^{k_{1,\sigma}} \prod_{\sigma\in S_c} \sigma^{k_{2,\sigma}}$, $\delta_2=\unr(q\alpha) \prod_{\sigma\in S_n} \sigma^{k_{2,\sigma}} \prod_{\sigma\in S_c}\sigma^{k_{1,\sigma}}$, $S_n$ is a subset of $\Sigma_{\wp}$, and $S_c=\Sigma_{\wp}\setminus S_n$. In fact, $S_c=S_c(\delta_1\delta_2^{-1})$ is the set of embeddings where $V$ is critical (cf. Def.\ref{def: pdeR-lir} below). Since $V$ is semi-stable non-crystalline, so is $W(V)$, one can thus associate to $V$ the Fontaine-Mazur $\cL$-invariants $\ul{\cL}_{S_n}\in E^{|S_n|}$ (see the end of \S \ref{sec: lpl-1.3}).

For $S\subseteq \Sigma_{\wp}$, let $\ul{h}_S:=(k_{1,\sigma},k_{2,\sigma})_{\sigma\in S}$ and put
\begin{equation}\label{equ: lpl-vssa}
\alg(\ul{h}_S):=\otimes_{\sigma\in S} \big(\Sym^{k_{2,\sigma}-k_{1,\sigma}-1} E^2 \otimes_E \dett^{-k_{2,\sigma}+1}\big)^{\sigma}\cong \Big(\otimes_{\sigma \in S}\big(\Sym^{k_{2,\sigma}-k_{1,\sigma}-1} E^2 \otimes_E \dett^{k_{1,\sigma}}\big)^{\sigma}\Big)^{\vee},
\end{equation} which is an irreducible algebraic representation of $\Res_{F_{\wp}/\Q_p} \GL_2$ with the action of $\GL_2(F_{\wp})$ on $(\cdot)^{\sigma}$ induced by the natural action of $\GL_2(E)$ via $\sigma$.  Put
\begin{equation}\label{equ: lpl-ado}
\chi(\alpha,\ul{h}_{S}):= \unr(\alpha) \prod_{\sigma\in S} \sigma^{-k_{1,\sigma}} \otimes \unr(\alpha) \prod_{\sigma \in S}\sigma^{-k_{2,\sigma}+1} ,
\end{equation}
which is a locally $S$-analytic character of $T(F_{\wp})$ over $E$. Consider the locally $S$-analytic parabolic induction (cf. \cite[\S 2.3]{Sch10})
$$I(\alpha,\ul{h}_{S}):=\big(\Ind_{\overline{B}(F_{\wp})}^{\GL_2(F_{\wp})} \chi(\alpha,\ul{h}_{S})\big)^{S-\an},$$
by \cite[Thm.4.1]{Br} (see also \cite[\S 2.3]{Sch10}), we have
\begin{enumerate}
  \item $\soc_{\GL_2(F_{\wp})}I(\alpha,\ul{h}_{S})\cong\big(\unr(\alpha)\circ \dett\big)\otimes_E \alg(\ul{h}_{S})=: F(\alpha,\ul{h}_{S})$; \\

  \item put $\Sigma(\alpha,\ul{h}_{S}):=I(\alpha,\ul{h}_{S})/F(\alpha,\ul{h}_{S})$, then
  \begin{equation*}\soc_{\GL_2(F_{\wp})} \Sigma(\alpha,\ul{h}_{S})\cong (\unr(\alpha) \circ \dett)\otimes_E \St \otimes_E \alg(\alpha,\ul{h}_{S})=:\St(\alpha,\ul{h}_{S}),\end{equation*} which is also the maximal locally algebraic subrepresentation of $\Sigma(\alpha,\ul{h}_{S})$, where $\St$ denotes the Steinberg representation;

  \item let $\sigma\in \Sigma_{\wp}$, $\chi(\alpha,\ul{h}_{\sigma})^c:=\unr(\alpha)\sigma^{-k_{1,\sigma}} \otimes \unr(\alpha) \sigma^{-k_{2,\sigma}+1}$, and $I_{\sigma}^c(\alpha,\ul{h}_{\sigma}):=\big(\Ind_{\overline{B}(F_{\wp})}^{\GL_2(F_{\wp})} \chi(\alpha,\ul{h}_{\sigma})^c\big)^{\sigma-\an}$ \big(which is irreducible by \cite[Thm.4.1]{Br}\big), then one has a non-split exact sequence
\begin{equation*}
  0 \ra \St(\alpha,\ul{h}_{\sigma}) \ra \Sigma(\alpha,\ul{h}_{\sigma}) \ra I_{\sigma}^c(\alpha,\ul{h}_{\sigma}) \ra 0.
\end{equation*}
\end{enumerate}

For $\cL\in E$ and $\sigma\in \Sigma_{\wp}$, let $\log_{\sigma,\cL}:=\psi_{\sigma,p}+\cL \psi_{\ur}$ (cf. \S \ref{sec: lpl-1.3.1}) which is thus the additive character of $F_{\wp}^{\times}$ in $E$ satisfying that $\log_{\sigma,\cL}|_{\co_{\wp}^{\times}}=\sigma \circ \log$ and that $\log_{\sigma,\cL}(p)=\cL$.

Let $d_n:=|S_n|$, and $\psi(\ul{\cL}_{S_n})$ be the following $(d_n+1)$-dimensional representation of $T(F_{\wp})$ over $E$
\begin{equation}\label{equ: lpl-noh}
    \psi(\ul{\cL}_{S_n})\begin{pmatrix}
      a & 0 \\ 0 & d
    \end{pmatrix} =\begin{pmatrix}
    1& \log_{\sigma_1, -\cL_{\sigma_1}}(ad^{-1}) & \log_{\sigma_2, -\cL_{\sigma_2}}(ad^{-1}) & \cdots &\log_{\sigma_{|d_n|}, -\cL_{\sigma_{d_n}}}(ad^{-1}) \\ 0 & 1 & 0 & \cdots & 0 \\ 0 & 0 & 1 &\cdots & 0\\
    \vdots & \vdots &\vdots &\ddots & \vdots \\
    0 & 0 & 0 & \cdots & 1
  \end{pmatrix},
\end{equation}
with $\sigma_i\in S_n$.  One gets thus an exact sequence of locally $S_n$-analytic representations of $\GL_2(F_{\wp})$:
\begin{equation}\label{equ: lpl-bpw2}
  0 \lra I(\alpha,\ul{h}_{S_n}) \lra \Big(\Ind_{\overline{B}(F_{\wp})}^{\GL_2(F_{\wp})} \chi(\alpha,\ul{h}_{S_n})\otimes_E \psi(\ul{\cL}_{S_n})\Big)^{S_n-\an} \xlongrightarrow{s}I(\alpha,\ul{h}_{S_n})^{\oplus d_n} \lra 0.
\end{equation}
Put $\Sigma\big(\alpha,\ul{h}_{S_n},\ul{\cL}_{S_n}\big):=s^{-1}\big(F(\alpha,\ul{h}_{S_n})^{\oplus d_n}\big)/F(\alpha,\ul{h}_{S_n})$, which is thus an extension of $d_n$-copies of $F(\alpha,\ul{h}_{S_n})$ by $\Sigma(\alpha,\ul{h}_{S_n})$:
\begin{equation*}
  \begindc{\commdiag}[30]
    \obj(0,2)[a]{$\Sigma\big(\alpha,\ul{h}_{S_n}\big)$}
    \obj(4,3)[b]{$F(\alpha,\ul{h}_{S_n})$}
    \obj(4,2)[c]{$F(\alpha,\ul{h}_{S_n})$}
  \obj(4,1)[d]{$\vdots$}
    \obj(4,0)[e]{$F(\alpha,\ul{h}_{S_n})$}
    \mor{a}{b}{$-\cL_{\sigma_1}$}[+1,2]
    \mor{a}{c}{$-\cL_{\sigma_2}$}[+1,2]
    \mor{a}{e}{$-\cL_{\sigma_{d_n}}$}[+1,2]
  \enddc
  .
\end{equation*}
\begin{remark}\label{rem: lpl-tagia}(1) Let $\ul{\cL}'_{S_n}\in E^{d_n}$, as in \cite[Prop.4.13]{Sch10}, one can show $\Sigma\big(\alpha,\ul{h}_{S_n},\ul{\cL}'_{S_n}\big)\cong \Sigma\big(\alpha, \ul{h}_{S_n}, \ul{\cL}_{S_n}\big)$ if and only if $\ul{\cL}'_{S_n}=\ul{\cL}_{S_n}$. In particular, one can recover the data $\{\alpha, \ul{h}_{S_n}, \ul{\cL}_{S_n}\}$ from $\Sigma\big(\alpha,\ul{h}_{S_n},\ul{\cL}_{S_n}\big)$.

(2) Let $\ul{h}'_{S_n}=(k_{1,\sigma}',k_{2,\sigma}')_{\sigma\in S_n}\in \Z^{2|S_n|}$ with $k_{1,\sigma}'-k_{1,\sigma}=k_{2,\sigma}'-k_{2,\sigma}=n_{\sigma}$, thus $\alg(\ul{h}'_{S_n})\cong \alg(\ul{h}_{S_n})\otimes_E \otimes_{\sigma\in S_n} \sigma\circ \dett^{n_{\sigma}}$. It's straightforward to see $\Sigma\big(\alpha,\ul{h}'_{S_n},\ul{\cL}_{S_n}\big)\cong \Sigma\big(\alpha,\ul{h}_{S_n},\ul{\cL}_{S_n}\big)\otimes_E \big(\otimes_{\sigma\in S_n} \sigma\circ \dett^{n_{\sigma}}\big)$.
%one can define a locally $S_n$-analytic representation $\Sigma_0(V,\ul{\cL'}_{S_n(V)})$ in the same way as $\Sigma_0\big(V,\ul{\cL}_{S_n(V)})$ by replacing $\ul{\cL}_{S_n(V)}$ by $\ul{\cL'}_{S_n(V)}$., one can show $\Sigma_0\big(V,\ul{\cL'}_{S_n(V)})\cong \Sigma_0\big(V,\ul{\cL}_{S_n(V)})$ if and only if $\ul{\cL'}_{S_n(V)}=\ul{\cL}_{S_n(V)}$.

(3) By replacing the terms $\log_{\sigma_i, -\cL_{\sigma_i}}(ad^{-1})$ in $\psi(\ul{\cL}_{S_n})$ by $\log_{\sigma_i, -\cL_{\sigma_i}}(ad^{-1}) + \chi_i \circ \dett$ with an arbitrary  locally $\sigma_i$-analytic (additive)  character $\chi_i$ of $F_{\wp}^{\times}$ in $E$, one can get a locally $\Q_p$-analytic representation $\Sigma\big(\alpha,\ul{h}_{S_n},\ul{\cL}_{S_n}\big)'$ in the same way as $\Sigma\big(\alpha,\ul{h}_{S_n},\ul{\cL}_{S_n}\big)$. By some cohomology arguments in \cite[\S 4.3]{Sch10} (see \cite[Lem.4.4]{Ding3}), one can actually  prove
\begin{equation}\label{equ: lpl-agia}
\Sigma\big(\alpha,\ul{h}_{S_n},\ul{\cL}_{S_n}\big)'\cong\Sigma\big(\alpha,\ul{h}_{S_n},\ul{\cL}_{S_n}\big).\end{equation}

 (4) For $\sigma\in S_n$, denote by $\psi(\cL_{\sigma})$ the following $2$-dimensional representation of $T(F_{\wp})$:
\begin{equation*}
    \psi(\cL_{\sigma})\begin{pmatrix}
      a & 0 \\ 0 & d
    \end{pmatrix} =\begin{pmatrix}
    1& \log_{\sigma, -\cL_{\sigma}}(ad^{-1})  \\ 0 & 1
  \end{pmatrix}.
\end{equation*}
One has thus an exact sequence
\begin{equation*}
  0 \lra I(\alpha,\ul{h}_{S_n}) \lra \Big(\Ind_{\overline{B}(F_{\wp})}^{\GL_2(F_{\wp})} \chi(\alpha,\ul{h}_{S_n})\otimes_E \psi(\cL_{\sigma})\Big)^{S_n-\an} \xlongrightarrow{s_{\sigma}}I(\alpha,\ul{h}_{S_n}) \lra 0.
\end{equation*}
Put $\Sigma(\alpha,\ul{h}_{S_n},\cL_{\sigma}):=s_{\sigma}^{-1}\big(F(\alpha,\ul{h}_{S_n})\big)/F(\alpha,\ul{h}_{S_n})$,
%and $\Sigma(V,\cL_{\sigma}):=\Sigma_0(V,\cL_{\sigma})\otimes_E \alg(V,S_n(V)^c)$.
%then one has an isomorphism of locally $\Q_p$-analytic representations of $\GL_2(F_{\wp})$ (see the discussion in \cite[\S 4.2]{Sch10}):
the following isomorphism is straightforward:
\begin{equation}
  \Sigma(\alpha,\ul{h}_{S_n},\cL_{\sigma_1})\oplus_{\Sigma(\alpha,\ul{h}_{S_n})} \Sigma(\alpha,\ul{h}_{S_n},\cL_{\sigma_2})\oplus_{\Sigma(\alpha,\ul{h}_{S_n})}\cdots \oplus_{\Sigma(\alpha,\ul{h}_{S_n})}\Sigma(\alpha,\ul{h}_{S_n},\cL_{\sigma_{d_n}}) \xlongrightarrow{\sim} \Sigma\big(\alpha,\ul{h}_{S_n},\ul{\cL}_{S_n}\big).
\end{equation}

\end{remark}

Put $F(\alpha,\ul{h}_{\Sigma_{\wp}}):=F(\alpha,\ul{h}_{S_n})\otimes_E \alg(\ul{h}_{S_c})$, $\Sigma(\alpha,\ul{h}_{\Sigma_{\wp}}):=\Sigma(\alpha,\ul{h}_{S_n})\otimes_E \alg(\ul{h}_{S_c})$, and  $\Sigma\big(\alpha,\ul{h}_{\Sigma_{\wp}},\ul{\cL}_{S_n}):=\Sigma\big(\alpha,\ul{h}_{S_n},\ul{\cL}_{S_n}\big)\otimes_E \alg(\ul{h}_{S_c})$ which is thus an extension of $d_n$-copies of $F(\alpha,\ul{h}_{\Sigma_{\wp}})$ by $\Sigma(\alpha,\ul{h}_{S_n})$, and carries the information of $\{\alpha,\ul{h}_{\Sigma_{\wp}},\ul{\cL}_{S_n}\}$. For $\sigma\in \Sigma_{\wp}$, put $I_{\sigma}^c(\alpha,\ul{h}_{\Sigma_{\wp}}):=I_{\sigma}^c(\alpha,\ul{h}_{\sigma})\otimes_E \alg(\ul{h}_{\Sigma_{\wp}\setminus \{\sigma\}})$.

\section{Local-global compatibility}
We prove some local-global compatibility results for completed cohomology of quaternion Shimura curves in semi-stable non-crystalline case. %We follow the strategy of \emph{loc. cit.} An important point is that the Galois representations associated to locally $\sigma$-analytic vectors are $\sigma'$-de Rham for $\sigma'\neq \sigma$ (cf. Prop.\ref{prop: lpl-ocj} and Prop.\ref{prop: lpl-rja}), which allows Thm.\ref{thm: lpl-fee} to apply.

%There is also a technical problem when applying global triangulation theory to critical points (see the arguments in the proof of Lem.\ref{lem: lpl-cri} at the end of this section).

%After introducing notations in \S \ref{sec: lpl-4.1}, we recall the construction of eigenvarieties from completed cohomology of quaternion Shimura curves and survey some properties in \S \ref{sec: lpl-4.2}. In \S \ref{sec: lpl-4.3}, we prove results of local-global compatibility, and in particular the equality of Fontaine-Mazur $\cL$-invariants and Breuil's $\cL$-invariants.

% prove a result of local-global compatibility in the completed cohomology of quaternion Shimura curves in critical case.

\subsection{Setup and notations}\label{sec: lpl-4.1}Let $F$ be a totally real field of degree $d$ over $\Q$, denote by $\Sigma_{\infty}$ the set of real embeddings of $F$.  For a finite place $\fl$ of $F$, we denote by $F_{\fl}$ the completion of $F$ at $\fl$, $\co_{\fl}$ the ring of integers of $F_{\fl}$ with $\varpi_{\fl}$ a uniformiser of $\co_{\fl}$. Denote by $\bA$ the ring of adeles of $\Q$ and $\bA_F$ the ring of adeles of $F$. For a set $S$ of places of $\Q$ (resp. of $F$), we denote by $\bA^S$ (resp. by $\bA_F^S$) the ring of adeles of $\Q$ (resp. of $F$) outside $S$, $S_F$ the set of places of $F$ above that in $S$, and $\bA_F^S:=\bA_F^{S_F}$.

Let $p$ be a prime number, suppose there exists only one prime $\wp$ of $F$ lying above $p$. Denote by $\Sigma_{\wp}$ the set of $\Q_p$-embeddings of $F_{\wp}$ in $\overline{\Q_p}$; let $\varpi$ be a uniformizer of $\co_{\wp}$, $F_{\wp,0}$ the maximal unramified extension of $\Q_p$ in $F_{\wp}$, $d_0:=[F_{\wp,0}:\Q_p]$, $e:=[F_{\wp}:F_{\wp,0}]$, $q:=p^{d_0}$ and $\us_{\wp}$ a $p$-adic valuation on $\overline{\Q_p}$ normalized by $\us_{\wp}(\varpi)=1$. Let $E$ be a finite extension of $\Q_p$ big enough such that $E$ contains all the $\Q_p$-embeddings of $F$ in $\overline{\Q_p}$, $\co_E$ the ring of integers of $E$ and $\varpi_E$ a uniformizer of $\co_E$.

Let $B$ be a quaternion algebra of center $F$ with $S(B)$ the set (of even cardinality) of places of $F$ where $B$ is ramified, suppose $|S(B)\cap \Sigma_{\infty}|=d-1$ and $S(B)\cap \Sigma_{\wp}=\emptyset$, i.e. there exists $\tau_{\infty}\in \Sigma_{\infty}$ such that $B\otimes_{F,\tau_{\infty}} \R\cong M_2(\R)$, $B\otimes_{F,\sigma} \R \cong \bH$ for all $\sigma\in \Sigma_{\infty}\setminus \{\tau_{\infty}\}$, where $\bH$ denotes the Hamilton algebra, and $B\otimes_{\Q} \Q_p\cong M_2(F_{\wp})$. We associate to $B$ a reductive algebraic group $G$ over $\Q$ with $G(R):=(B\otimes_{\Q} R)^{\times}$ for any $\Q$-algebra $R$. Set $\bS:=\Res_{\bC/\R}\bG_m$, and denote by $h$ the morphism
\begin{equation*}
  h: \bS(\R)\cong \bC^{\times} \lra G(\R)\cong \GL_2(\R)\times (\bH^*)^{d-1}, \ a+bi\mapsto \bigg(\begin{pmatrix}a&b\\ -b&a\end{pmatrix}, 1, \cdots, 1\bigg).
\end{equation*}
The space of $G(\R)$-conjugacy classes of $h$ has a structure of complex manifold, and is isomorphic to $\fh^{\pm}:=\bC \setminus \R$ (i.e. $2$ copies of the Poincar\'e's upper half plane). We get a projective system of Riemann surfaces indexed by open compact subgroups of $G(\bA^{\infty})$:
\begin{equation*}
  M_K(\bC):=G(\Q)\setminus \big(\fh^{\pm} \times (G(\bA^{\infty})/K)\big)
\end{equation*}
where $G(\Q)$ acts on $\fh^{\pm}$ via $G(\Q)\hookrightarrow G(\R)$ and the transition map is given by
\begin{equation}\label{equ: clin-sab}
G(\Q)\setminus \big(\fh^{\pm} \times (G(\bA^{\infty})/K_1)\big) \lra G(\Q)\setminus \big(\fh^{\pm} \times (G(\bA^{\infty})/K_2)\big), \ (x,g)\mapsto (x,g),
\end{equation} for $K_1\subseteq K_2$. It's known that $M_K(\bC)$ has a canonical proper smooth model over $F$ (via the embedding $\tau_{\infty}$), denoted by $M_K$, and these $\{M_K\}_{K}$ form a projective system of proper smooth algebraic curves over $F$.  %For $g\in G(\bA^{\infty})$, the map
Note that one has a natural isomorphism $G(\Q_p) \xlongrightarrow{\sim} \GL_2(F_{\wp})$.
 %Put $G(\bA^{\infty})^{\wp}:= \Big(G(\bA^{\infty,p})\big(\prod_{\substack{\wp_i|p\\ \wp_i\neq \wp}}\GL_2(F_{\wp_i})\big)\Big)\subset G(\bA^{\infty})$ (where the inclusion is induced by the above isomorphism).
  For an open compact subgroup $K$ of $G(\bA^{\infty})$, let $K_p:=K\cap G(\Q_p)$, and $K^{p}:=K\cap G(\bA^{\infty,p})$, so one has $K=K^{p}K_{p}$.
%  For $\wp Set $K_{\wp,0}:=\GL_2(\co_{\wp})$, and
%\begin{equation*}K_{\wp,n}:=\bigg\{g\in \GL_2(\co_{\wp})\ |\ g\equiv \begin{pmatrix}1 &0 \\ 0 & 1\end{pmatrix} \pmod{\varpi^n}\bigg\}.\end{equation*}

%Denote by $Z$ the algebraic group $\Res_{F/\Q} \bG_m$, which is hence a subgroup of $G$ (since $B$ is of center $F$).
%\begin{lemma}[$\text{\cite[Lem.1.4.1.1]{Ca}}$]\label{lem: clin-pan}Let $K^p$ be a open compact subgroup of $G(\bA^{\infty})^{\wp}$, when $K^{\wp}$ is small enough, the covering
% \begin{equation*}
% M_{K^{\wp}K_{\wp,n}} \lra M_{K^{\wp}K_{\wp,0}}
% \end{equation*}is Galois of group $K_{\wp,0}/\big(K_{\wp,n}(Z(\Q)\cap K^{\wp}K_{\wp,0})_{\wp}\big)$ for any $n\in %\Z_{>0}$.
%\end{lemma}
Let $K_{\wp,0}:=\GL_2(\co_{\wp})$, and in the following, we fix an open compact subgroup $K^{p}$ of $G(\bA^{\infty,p})$ of the form $\prod_{v\nmid p} K_v$ small enough such that $K^pK_{\wp,0}$ is neat (e.g. see \cite[Def.4.11]{New}). Denote by $S(K^{p})$ the set of finite places $\fl$ of $F$ such that $p \nmid \fl$, that $B$ is split at $\fl$, i.e. $B\otimes_F F_{\fl}\xlongrightarrow{\sim}  \mathrm{M}_2(F_{\fl})$, and that $K^{p}\cap \GL_2(F_{\fl})\cong \GL_2(\co_{\fl})$. Denote by $\cH^p$ the commutative $\co_E$-algebra generated by the double coset operators $[\GL_2(\co_{\fl}) g_{\fl} \GL_2(\co_{\fl})]$ for all $g_{\fl}\in \GL_2(F_{\fl})$ with $\dett(g_{\fl})\in \co_{\fl}$ and for all $\fl \in S(K^p)$. Set
\begin{eqnarray*}T_{\fl}&:=&\bigg[\GL_2(\co_{\fl})\begin{pmatrix} \varpi_{\fl} & 0 \\ 0 & 1\end{pmatrix}\GL_2(\co_{\fl})\bigg], \\
S_{\fl} &:=& \bigg[\GL_2(\co_{\fl})\begin{pmatrix} \varpi_{\fl} & 0 \\ 0 & \varpi_{\fl}\end{pmatrix}\GL_2(\co_{\fl})\bigg],
\end{eqnarray*}
then $\cH^p$ is the polynomial algebra over $\co_E$ generated by $\{T_{\fl}, S_{\fl}\}_{\fl \in S(K^p)}$.
%We have then
%\begin{lemma}
 % Let $K_1$ be a open compact subgroup of $K_{\wp}^0$, $K_2$ be a open compact normal subgroup of $K_1$, then the covering
  %\begin{equation}\label{equ: clin-a2k}
  %  M_{K^{\wp}K_2} \lra M_{K^{\wp}K_1}
  %\end{equation}
  %is Galois of group $K_1/\big(K_2(Z(\Q) \cap K^{\wp}K_1)_{\wp}\big)$.
%\end{lemma}
%\begin{proof}
%This lemma follows easily from the proof of \cite[Lem.1.4.1.1]{Ca}.
%\texttt{With the notations in the proof of \ref{lem: clin-pan}, we have a similar $\gamma$, which can be proved lying inside $Z(\Q)$ and then $\gamma \in Z(\Q) \cap K^pK_1$, from which the lemma follows.}
%\end{proof}
%\textbf{\emph{It will be useful to fix a $K^p$ and describe the group $(Z(\Q)\cap K^pK_{\wp,0})_{\wp}$ explicitly (in terms of Leopoldt's conjecture)! On the other hand, by shrinking $K_1$, we can assume $\overline{(Z(\Q)\cap K^pK_1)_{\wp}}$ is contained in the kernel of $N: F_{\wp}^{\times} \ra \Q_p^{\times}$.}}

%Denote by $Z_0$ the kernel of the norm map: $F_{\wp}^{\times} \ra \Q_p^{\times}$, and we view $Z_0$ as a subgroup of $Z(\Q)$.
%Denote by $Z$ the algebraic group $\Res_{F/\Q} \bG_m$, which is hence a subgroup of $G$ (since $B$ is of center $F$).
Denote by $Z_0$ the kernel of the norm map $\sN: \Res_{F/\Q} \bG_m \ra \bG_m$ which is a subgroup of $Z=\Res_{F/\Q} \bG_m$. We set $G^c:=G/Z_0$.

For a  Banach representation $V$ of $\GL_2(F_{\wp})$ over $E$ (cf. \cite[\S 2]{ST}), denote by $V_{\Q_p-\an}$ the $E$-vector subspace generated by the locally $\Q_p$-analytic vectors of $V$, which is stable by $\GL_2(F_{\wp})$ and hence is a locally $\Q_p$-analytic representation of $\GL_2(F_{\wp})$. If $V$ is moreover admissible, by \cite[Thm.7.1]{ST03}, $V_{\Q_p-\an}$ is an admissible locally $\Q_p$-analytic representation of $\GL_2(F_{\wp})$ and  dense in $V$. For $J\subseteq \Sigma_{\wp}$, denote by $V_{J-\an}$ the subrepresentation generated by locally $J$-analytic vectors of $V_{\Q_p-\an}$ (cf. \cite[\S 2]{Sch10}), put $V_{\infty}:=V_{\emptyset-\an}$.

Let $A$ be a local artinian $E$-algebra, for a locally $\Q_p$-analytic character  $\chi=\chi_1\otimes \chi_2$ of $T(L)$ over $A$, let $\wt(\chi):=(\wt(\chi)_{1,\sigma},\wt(\chi)_{2,\sigma})_{\sigma\in \Sigma_{\wp}}:=(\wt(\chi_1)_{\sigma},\wt(\chi_2)_{\sigma})_{\sigma\in \Sigma_{\wp}}\in A^{2|d|}$ be the weight of $\chi$. For an integer weight $\lambda\in \Z^{2|d|}$, denote by $\delta_{\lambda}$ the algebraic character of $T(L)$ over $E$ with weight $\lambda$.

Let $V$ be an $E$-vector space equipped with an $E$-linear action of $A$ (with $A$ a set of operators), $\chi$ a system of eigenvalues of $A$, denote by $V^{A=\chi}$ the $\chi$-eigenspace, $V[A=\chi]$ the generalized $\chi$-eigenspace, $V^A$ the vector space of $A$-fixed vectors.
\subsection{Completed cohomology and eigenvarieties}\label{sec: lpl-4.2}We recall the construction of eigenvarieties from completed cohomology of quaternion Shimura curves and survey some properties. %We refer to \cite[\S 6.2]{Ding} for similar (and more detailed) discussion in the unitary Shimura curves case.
\subsubsection{Completed cohomology of quaternion Shimura curves}
Let $W$ be a finite dimensional algebraic representation of $G^c$ over $E$, as in \cite[\S 2.1]{Ca2}, one can associate to $W$ a local system $\cV_W$ of $E$-vector spaces over $M_K$. Let $W_0$ be an $\co_E$-lattice of $W$, denote by $\cS_{W_0}$ the set (ordered by inclusions) of open compact subgroups of $G(\Q_p)\cong \GL_2(F_{\wp})$ which stabilize $W_0$. For any $K_{p}\in S_{W_0}$, one can associate to $W_0$ \big(resp. to $W_0/\varpi_E^s$ for $s\in \Z_{\geq 1}$\big) a local system $\cV_{W_0}$ \big(resp. $\cV_{W_0/\varpi_E^s}$\big) of $\co_E$-modules \big(resp. of $\co_E/\varpi_E^s$-modules\big) over $M_{K_pK^p}$.  Following Emerton (\cite{Em1}), we put
\begin{eqnarray*}
   H^i_{\et}(K^p,W_0)&:=&\varinjlim_{K_p\in \cS_{W_0}} H^i_{\et} \big( M_{K_pK^p, \overline{\Q}}, \cV_{W_0}\big) \\
&\cong& \varinjlim_{K_p\in \cS_{W_0}} \varprojlim_{s} H^i_{\et} ( M_{K_pK^p, \overline{\Q}}, \cV_{W_0/\varpi_E^s});\\
\widetilde{H}^i_{\et}(K^p,W_0)&:=& \varprojlim_{s}\varinjlim_{K_p\in \cS_{W_0}}H^i_{\et} ( M_{K_pK^p,\overline{\Q}}, \cV_{W_0/\varpi_E^s});\\
H^i_{\et}(K^p,W_0)_E &:=& H^i_{\et}(K^p,W_0) \otimes_{\co_E}E; \\
\widetilde{H}^i_{\et}(K^p,W_0)_E &:=&\widetilde{H}^i_{\et}(K^p,W_0) \otimes_{\co_E}E.
\end{eqnarray*}
%\Big(resp. by replacing $K^p$, $K_{\wp}$, $S_{W_0}$ by $K^p$, $K^p$, $R_{W_0}$ one can define $H^i_{\et}(K^p,W_0)$, $\widetilde{H}^i_{\et}(K^p,W_0)$, $H^i_{\et}(K^p,W_0)_E$, $\widetilde{H}^i_{\et}(K^p,W_0)_E$\Big).
All these groups ($\co_E$-modules or $E$-vector spaces) are equipped with a natural topology induced from the discrete topology on the finite group $H^i_{\et} \big( M_{K_pK^p, \overline{\Q}}, \cV_{W_0/\varpi_E^s}\big)$, and equipped with a natural continuous action of $\cH^p\times \Gal_{F}$ and of $K_p\in \cS_{W_0}$. Moreover, for any $\fl\in S(K^p)$, the action of $\Gal_{F_{\fl}}$ (induced by that of $\Gal_{F}$)  is unramified and satisfies the Eichler-Shimura relation:
\begin{equation*}\label{equ: clin-llf-}
  \Frob_{\fl}^{-2} -T_{\fl}\Frob_{\fl}^{-1}+\ell^{f_{\fl}}S_{\fl}=0
\end{equation*}
where $\Frob_{\fl}$ denotes the arithmetic Frobenius, $\ell$ the prime number lying below $\fl$, $f_{\fl}$ the degree of the maximal unramified extension in $F_{\fl}$ over $\Q_{\ell}$ (thus $\ell^{f_{\fl}}= |\co_{\fl}/\varpi_{\fl}|$). Note that $\widetilde{H}^i_{\et}(K^p,W_0)_E$ is an $E$-Banach space with the norm defined by the $\co_E$-lattice $\widetilde{H}^i_{\et}(K^p,W_0)$.

Consider the ordered set (by inclusion) $\{W_0\}$ of $\co_E$-lattices of $W$, following \cite[Def.2.2.9]{Em1}, we put
\begin{eqnarray*}
H^i_{\et}(K^p,W)&:=& \varinjlim_{W_0} H^i_{\et}(K^p,W_0)_E, \\
\widetilde{H}^i_{\et}(K^p,W)&:=& \varinjlim_{W_0} \widetilde{H}^i_{\et}(K^p,W_0)_E,
\end{eqnarray*}%\widehat{H}^i_{\et}(K^p,W)&:=& \varinjlim_{W_0} \widehat{H}^i_{\et}(K^p,W_0)_E ,\\
where all the transition maps are topological isomorphisms (cf. \cite[Lem.2.2.8]{Em1}). These $E$-vector spaces are moreover equipped with a natural continuous action of $\GL_2(F_{\wp})$ (cf. \cite[Lem.2.2.10]{Em1}).
\begin{theorem}[$\text{cf. \cite[Thm.2.2.11 (i), Thm.2.2.17]{Em1}}$]\label{thm: lpl-enc}(1) The $E$-Banach space $\widetilde{H}^i_{\et}(K^p,W)$ is an admissible Banach representation of $\GL_2(F_{\wp})$. If $W$ is the trivial representation, the $\GL_2(F_{\wp})$-representation $\widetilde{H}^i_{\et}(K^p,W)$ is unitary.

(2) One has a natural isomorphism of Banach representations of $\GL_2(F_{\wp})$ invariant under the action of $\cH^p \times \Gal(\overline{F}/F)$:
\begin{equation*}\label{equ: clin-wpkw}
  \widetilde{H}^i_{\et}(K^p,W) \xlongrightarrow{\sim} \widetilde{H}^i_{\et}(K^p,E)\otimes_E W.
\end{equation*}

(3) One has a natural $\GL_2(F_{\wp})\times \cH^p \times \Gal_F$-invariant map
\begin{equation*}\label{equ: clin-ehiw}
  H^i_{\et}(K^p,W) \lra \widetilde{H}^i_{\et}(K^p,W).
\end{equation*}
\end{theorem}
Let $\rho$ be a $2$-dimensional continuous representation of $\Gal_{F}$ over $E$ such that $\rho$ is unramified at all $\fl\in S(K^p)$ and that the reduction $\overline{\rho}$ over $k_E$ (up to semi-simplification a priori) is absolutely irreducible. To $\overline{\rho}$, one can associate a maximal ideal $\fm(\overline{\rho})$ of $\cH^p$ as the kernel of the following morphism
\begin{equation*}
  \cH^p \lra k_E:=\co_E/\varpi_E, \ T_{\fl} \mapsto \tr(\Frob_{\fl}^{-1}),\ S_{\fl} \mapsto \ell^{-f_{\fl}}\dett(\Frob_{\fl}^{-1}), \  \forall\ \fl\in S(K^p).
\end{equation*}For an $\cH^p$-module $M$, denote by $M_{\overline{\rho}}$ the localisation of $M$ at $\fm(\rho)$.
%
%\begin{notation}
%  For an $\cH^p$-module $M$, denote by $M_{\overline{\rho}^{\sss}}$ the localisation of $M$ at $\fm(\overline{\rho}^{\sss})$.
%\end{notation}
%Suppose in the following that  $\rho$ is absolutely irreducible modulo $\varpi_E$ and put $\overline{\rho}:=\overline{\rho}^{\sss}$.

Put $Z_1:=1+2\varpi \co_{\wp} \subseteq Z(\GL_2(F_{\wp}))$ the center of $\GL_2(F_{\wp})$.
%Thus $Z_1=Z(K_{\wp,1})$ when $p\neq 2$.
Put\begin{equation*}\label{equ: clin-2ls}U_1:=\{g_{\wp}\in 1+2\varpi M_2(\co_{\wp})\ |\ \dett(g_{\wp})=1\},\end{equation*} and $H_{\wp}:=Z_1U_1$ which is a pro-$p$ open compact subgroup of $\GL_2(\co_{\wp})$.
\begin{proposition}\label{prop: lpl-rnf}Let $W$ be an irreducible algebraic representation of $G^c$, and suppose $\widetilde{H}^1_{\et}(K^p,W)_{\overline{\rho}}\neq 0$.

 (1) The natural morphism
  \begin{equation*}
  H^1_{\et}(K^p,W)_{\overline{\rho}} \lra \widetilde{H}^1_{\et}(K^p,W)_{\overline{\rho},\infty}
  \end{equation*}
  is an isomorphism, where $\infty$ denotes the smooth vectors for the action of $\GL_2(F_{\wp})$.

  (2) Let $\psi$ be a continuous character of $Z_1$ such that $\psi|_{\overline{(Z(\Q)\cap K^{p}H_{\wp})_{p}}}=1$, then one has an isomorphism of $H_{\wp}$-representations
  \begin{equation*}
   \widetilde{H}_{\et}^1(K^{p},W)_{\overline{\rho}}^{Z_1=\psi}\xlongrightarrow{\sim} \cC\big(U_1, E\big)^{\oplus r}
  \end{equation*}
  where $Z_1$ acts on $\cC\big(U_1, E\big)^{\oplus r}$ by the character $\psi$, and $U_1$ by the right regular action.
\end{proposition}
\begin{proof}
  Part (1) follows from \cite[Prop.5.2]{New}. Part (2) follows by the same arguments as in \cite[\S 5]{New} (see also \cite[Cor.2.5]{Ding3}).
\end{proof}
\begin{remark}
  One can check $\overline{(Z(\Q)\cap K^pH_{\wp})_p}\subseteq Z_0(\Q_p)$, in particular, any continuous character of $Z_1$ factoring through $Z_1/(Z_1\cap Z_0(\Q_p))$ satisfies the assumption in Prop.\ref{prop: lpl-rnf} (2).
\end{remark}

\subsubsection{Eigenvarieties}\label{sec: lpl-4.2.2}
Let $J\subseteq \Sigma_{\wp}$, $k_{\sigma}\in 2\Z_{\geq 1}$ for all $\sigma\in J$ and $w\in 2\Z$, we set
\begin{equation*}
  W(\ul{k}_{J},w):=\otimes_{\sigma\in J} \big(\Sym^{k_{\sigma}-2} E^2 \otimes_E \dett^{\frac{w-k_{\sigma}+2}{2}}\big)^{\sigma}\otimes_E \big( \otimes_{\sigma\in \Sigma_{\wp}\setminus J} (\dett^{\frac{w}{2}})^{\sigma}\big) ,
\end{equation*}
which is an irreducible algebraic representation of  $G$  over $E$ with the action of $\GL_2(F_{\wp})$ on $(*)^{\sigma}$ induced from the standard action of $\GL_2(E)$ via $\sigma: \GL_2(F_{\wp})\hookrightarrow \GL_2(E)$. Note the central character of $W(\ul{k}_J,w)$ is given by $\sN^{w}$ (where $\sN$ denotes the norm map), thus $W(\ul{k}_J,w)$ can be viewed as an algebraic representation of $G^c$. One has $W(\ul{k}_J,w')=W(\ul{k}_J,w)\otimes_E W(\ul{2}_J,w'-w)$ (where $w'\in 2\Z$), and $W(\ul{k}_{J},w)=W(\ul{k}_{J'},w)\otimes_E W(\ul{k}_{J\setminus J'},0)$ for $J'\subseteq J$.

Let $\rho$ be a $2$-dimensional continuous representation of $\Gal_{F}$ over $E$, absolutely irreducible modulo $\varpi_E$, such that $\rho$ is unramified at all $\fl\in S(K^p)$ and that $\widetilde{H}^1_{\et}(K^p,E)_{\overline{\rho}}\neq 0$.
By Thm.\ref{thm: lpl-enc} (2),  $\widetilde{H}^1_{\et}(K^p,W(\ul{k}_J,w))\cong \widetilde{H}^1_{\et}(K^p,E)\otimes_E W(\ul{k}_J,w)$, thus $\widetilde{H}^1_{\et}(K^p,W(\ul{k}_J,w))_{\overline{\rho}}\neq 0$ for $J\subseteq \Sigma_{\wp}$. Put $\Pi:=\widetilde{H}^1_{\et}(K^p,E)_{\overline{\rho},\Q_p-\an}$, and for $J\subseteq \Sigma_{\wp}$, put
\begin{equation*}
  \Pi(\ul{k}_{J},w):=\widetilde{H}^1_{\et}(K^p,W(\ul{k}_{J},w))_{\overline{\rho},{\Sigma_{\wp}\setminus J}-\an}\otimes_E W(\ul{k}_{J},w)^{\vee},
\end{equation*}
which is in fact a closed subrepresentation of $\Pi$. Indeed, we have
\begin{multline*}
\widetilde{H}^1_{\et}(K^p,W(\ul{k}_{J},w))_{\overline{\rho},{\Sigma_{\wp}\setminus J}-\an}\otimes_E W(\ul{k}_{J},w)^{\vee} \xlongrightarrow{\sim} \big(\Pi \otimes_E W(\ul{k}_{J},w)\big)_{{\Sigma_{\wp}\setminus J}-\an}\otimes_E W(\ul{k}_{J},w)^{\vee} \\
\xlongrightarrow{\sim} \Big(\Pi \otimes_E\otimes_{\sigma\in J}\big(\Sym^{k_{\sigma}-2} E^2 \otimes_E \dett^{\frac{w-k_{\sigma}+2}{2}}\big)^{\sigma}\Big)_{{\Sigma_{\wp}\setminus J}-\an} \otimes_E \big(\otimes_{\sigma\in {\Sigma_{\wp}\setminus J}} (\dett^{\frac{w}{2}})^{\sigma}\big)\otimes_E W(\ul{k}_{J},w)^{\vee}
\\ \xlongrightarrow{\sim} \Big(\Pi\otimes_E\otimes_{\sigma\in J}\big(\Sym^{k_{\sigma}-2} E^2 \otimes_E \dett^{\frac{w-k_{\sigma}+2}{2}}\big)^{\sigma}\Big)_{{\Sigma_{\wp}\setminus J}-\an} \otimes_E \big(\otimes_{\sigma\in J}\Sym^{k_{\sigma}-2} E^2 \otimes_E \dett^{\frac{w-k_{\sigma}+2}{2}}\big)^{\vee}
\hooklongrightarrow \Pi,
\end{multline*}
where the first isomorphism is from Thm.\ref{thm: lpl-enc}(2), and the last injection follows from \cite[Prop.5.1.3]{Ding}. Similarly, for $J'\supseteq J$, we have a natural closed embedding invariant under the action of $\GL_2(F_{\wp})\times \cH^p$:
\begin{equation}\label{equ: lpl-injJ}
 \Pi(\ul{k}_{J'},w) \hooklongrightarrow \Pi(\ul{k}_{J},w).
\end{equation}
Note $\Pi(\ul{k}_{\emptyset},w)\cong \Pi$, and by Prop.\ref{prop: lpl-rnf} (1), $\Pi(\ul{k}_{\Sigma_{\wp}},w)\cong H^1_{\et}(K^p,W(\ul{k}_{\Sigma_p},w))_{\overline{\rho}}\otimes_E W(\ul{k}_{\Sigma_p},w)^{\vee}$. %By \cite[Cor.5.1.6]{Ding}, W
We have the following easy lemma.
\begin{lemma}\label{lem: lpl-lav}
  Keep the above notation, let $V$ be a locally $\Sigma_{\wp}\setminus J$-analytic representation of $\GL_2(F_{\wp})$, then
  \begin{multline}\label{equ: lpl-lav}
    \Hom_{\GL_2(F_{\wp})}\Big(V,\widetilde{H}^1_{\et}\big(K^p,W(\ul{k}_J,w)\big)_{\overline{\rho},\Sigma_{\wp}\setminus J-\an}\Big)\xlongrightarrow{\sim} \Hom_{\GL_2(F_{\wp})}\big(V\otimes_E W(\ul{k}_J,w)^{\vee},\Pi(\ul{k}_J,w)\big)\\ \xlongrightarrow{\sim}\Hom_{\GL_2(F_{\wp})}\big(V\otimes_E W(\ul{k}_J,w)^{\vee},\Pi\big),
  \end{multline}
  where the first map is given by $f\mapsto f\otimes \id$, and the second is induced by the injection $\Pi(\ul{k}_J,w)\hookrightarrow \Pi$.
\end{lemma}
\begin{proof}Given a morphism $g: V\otimes_E W(\ul{k}_J,w)^{\vee} \ra \Pi$, consider the composition $V\ra V\otimes_E W(\ul{k}_J,w)^{\vee}\otimes_E W(\ul{k}_J,w) \xrightarrow{g\otimes \id} \widetilde{H}^1_{\et}(K^p,W(\ul{k}_J,w))_{\overline{\rho},\Q_p-\an}$ whose image is contained in $\widetilde{H}^1_{\et}(K^p,W(\ul{k}_J,w))_{\overline{\rho},\Sigma_{\wp}\setminus J-\an}$ since $V$ is locally $\Sigma_{\wp}\setminus J$-anlaytic, and it's straightforward to check this gives an inverse (up to non-zero scalars) of the composition (\ref{equ: lpl-lav}). The lemma follows.
\end{proof}

Let $J\subsetneq \Sigma_{\wp}$, $\ul{k}_J\in 2\Z_{\geq 1}^{|J|}$ and $w\in 2\Z$, consider
\begin{equation}\label{equ: lpl-wHp}
  \Pi(\ul{k}_J,w)^{Z_1=\cN^{-w}}\cong \widetilde{H}^1(K^p,W(\ul{k}_J,w))_{\overline{\rho},\Sigma_{\wp}\setminus J-\an}^{Z_1}\otimes_E W(\ul{k}_J,w)^{\vee},
\end{equation}
which is an admissible locally $\Q_p$-analytic representation of $\GL_2(F_{\wp})$ equipped with a continuous action of $\cH^p$ commuting with $\GL_2(F_{\wp})$. Applying Jacquet-Emerton functor (for the upper triangular subgroup $B(F_{\wp})$, cf. \cite{Em11}), we get an essentially admissible locally $\Q_p$-analytic representation $J_B\big(\Pi(\ul{k}_J,w)^{Z_1=\cN^{-w}}\big)$ of $T(F_{\wp})$ over $E$, whose strong dual corresponds to a coherent sheaf $\cM_0(\ul{k}_J,w)$ over $\widehat{T}$ \big(which denotes the rigid space over $E$ parameterizing continuous characters of $T(F_{\wp})$\big) such that
\begin{equation*}
  \Gamma\big(\widehat{T},\cM_0(\ul{k}_J,w)\big)\xlongrightarrow{\sim} J_B\big(\Pi(\ul{k}_J,w)^{Z_1=\cN^{-w}}\big)^{\vee}
\end{equation*}
as coadmissible modules over the Fr\'echet-Stein algebra $\co(\widehat{T})$. By funtoriality, $\cM_0(\ul{k}_J,w)$ is equipped with a $\co(\widehat{T})$-linear action of $\cH^p$. Following Emerton \cite[\S 2.3]{Em1}, we can construct an eigenvariety from the triplet $\big\{\cM_0(\ul{k}_J,w),\widehat{T},\cH^p\big\}$:
\begin{theorem}\label{thm: clin-cjw}
  There exists a rigid analytic space $\cE(\ul{k}_{J},w)_{\overline{\rho}}$ over $E$ together with a finite morphism of rigid spaces $i: \cE(\ul{k}_{J},w)_{\overline{\rho}} \ra \widehat{T}$
  and a morphism of $E$-algebras with dense image
  \begin{equation}\label{equ: clin-oct}
    \cH^p\otimes_{\co_E} \co\big(\widehat{T}\big) \lra \co\big(\cE(\ul{k}_{J},w)_{\overline{\rho}}\big)
  \end{equation}
  such that
  \begin{enumerate}
    \item a closed point $z$ of $\cE(\ul{k}_{J},w)_{\overline{\rho}}$ is uniquely determined by its image $\chi$ in $\widehat{T}(\overline{E})$ and the induced morphism $\lambda: \cH^p\lra \overline{E}$, called a system of eigenvalues of $\cH^p$; hence $z$ would be  denoted by $(\chi,\lambda)$);
    \item for a finite extension $L$ of $E$, a closed point $(\chi,\lambda)\in \cE(\ul{k}_{J},w)_{\overline{\rho}}(L)$ if and only if the corresponding eigenspace
    \begin{equation*}
     J_B\big(\Pi(\ul{k}_J,w)^{Z_1=\cN^{-w}}\otimes_E L\big)^{T(F_{\wp})=\chi, \cH^p=\lambda}
    \end{equation*}
    is non-zero;
    \item there exists a coherent sheaf, denoted by $\cM(\ul{k}_{J},w)$, over $\cE(\ul{k}_{J},w)_{\overline{\rho}}$, such that $i_* \cM(\ul{k}_{J},w) \cong \cM_0(\ul{k}_{J},w)$ and that for an $L$-point $z=(\chi, \lambda)$, the special fiber $\cM(\ul{k}_{J},w)\big|_z$ is naturally dual to the (finite dimensional) $L$-vector space
        \begin{equation*}
          J_B\big(\Pi(\ul{k}_J,w)^{Z_1=\cN^{-w}}\otimes_E L\big)^{T(F_{\wp})=\chi, \cH^p=\lambda}.
        \end{equation*}
  \end{enumerate}
\end{theorem}

By (\ref{equ: lpl-wHp}), one has an isomorphism
\begin{equation}\label{equ: lpl-NwwJ}
  J_B\big(\Pi(\ul{k}_J,w)^{Z_1=\cN^{-w}}\big)\cong J_B\big(\widetilde{H}^1_{\et}(K^p,W(\ul{k}_J,w))_{\Sigma_{\wp}\setminus J-\an}^{Z_1}\big)\otimes_E \chi(\ul{k}_J,w),
\end{equation}
where $\chi(\ul{k}_J,w):=\big(\prod_{\sigma\in J}(\sigma^{\frac{k_{\sigma}-2}{2}}\otimes \sigma^{\frac{2-k_{\sigma}}{2}})\big)\big(\prod_{\sigma\in \Sigma_{\wp}}(\sigma^{-w/2}\otimes \sigma^{-w/2})\big)$ is a character of $T(F_{\wp})$ over $E$. Thus, by Thm.\ref{thm: clin-cjw} (2), if $(\chi,\lambda)\in \cE(\ul{k}_J,w)_{\overline{\rho}}$, then $\wt(\chi)_{1,\sigma}+\wt(\chi)_{2,\sigma}=-w$ for all $\sigma\in \Sigma_{\wp}$, and $\wt(\chi)_{1,\sigma}-\wt(\chi)_{2,\sigma}=k_{\sigma}-2$ for all $\sigma\in J$.

Denote by $\widehat{T}_{\Sigma_{\wp}\setminus J}$ the rigid space over $E$ parameterizing the locally $\Sigma_{\wp}\setminus J$-analytic characters of $T(F_{\wp})$, and denote by $\widehat{T}(\ul{k}_J,w)$ the image of the following closed embedding
\begin{equation*}
  \widehat{T}_{\Sigma_{\wp}\setminus J} \hooklongrightarrow \widehat{T}, \ \chi \mapsto \chi\chi(\ul{k}_J,w),
\end{equation*}
which parameterizes characters of $T(F_{\wp})$ with fixed weights $(\frac{k_{\sigma}-w-2}{2},\frac{2-k_{\sigma}-w}{2})$ for $\sigma\in J$.
By the isomorphism (\ref{equ: lpl-NwwJ}), it's easy to see the action of $\co(\widehat{T})$ on $\cM_0(\ul{k}_J,w)$ factors through $\co\big(\widehat{T}(\ul{k}_J,w)\big)$, consequently, the morphism $\cE(\ul{k}_J,w)_{\overline{\rho}}\ra \widehat{T}$ factors through $\widehat{T}(\ul{k}_J,w)$. Denote by $\widehat{T}(\ul{k}_J,w)_0$ the closed subspace of $\widehat{T}(\ul{k}_J,w)$ consisting of the points $\chi$ with $\chi|_{Z_1}=\cN^{-w}$, thus the morphism $\cE(\ul{k}_J,w)_{\overline{\rho}} \ra \widehat{T}(\ul{k}_J,w)$ factors through $\widehat{T}(\ul{k}_J,w)_0$. Denote by $Z_1':=\bigg\{\begin{pmatrix} a& 0 \\ 0 &a^{-1}\end{pmatrix}\ \bigg|\ a\in 1+2\varpi \co_{\wp}\bigg\}$, and $\cW_1$ the rigid space over $E$ parameterizing continuous characters of $1+2\varpi \co_{\wp}$ (thus of $Z_1'$), and $\cW_1(\ul{k}_J)$ the closed subspace of $\cW_1$ of characters $\chi$ with $\wt(\chi)_{\sigma}=k_{\sigma}-2$ for all $\sigma\in J$. One has thus a natural projection
\begin{equation*}
  j: \widehat{T}(\ul{k}_J,w)_0 \twoheadlongrightarrow \cW_1(\ul{k}_J) \times \bG_m,\ \chi\mapsto (\chi|_{Z_1'},\chi(z_{\wp})),
\end{equation*}
where $z_{\wp}:=\begin{pmatrix} \varpi & 0 \\ 0 & 1\end{pmatrix}$. By Prop.\ref{prop: lpl-rnf}(2) and (the proof of ) \cite[Prop.4.2.36]{Em11}, $J_B\big(\Pi(\ul{k}_J,w)^{Z_1=\cN^{-w}}\big)^{\vee}$ is in fact a coadmissible module over $\co\big(\cW_1(\ul{k}_J)\times \bG_m)$, in other words, $j_* \cM_0(\ul{k}_J,w)$ is a coherent sheaf over $\cW_1(\ul{k}_J)\times \bG_m$.
\begin{proposition}\label{prop: lpl-evb}
   (1) The support $\cZ(\ul{k}_J,w)$ of $j_*\cM_0(\ul{k}_J,w)$ on $\cW_1(\ul{k}_J)\times \bG_m$ is a Fredholm hypersurface in $\cW_1(\ul{k}_J)\times \bG_m$, and there exists an admissible covering $\{U_i\}$ of $\cZ(\ul{k}_J,w)$ by affinoids $U_i$ such that the natural morphism $U_i\ra \cW_1$ induces a finite surjective map from $U_i$ to an affinoid open $W_i$ of $\cW_1(\ul{k}_J)$, and that $U_i$ is a connected component of the preimage of $W_i$. Moreover, $\Gamma(U_i,j_*\cM_0(\ul{k}_J,w))$ is a finite projective $\co(W_i)$-module.

   (2) Denote by $g$ the natural morphism $\cE(\ul{k}_J,w)_{\overline{\rho}}\ra \cZ(\ul{k}_J,w)$, and let $\{U_i\}$ as in (1), then $g^{-1}(U_i)$ is an affinoid open in $\cE(\ul{k}_J,w)_{\overline{\rho}}$; $\Gamma(g^{-1}(U_i),\cM(\ul{k}_J,w))\cong \Gamma(j^{-1}(U_i),\cM_0(\ul{k}_J,w))\cong M_i$; let $B_i$ be the affinoid algebra with $\Spm B_i\cong g^{-1}(U_i)$, then $B_i$ is the $\co(W_i)$-subalgebra of $\End_{\co(W_i)}(M_i)$ generated by the $\co(W_i)$-linear operators in $T(F_{\wp})\times \cH^p$.
\end{proposition}
\begin{proof}
  By \emph{loc. cit.}, the discussion in \cite[App.]{Ding} and the arguments before \cite[Prop.6.2.31]{Ding}, we can reconstruct $\cE(\ul{k}_J,w)_{\overline{\rho}}$ by the method of Coleman-Mazur-Buzzard, and then the proposition follows from \cite[\S 4, \S 5]{Bu}
\end{proof}
Denote by $\kappa$ the composition
\begin{equation*}
  \kappa: \cE(\ul{k}_J,w)_{\overline{\rho}}\lra \cZ(\ul{k}_J,w)\lra \cW_1(\ul{k}_J),
\end{equation*}
which also equals the composition $\cE(\ul{k}_J,w)_{\overline{\rho}}\ra \widehat{T}(\ul{k}_J,w)_0\ra \cW_1(\ul{k}_J)$. %We have
\subsubsection{Classicality}
Let $z=(\chi, \lambda)$ be a closed point of $\cE(\ul{k}_{J},w)_{\overline{\rho}}$, $z$ is called \emph{classical} if there exist $k_{\sigma}\in 2\Z_{\geq 1}$ for all $\sigma\in \Sigma_{\wp}\setminus J$ such that \begin{equation*}
  \big(J_B\big(\Pi(\ul{k}_{\Sigma_{\wp}},w)\big)\otimes_E \overline{E}\big)^{\cH^p=\lambda,T(F_{\wp})=\chi}\neq 0.
\end{equation*}
Note $\Pi(\ul{k}_{\Sigma_{\wp}},w)$ is a locally algebraic subrepresentation of $\Pi(\ul{k}_J,w)$ by (\ref{equ: lpl-injJ}). In fact, by the description of locally algebraic vectors of $\Pi$ (\cite[Thm.5.3]{New}), one sees  $z$ is classical \big(for $z\in \cE(\ul{k}_J,w)_{\overline{\rho}}$\big) if and only if
\begin{equation*}
  \big(J_B(\Pi_{\lalg})\otimes_E \overline{E}\big)^{\cH^p=\lambda,T(F_{\wp})=\chi}\neq 0,
\end{equation*}
where ``$\lalg$'' denotes the locally algebraic vectors.

%where ``$\lalg$" denotes the locally algebraic vectors. Note that $\Pi(\ul{k}_J,w)_{\lalg}^{Z_1=\cN^{-w}}=\Pi_{\lalg}^{Z_1=\cN^{-w}}\cap \Pi(\ul{k}_J,w)$, and that $z$ being classical would imply that $\chi$ being locally algebraic and dominant.

For a locally analytic character $\chi$ of $T(F_{\wp})$ over $E$, put
\begin{equation*}C(\chi):=\{\sigma\in \Sigma_{\wp}\ |\ \wt(\chi)_{1,\sigma}-\wt(\chi)_{2,\sigma}\in \Z_{\geq 0}\}; \end{equation*}
for $S\subseteq C(\chi)$, put
\begin{equation*}\chi_S^c:=\chi\prod_{\sigma\in S}(\sigma^{\wt(\chi)_{2,\sigma}-\wt(\chi)_{1,\sigma}-1}\otimes \sigma^{\wt(\chi)_{1,\sigma}-\wt(\chi)_{2,\sigma}+1}\big).\end{equation*}
%where $\wt(\chi)=(\wt(\chi)_{1,\sigma},\wt(\chi)_{2,\sigma})_{\sigma\in \Sigma_{\wp}}$ denotes the weight of $\chi$. For a locally $\Q_p$-analytic character $\chi$ of $T(F_{\wp})$ over $E$,
Let
\begin{equation*}
  I(\chi):=\soc \big(\Ind_{\overline{B}(F_{\wp})}^{\GL_2(F_{\wp})} \chi\big)^{\Q_p-\an}.
\end{equation*}
Note $I(\chi)$ is locally algebraic if and only if $\chi$ is locally algebraic and dominant.
\begin{definition}
  Let $z=(\chi,\lambda)$ be a closed point of $\cE(\ul{k}_J,w)_{\overline{\rho}}$, for $S\subseteq C(\chi)\cap J$, we say $z$ admits an $S$-companion point if $z_S^c:=(\chi_S^c, \lambda)$ is also a closed point of $\cE(\ul{k}_{J},w)_{\overline{\rho}}$.
\end{definition}
Denote by $\delta_B=\unr(q^{-1})\otimes \unr(q)$ the modulus character of $B(F_{\wp})$.
\begin{lemma}\label{lem: lpl-ccl}
  (1) Let $z=(\chi,\lambda)$ be a closed point of $\cE(\ul{k}_J,w)_{\overline{\rho}}$ with $\chi$ locally algebraic and dominant, suppose for any $\emptyset\neq S\subseteq \Sigma_{\wp}\setminus J$, $I(\chi_S^c\delta_B^{-1})$ is not a subrepresentation of $\Pi(\ul{k}_J,w)^{\cH^p=\lambda}$, then the point $z$ is classical. We call the points satisfying this assumption $\Sigma_{\wp}\setminus J$-very classical.

  (2) Let $z=(\chi,\lambda)$ be a closed point of $\cE(\ul{k}_J,w)_{\overline{\rho}}$ with $\chi$ locally algebraic and dominant, then $z$ is $\Sigma_{\wp}\setminus J$-very classical if and only if $z$ does not have $S$-companion point for all $\emptyset \neq S \subseteq \Sigma_{\wp}\setminus J$.

  (3) Let $z$ be a $\Sigma_{\wp}\setminus J$-very classical point of $\cE(\ul{k}_J,w)_{\overline{\rho}}$, then the natural injection
  \begin{equation}\label{equ: lpl-itcl}
    J_B\big(\Pi(\ul{k}_J,w)_{\lalg}\big)^{T(\co_{\wp})=\chi}[\cH^p=\lambda,T(F_{\wp})=\chi]\hooklongrightarrow J_B\big(\Pi(\ul{k}_J,w)\big)^{T(\co_{\wp})=\chi}[\cH^p=\lambda,T(F_{\wp})=\chi]
  \end{equation}
  is an isomorphism \big(where $T(\co_{\wp})=\co_{\wp}^{\times}\times \co_{\wp}^{\times}\hookrightarrow T(F_{\wp})$\big).
\end{lemma}
\begin{proof}
  (1) Suppose $z$ is not classical, by  \cite[Lem.6.2.25]{Ding}, there exists $\emptyset \neq S\subseteq \Sigma_{\wp}\setminus J$ such that $z$ admits an effective $S$-companion point (we refer to \cite[Def.6.2.21]{Ding} for effective companion points) which induces, by (an easy variation of) \cite[Prop.6.2.23]{Ding}, an injection  $I(\chi_S^c\delta_B^{-1})\hookrightarrow \Pi(\ul{k}_J,w)^{\cH^p=\lambda}$, (1) follows.

  (2) If $z$ admits an $S$-companion point for $\emptyset \neq S \subseteq \Sigma_{\wp}\setminus J$, as in the proof of \cite[Lem.6.2.25]{Ding}, there exists $S'\supseteq S$, $S'\subseteq \Sigma_{\wp}\setminus J$ such that $z$ admits an effective $S'$-companion point, which induces an inclusion $I(\chi_{S'}^c\delta_B^{-1})\hookrightarrow \Pi(\ul{k}_J,w)^{\cH^p=\lambda}$. Conversely, if there exists $\emptyset \neq S \subseteq \Sigma_{\wp}\setminus J$ such that $I(\chi_S^c\delta_B^{-1})\hookrightarrow \Pi(\ul{k}_J,w)^{\cH^p=\lambda}$, applying the Jacquet-Emerton functor, we get the $S$-companion point $z_S^c$ of $z$.

  (3) By the same arguments as in \emph{loc. cit.}, together with \cite[Lem.6.3.15]{Ding}, if (\ref{equ: lpl-itcl}) is not bijective, there exists $\emptyset \neq S\subseteq \Sigma_{\wp}\setminus J$, such that $z$ admits an effective $S$-companion points, and hence $I(\chi_S^c\delta_B^{-1})\hookrightarrow  \Pi(\ul{k}_J,w)^{\cH^p=\lambda}$, a contradiction.
   %$I(\chi_S^c\delta_{B}^{-1})\hookrightarrow \Pi(\ul{k}_J,w)[\cH^p=\lambda]$, from which we deduce
  %\begin{equation*}
   % \Hom_{\GL_2(F_{\wp})}\big(I(\chi_S^c\delta_B^{-1}), \Pi(\ul{k}_J,w)^{\cH^p=\lambda}\big)\neq 0,
  %\end{equation*}
  %a contradiction.
\end{proof}
\begin{remark}
  The lemma can also be deduced from Breuil's adjunction formula \cite[Thm.4.3]{Br13II}.
\end{remark}
Since $\Pi(\ul{k}_J,w)$ is contained in the unitary Banach representation $\widetilde{H}^1(K^p,E)$, the following proposition follows easily from \cite[Prop.5.1]{Br}:
\begin{proposition}\label{prop: lpl-cla}
  Let $z=(\chi,\lambda)$ be closed point in $\cE(\ul{k}_J,w)_{\overline{\rho}}$ with $\chi$ locally algebraic and dominant, and suppose
  \begin{equation}\label{equ: lpl-ncr}
    \us(q\chi_1(\varpi))< \inf_{\sigma\in \Sigma_{\wp}\setminus J}\{\wt(\chi)_{1,\sigma}-\wt(\chi)_{2,\sigma}+1\}
  \end{equation} then the point $z$ is $\Sigma_{\wp}\setminus J$-very classical.
\end{proposition}
A closed point $z=(\chi,\lambda)$ of $\cE(\ul{k}_J,w)_{\overline{\rho}}$ is called \emph{spherical} if $\chi$ is the product of an unramified character with an algebraic character \big(i.e. $\wt(\chi)\in \Z^{2|d|}$ and $\chi\delta_{\wt(\chi)}^{-1}$ is unramified\big). By the standard arguments as in \cite[\S 6.4.5]{Che} (see also \cite[Prop.6.2.7]{Che}), one can deduce from Prop.\ref{prop: lpl-cla}:
\begin{theorem}\label{thm: lpl-dense}
  (1) The set of spherical points satisfying the assumption in Prop.\ref{prop: lpl-cla} are Zariski dense in $\cE(\ul{k}_J,w)_{\overline{\rho}}$ and accumulates over spherical points.

  (2) The set of points satisfying the assumption in Prop.\ref{prop: lpl-cla} accumulates over points with integer weights.
\end{theorem}
By Chenevier's method \cite[\S 4.4]{Che11}, one can prove
\begin{proposition}\label{prop: lpl-njz}
  Let $z\in \cE(\ul{k}_J,w)_{\overline{\rho}}(E)$ be a $\Sigma_{\wp}\setminus J$-very classical point, then the weight map $\kappa$ is \'etale at $z$; moreover, there exists an affinoid neighborhood $U$ of $z$ with $\kappa(U)$ affinoid open in $\cW(\ul{k}_J)$ such that $\co(U)\cong \co(\kappa(U))$.
\end{proposition}
\begin{proof}
  Indeed, by Prop.\ref{prop: lpl-evb}, Thm.\ref{thm: lpl-dense}, one can reduce to a similar situation as in the beginning of the proof of \cite[Thm.4.8]{Che11}. Since $z$ is $\Sigma_{\wp}\setminus J$-very classical, one has the bijection (\ref{equ: lpl-itcl}) (which is an analogue of \cite[(4.20)]{Che11}, see also \cite[Lem.3.27]{Ding3}). The proposition then follows from the multiplicity one result for automorphic representations of $G(\bA)$, by the same argument as in the proof of \cite[Thm.4.8]{Che11} (see also \cite[\S 3.4.3]{Ding3} especially the arguments after \cite[Lem.3.27]{Ding3}).
\end{proof}
\subsubsection{Families of Galois representations}By Carayol's results \cite{Ca2}, the theory of pseudo-characters and the density of classical points, we have% (cf. \cite[Prop.3.16]{Ding3})
\begin{theorem}
  For a closed point $z=(\chi,\lambda)$ of $\cE(\ul{k}_{J},w)_{\overline{\rho}}$, there exists a unique continuous irreducible representation $\rho_z: \Gal_{F}\ra \GL_2(k(z))$ which is unramified at places $\fl\notin S(K^p)$ satisfying $\rho_z(\Frob_{\fl}^{-2}) -\lambda(T_{\fl})\rho_z(\Frob_{\fl}^{-1}) + \lambda(S_{\fl})=0$,  where $k(z)$ denotes the residue field at $z$.
\end{theorem}
By the fact that $\rho_z|_{\Gal_{F_{\wp}}}$ is de Rham for classical points $z\in \cE(\ul{k}_J,w)_{\overline{\rho}}$ \big(and of Hodge-Tate weights $\big(\frac{w-k_{\sigma}+2}{2},\frac{w+k_{\sigma}}{2}\big)$ for $\sigma\in J$\big), Shah's results \cite{Sha} and the density of classical points, one can prove as in \cite[Prop.6.2.40]{Ding}
\begin{theorem}\label{prop: lpl-ocj}
  Let $z\in \cE(\ul{k}_{J},w)_{\overline{\rho}}(\overline{E})$, the restriction $\rho_z|_{\Gal_{F_{\wp}}}$ is $J$-de Rham of Hodge-Tate weights $\big(\frac{w-k_{\sigma}+2}{2},\frac{w+k_{\sigma}}{2}\big)$ for $\sigma\in J$.
\end{theorem}
\begin{proposition}\label{prop: lpl-uoc}
  For $z\in \cE(\ul{k}_J,w)_{\overline{\rho}}(\overline{E})$, there exists an open affinoid $U$ of $\cE(\ul{k}_J,w)_{\overline{\rho},\red}$ and a continuous representation $\rho_U: \Gal_F\ra \GL_2(\co(U))$ such that the specialization of $\rho_U$ at any point $z'\in U(\overline{E})$ equals $\rho_{z'}$. Moreover, for $\sigma\in J$, $D_{\dR}(\rho_U)_{\sigma}:=\big(B_{\dR,\sigma}\widehat{\otimes}_E \rho_U\big)^{\Gal_{F_{\wp}}}$ is a locally free $\co(U)$-module of rank $2$.
\end{proposition}
\begin{proof}
  The first part follows from \cite[Lem.5.5]{Bergd}; the second is hence from  Prop.\ref{prop: lpl-ocj} and \cite[Thm.2.19]{Sha} applied to $\rho_U$.
\end{proof}
By \cite{Sa}, $\rho_{z,\wp}:=\rho_z|_{\Gal_{F_{\wp}}}$ is semi-stable (thus trianguline) for any spherical classical point $z$ of $\cE(\ul{k}_{\emptyset},w)_{\overline{\rho}}$. As in \cite[Cor.6.2.50]{Ding}, by global triangulation theory \cite{KPX} \cite{Liu} applied to $\cE(\ul{k}_{\emptyset},w)_{\overline{\rho}}$ \big(note $\cE(\ul{k}_J,w)_{\overline{\rho}}$ is a closed rigid subspace of $\cE(\ul{k}_{\emptyset},w)_{\overline{\rho}}$\big), we get
\begin{theorem}\label{prop: lpl-gpa}
  For any closed point $z=(\chi=\chi_{1}\otimes \chi_{2},\lambda)$ of $\cE(\ul{k}_J,w)_{\overline{\rho}}$, $\rho_{z,\wp}$ is trianguline with a triangulation given by
  \begin{equation*}
    0\ra \cR_{k(z)}(\delta_{1}) \ra D_{\rig}(\rho_{z,\wp}) \ra \cR_{k(z)}(\delta_{2})\ra 0
  \end{equation*}
  with \begin{equation*}
    \begin{cases}\delta_{1}=\unr(q) \chi_{1} \prod_{\sigma\in \Sigma_z} \sigma^{\wt(\chi)_{2,\sigma}-\wt(\chi)_{1,\sigma}-1}, \\
    \delta_{2,z}=\chi_{2}\prod_{\sigma\in \Sigma_{\wp}}\sigma^{-1}\prod_{\sigma\in \Sigma_z} \sigma^{\wt(\chi)_{1,\sigma}-\wt(\chi)_{2,\sigma}+1},
    \end{cases}
  \end{equation*}
  where $\Sigma_z\subseteq C(\chi)$, $\cR_{k(z)}$ denotes the Robba ring $B_{\rig,F_{\wp}}^{\dagger} \otimes_{\Q_p} k(z)$, $D_{\rig}(\rho_{z,\wp}):=\big(B_{\rig}^{\dagger}\otimes_{\Q_p} \rho_{z,\wp}\big)^{\Gal_{F_{\wp}}}$ is the $(\varphi,\Gamma)$-module (of rank $2$) over $\cR_{k(z)}$ associated to $\rho_{z,\wp}$ (we refer to \cite{Berger} for $B_{\rig,F_{\wp}}^{\dagger}$, $B_{\rig}^{\dagger}$ and $(\varphi,\Gamma)$-modules).
\end{theorem}
%\begin{remark}
 % There is a tricky part when applying global triangulation theory to get the above proposition: the %  \emph{non-critical} classical points might not be Zariski-dense in $\cE(J;\ul{k}_{\Sigma_{\wp}\setminus J},w)_{\overline{\rho}}$ except when $J=\Sigma_{\wp}$ (while the \emph{non-$J$-critical} classical points are indeed Zariski-dense in  $\cE(J;\ul{k}_{\Sigma_{\wp}\setminus J},w)_{\overline{\rho}}$ by Prop.\ref{prop: lpl-clt}, cf. Def.\ref{def:: lpl-zie} below). So one can first apply global triangulation to $\cE(\Sigma_{\wp};\ul{k}_{\emptyset},w)_{\overline{\rho}}$ (in which the  non-critical classical points are Zariski-dense). Since $\cE(J;\ul{k}_{\Sigma_{\wp}\setminus J},w)_{\overline{\rho}}$ is a closed subspace of $\cE(\Sigma_{\wp};\ul{k}_{\emptyset},w)_{\overline{\rho}}$ by Prop.\ref{prop: lpl-ige}, one can thus deduce Prop.\ref{prop: lpl-gpa} for general $\cE(J;\ul{k}_{\Sigma_{\wp}\setminus J},w)_{\overline{\rho}}$ from that for $\cE(\Sigma_{\wp};\ul{k}_{\emptyset},w)_{\overline{\rho}}$.
%\end{remark}

\begin{corollary}\label{cor lpl-cpc}
  Let $z=(\chi,\lambda)\in \cE(\ul{k}_J,w)_{\overline{\rho}}(\overline{E})$ and suppose
  \begin{equation}\label{equ: lpl-dis}\unr(q)\chi_1\chi_2^{-1}\neq \prod_{\sigma\in \Sigma_{\wp}}\sigma^{n_{\sigma}} \text{ for all $(n_{\sigma})_{\sigma\in \Sigma_{\wp}}\in \Z^{d}$,}
  \end{equation}
  for $S\subseteq \Sigma_{\wp}\setminus J$, if $z$ admits an $S$-companion point then $S\subseteq \Sigma_z$.
\end{corollary}
\begin{proof}
  Applying Prop.\ref{prop: lpl-gpa} to the point $z_S^c$, the corollary then follows from \cite[Thm.3.7]{Na}.
\end{proof}
One can moreover deduce from the proof of \cite[Thm.6.3.9]{KPX} (see also \cite[Prop.6.2.49]{Ding}):
\begin{proposition}\label{prop: lpl-zccs}
  Let $z$ be a classical point of $\cE(\ul{k}_J,w)_{\overline{\rho}}$, $U$ be an affinoid neighborhood of $z$, suppose any closed point of $U$ satisfies (\ref{equ: lpl-dis}), then for any $\sigma\in \Sigma_{\wp}$, $Z_{U,\sigma}:=\{z'\in U(\overline{E})\ |\ \sigma\in \Sigma_{z'}$ is a Zariski-closed subset of $U(\overline{E})$.
\end{proposition}
\begin{definition} Let $z=(\chi,\lambda)$ be a closed point of $\cE(\ul{k}_J,w)_{\overline{\rho}}$, for $S\subseteq \Sigma_{\wp}$, we say $z$ is non-$S$-critical if (\ref{equ: lpl-dis}) is satisfied and $\Sigma_z\cap S=\emptyset$.
\end{definition}
 % \begin{equation*}
  %  \begin{cases}
  %    \unr(q^{-1}) \chi_{1,z}^{-1}\chi_{2,z}\neq \prod_{\sigma\in \Sigma_{\wp}} \sigma^{n_{\sigma}} \text{ for any $\ul{n}_{\Sigma_{\wp}}\in \Z^{d}$}, \\
  %    \Sigma_z\cap S=\emptyset.
   % \end{cases}
  %\end{equation*}

\begin{corollary}\label{cor: lpl-ncvc}
  Let $z=(\chi,\lambda)\in \cE(\ul{k}_J,w)_{\overline{\rho}}(\overline{E})$ with $\chi$ locally algebraic and $C(\chi)=\Sigma_{\wp}$, if $z$ is non-$\Sigma_{\wp}\setminus J$-critical, then $z$ is $\Sigma_{\wp}\setminus J$-very classical.
\end{corollary}
\begin{proof}
  By Lem.\ref{lem: lpl-ccl} (2), it's sufficient to show $z$ does not have $S$-companion point for $\emptyset \neq S \subseteq \Sigma_{\wp}\setminus J$. But this follows from Cor.\ref{cor lpl-cpc}.
\end{proof}
\begin{theorem}\label{thm: lpl-eta}
  Let $z=(\chi,\lambda)\in \cE(\ul{k}_J,w)_{\overline{\rho}}(E)$ be a non-$\Sigma_{\wp}\setminus J$-critical classical point, then the weight map $\kappa$ is \'etale at $z$. Moreover, there exists an affinoid neighborhood $U$ of $z$ such that $W=\kappa(U)$ is an affinoid open in $\cW(\ul{k}_J)$ and $\co(U)\cong \co(W)$.
\end{theorem}
\begin{proof}
  The theorem follows from Prop.\ref{prop: lpl-njz} combined with Cor.\ref{cor: lpl-ncvc}.
\end{proof}
The following proposition, which follows from the same argument as in \cite[Cor.3.26]{Ding3}, would be useful to apply the adjunction formula in families.
\begin{proposition}\label{prop: lpl-ncne}
  Let $z=(\chi,\lambda)\in \cE(\ul{k}_J,w)_{\overline{\rho}}$ be a non-$\Sigma_{\wp}\setminus J$-critical classical point, and suppose $\unr(q^e)\psi_{\chi,1}(p)\psi_{\chi,2}^{-1}(p)\neq 1$ where $\psi_{\chi}:=\chi \delta_{\wt(\chi)}^{-1}$, then there exists an admissible open $U$ of $z$ in $\cE(\ul{k}_J,w)_{\overline{\rho}}$ such that any point of $U$ is non-$\Sigma_{\wp}\setminus J$-critical.
\end{proposition}
\subsection{Local-global compatibility}\label{sec: lpl-4.3}
Let $\rho: \Gal_F \ra \GL_2(E)$ be a continuous representation such that
\begin{enumerate}
  \item $\rho_{\wp}:=\rho|_{\Gal_{F_{\wp}}}$ is semi-stable non-crystalline of Hodge-Tate weights $\ul{h}_{\Sigma_{\wp}}=(\frac{w-k_{\sigma}+2}{2}, \frac{w+k_{\sigma}}{2})_{\sigma\in \Sigma_{\wp}}$ for $k_{\sigma}\in 2\Z_{\geq 1}$ and $w\in 2\Z$ with $\{\alpha,q\alpha\}$ the eigenvalues of $\varphi^{d_0}$ on $D_{\st}(\rho_{\wp})$, $S_c:=S_c(\rho_{\wp})$ (cf. the discussion before Cor.\ref{cor: lpl-rvs}) the set of embeddings where $\rho_{\wp}$ is critical, $S_n:=S_n(\rho_{\wp})=\Sigma_{\wp}\setminus S_c$ and $\ul{\cL}_{S_n}\in E^{|S_n|}$ the associated Fontaine-Mazur $\cL$-invariants;
  \item $\Hom_{\Gal_F}\big(\rho, H^1_{\et}(K^p,W(\ul{k}_{\Sigma_p},w))\big)\neq 0$ (in particular, $\rho$ is associated to certain Hilbert eigenforms);
  \item $\rho$ is absolutely irreducible modulo $\varpi_E$.
\end{enumerate}
Note that, by the condition (2), $\rho$ is unramified for places in $S(K^p)$. And by the Eichler-Shimura relations, one can associate to $\rho$ a system of eigenvalues $\lambda_{\rho}:\cH^p \ra E$. Put $\widehat{\pi}(\rho):=\Hom_{\Gal_F}\big(\rho, \widetilde{H}^1_{\et}(K^p,E)\big)$, which is an admissible  unitary Banach representation of $\GL_2(F_{\wp})$. One has
\begin{equation*}
  \widehat{\pi}(\rho)=\Hom_{\Gal_F}\big(\rho, \Pi\big)=\Hom_{\Gal_F}\big(\rho,\Pi^{\cH^p=\lambda_{\rho}}\big).
\end{equation*}
The injection $H^1_{\et}\big(K^p,W(\ul{k}_{\Sigma_{\wp}},w)\big)_{\overline{\rho}} \hookrightarrow \widetilde{H}^1_{\et}\big(K^p,W(\ul{k}_{\Sigma_p},w)\big)_{\overline{\rho},\Q_p-\an}$ induces an injection (e.g. see \cite[Prop.5.1.3]{Ding})
\begin{equation*}
  H^1_{\et}(K^p,W(\ul{k}_{\Sigma_{\wp}},w))_{\overline{\rho}} \otimes_E W(\ul{k}_{\Sigma_{\wp}},w)^{\vee} \hooklongrightarrow \widetilde{H}^1(K^p,E)_{\Q_p-\an},
\end{equation*}
thus the condition (2) implies in particular $\widehat{\pi}(\rho)\neq 0$.

By local-global compatibility in classical local Langlands correspondence (for $\ell=p$, cf. \cite{Sa}) and the isomorphism in Prop.\ref{prop: lpl-rnf} (1), there exists an isomorphism of locally algebraic representations of $\GL_2(F_{\wp})$:
\begin{equation}\label{equ: lpl-psr}
  \St(\alpha,\ul{h}_{\Sigma_{\wp}})^{\oplus r} \xlongrightarrow{\sim}\widehat{\pi}(\rho)_{\lalg},
\end{equation}
with some $r\in \Z_{\geq 1}$ \big(note $\alg(\ul{h}_{\Sigma_{\wp}})\cong W(\ul{k}_{\Sigma_{\wp}},w)^{\vee}$ and thus $\St(\alpha,\ul{h}_{\Sigma_{\wp}})\cong \St \otimes_E \unr(\alpha) \circ \dett \otimes_E W(\ul{k}_{\Sigma_{\wp}},w)^{\vee}$\big). The main result of this section is (see \S \ref{sec: lpl-3} for notations)
\begin{theorem}\label{thm: lpl-giao}(1) Let $\tau\in \Sigma_{\wp}$, then $\tau\in S_c$ if and only $I^c_{\tau}(\alpha,\ul{h}_{\Sigma_{\wp}})$ is a subrepresentation of $\widehat{\pi}(\rho)$.

  (2) The natural restriction map
  \begin{equation}\label{equ: lpl-kbij}
    \Hom_{\GL_2(F_{\wp})}\big(\Sigma(\alpha,\ul{h}_{\Sigma_{\wp}},\ul{\cL}_{S_n}),\widehat{\pi}(\rho)_{\Q_p-\an}\big)
    \lra \Hom_{\GL_2(F_{\wp})} \big(\St(\alpha,\ul{h}_{\Sigma_{\wp}}),\widehat{\pi}(\rho)_{\Q_p-\an}\big)
  \end{equation}
  is bijective. In particular, $\Sigma(\alpha,\ul{h}_{\Sigma_{\wp}},\ul{\cL}_{S_n})$ is a subrepresentation of $\widehat{\pi}(\rho)_{\Q_p-\an}$.
\end{theorem}
By the same argument as in the proof of \cite[Cor.4.7]{Ding3}, we have
\begin{corollary}\label{cor: lpl-luni}
  Let $\ul{\cL}'_{S_n}\in E^d$, then $\Sigma(\alpha,\ul{h}_{\Sigma_{\wp}},\ul{\cL}'_{S_n})$ is a subrepresentation of $\widehat{\pi}(\rho)$ if and only if  $\ul{\cL}'_{S_n}=\ul{\cL}_{S_n}$. %Consequently, $\Sigma_0(\rho_{\wp},\ul{\cL'}_{S_n}) \otimes_E W(\ul{k}_{S_c},w)^{\vee}$ is a subrepresentation of $\widehat{\pi}(\rho)_{\Q_p-\an}$ if and only if $\ul{\cL'}_{S_n} =\ul{\cL}_{S_n}$.
\end{corollary}
Combining Cor.\ref{cor: lpl-luni} and Thm.\ref{thm: lpl-giao} (1), we see
\begin{corollary}
  The local Galois representation $\rho_{\wp}$ can be determined by $\widehat{\pi}(\rho)$.
\end{corollary}

\begin{proof}[Proof of the theorem \ref{thm: lpl-giao}] First note that we only need (and do) prove the same result with $\widehat{\pi}(\rho)$ replaced by $\Pi^{\cH^p=\lambda_{\rho}}$.
For $S\subseteq \Sigma_{\wp}$, $\sigma\in \Sigma_{\wp}^{\sigma}$, put $S^{\sigma}:=S\setminus \{\sigma\}$. Note the injection $\St(\alpha,\ul{h}_{\Sigma_{\wp}})\hookrightarrow \Pi(\ul{k}_{\emptyset},w)^{\cH^p=\lambda}$ gives a spherical classical point $z_{\rho}=(\chi_{\rho},\lambda_{\rho})\in \cE(\ul{k}_{\emptyset},w)_{\overline{\rho}}(E)$ where $\chi_{\rho}:=\chi(\alpha,\ul{h}_{\Sigma_{\wp}})\delta_B$; moreover, for any $S \subseteq \Sigma_{\wp}$, $z_{\rho}\in\cE(\ul{k}_{J},w)_{\overline{\rho}}(E)$.

(1) Let $\tau\in \Sigma_{\wp}$, and consider $\Pi(\ul{k}_{\Sigma_{\wp}^{\tau}},w)$ and $\cE(\ul{k}_{\Sigma_{\wp}^{\tau}},w)_{\overline{\rho}}$. By \cite[Prop.6.2.23]{Ding} (and Lem.\ref{lem: lpl-lav}), $I_{\tau}^c(\alpha,\ul{h}_{\Sigma_{\wp}})\hookrightarrow \Pi(\ul{k}_{\Sigma_{\wp}^\tau},w)^{\cH^p=\lambda_{\rho}}$ \big(which is equivalent to $I_{\tau}^c(\alpha,\ul{h}_{\Sigma_{\wp}})\hookrightarrow \Pi^{\cH^p=\lambda_{\rho}}$, since any latter morphism factors through $\Pi(\ul{k}_{\Sigma_{\wp}^\tau},w)$ by Lem.\ref{lem: lpl-lav}\big) if and only if $(z_{\rho})_{\tau}^c=((\chi_{\rho})_{\tau}^c,\lambda_{\rho})\in \cE(\ul{k}_{\Sigma_{\wp}^\tau},w)_{\overline{\rho}}$.

If $(z_{\rho})_{\tau}^c\in \cE(\ul{k}_{\Sigma_{\wp}^{\tau}},w)_{\overline{\rho}}$, by Cor.\ref{cor lpl-cpc}, $\tau\in \Sigma_{z_{\rho}}=S_c$, the ``if" part follows.

Now suppose $\tau\in S_c=\Sigma_{z_{\rho}}$, we first use Bergdall's method \cite{Bergd14} to show the weight map $\kappa: \cE(\ul{k}_{\Sigma_{\wp}^{\tau}},w) \ra \cW_1(\ul{k}_{\Sigma_{\wp}^{\tau}})_{\overline{\rho}}$ is not \'etale at ${z_{\rho}}$:

We only need to consider the case where $\cE(\ul{k}_{\Sigma_{\wp}^{\tau}},w)_{\overline{\rho}}$ is reduced at ${z_{\rho}}$ since otherwise, $\kappa$ is not \'etale at ${z_{\rho}}$ \big(in fact, by the same argument as in \cite[\S 3.8]{Che05}, one can probably prove that  $\cE(\ul{k}_{\Sigma_{\wp}^{\tau}},w)_{\overline{\rho}}$ is reduced at ${z_{\rho}}$\big). Take $U$ to be an irreducible affinoid neighborhood of ${z_{\rho}}$ in $\cE(\ul{k}_{\emptyset},w)_{\overline{\rho}}$ small enough such that Prop.\ref{prop: lpl-uoc} holds. The composition $\co(U_{\red})\ra \cE(\ul{k}_{\Sigma_{\wp}^{\tau}},w)_{\overline{\rho}}\ra \widehat{T}$ gives a continuous character $\widetilde{\delta}: T(L) \ra \co(U_{\red})^{\times}$ \big(with $\widetilde{\delta}|_{Z_1}=\cN^{-w}$\big). By \cite[Prop.4.3.5]{Liu}, there exists (shrinking $U$ if necessary) an injection of $(\varphi,\Gamma)$-modules over $\cR_{\co(U_{\red})}:=\cR_{F_{\wp}}\widehat{\otimes}_{\Q_p} \co(U_{\red})$:
\begin{equation}\label{equ: lpl-aqg}
  \cR_{\co(U_{\red})}\big(\widetilde{\delta}_1\unr(q)\big)\hooklongrightarrow D_{\rig}(\rho_{U_{\red}});
\end{equation}
moreover, the specialisation of the above morphism to any closed point in $U$ is still injective. Let $t: \Spec E[\epsilon]/\epsilon^2\ra U_{\red}$ be an element in the tangent space of $U_{\red}$ at ${z_{\rho}}$, one deduces from (\ref{equ: lpl-aqg}) an injection of $(\varphi,\Gamma)$-modules over $\cR_{E[\epsilon]/\epsilon^2}$ \big(where the injectivity follows from the fact that (\ref{equ: lpl-aqg}) specializing to ${z_{\rho}}$ is still injective\big) \begin{equation}\label{equ: lpl-nnt}\cR_{E[\epsilon]/\epsilon^2}\big(t^* \widetilde{\delta}_1\unr(q)\big) \hooklongrightarrow D_{\rig}(t^* \rho_{U_{\red}}).
 \end{equation}
Note $t^*\widetilde{\delta} \equiv \chi_{\rho} \pmod{\epsilon}$ and $D_{\rig}(t^* \rho_{U_{\red}})$ is an extension of $D_{\rig}(\rho)$ by $D_{\rig}(\rho)$. Since $\widetilde{\delta}|_{Z_1}=\cN^{-w}$,  $\wt(t^*\widetilde{\delta})=\big(\frac{k_{\sigma}-w-2}{2}-a_{\sigma}\epsilon, \frac{2-k_{\sigma}-w}{2}+a_{\sigma}\epsilon\big)_{\sigma\in \Sigma_{\wp}}$ for $(a_{\sigma})_{\sigma\in \Sigma_{\wp}}\in E^d$ and thus the Sen weights of $D_{\rig}(t^* \rho_{U_{\red}})$ are given by $\big(\frac{-k_{\sigma}-w}{2}+a_{\sigma}\epsilon,\frac{k_{\sigma}-w-2}{2}-a_{\sigma}\epsilon\big)_{\sigma\in \Sigma_{\wp}}$. The map (\ref{equ: lpl-nnt}) induces an injection
\begin{equation*}
  \cR_{E[\epsilon]/\epsilon^2}  \hooklongrightarrow D_{\rig}(t^*\rho_{U_{\red}})\otimes_{E[\epsilon]/\epsilon^2} \cR_{E[\epsilon]/\epsilon^2}\big((t^* \widetilde{\delta}_1\unr(q))^{-1}\big)=:D,
\end{equation*}
where $D$ is an extension of $D_{\rig}(\rho)\otimes_{\cR_E} \cR_E(\chi_{\rho,1}^{-1}\unr(q^{-1}))$ by itself and has Sen weights $\big((1-k_{\sigma})+2a_{\sigma}\epsilon, 0\big)_{\sigma\in \Sigma_{\wp}}$.
By the same argument of \cite[Lem.9.6]{Br13II} \big(replacing the functor $D_{\cris}(\cdot)$ by $D_{\st}(\cdot)$\big), one can show $1-k_{\tau}$ is a constant Sen weight of $D$, and hence $a_{\tau}=0$. Consequently, we see the composition $T_{U_{\red},{z_{\rho}}}\ra T_{\cW_1,\kappa({z_{\rho}})} \ra T_{\cW_1(\ul{k}_{\Sigma_{\wp}^{\tau}}),\kappa({z_{\rho}})}$ is zero, where the $T_{X,x}$ denotes the tangent space of $X$ at $x$ for a closed point $x$ in a rigid analytic space $X$, and the first map denotes the tangent map induced by $\kappa$. Thus the map $T_{\cE(\ul{k}_{\emptyset},w)_{\overline{\rho},\red},{z_{\rho}}}\ra T_{\cW_1(\ul{k}_{\Sigma_{\wp}^{\tau}}),\kappa({z_{\rho}})}$ is zero; however, since we assume $\cE(\ul{k}_{\Sigma_p^{\wp}},w)_{\overline{\rho}}$ to be reduced at ${z_{\rho}}$, we see the induced tangent map $T_{\cE(\ul{k}_{\Sigma_{\wp}^{\tau}},w)_{\overline{\rho}},{z_{\rho}}}\ra T_{\cW_1(\ul{k}_{\Sigma_{\wp}^{\tau}}),\kappa({z_{\rho}})}$ factors though the above zero map and thus also equals zero, from which we see $\kappa: \cE(\ul{k}_{\Sigma_{\wp}^{\tau}},w)_{\overline{\rho}}\ra \cW_1(\ul{k}_{\Sigma_{\wp}^{\tau}})$ is not \'etale at ${z_{\rho}}$.

By Prop.4.10, ${z_{\rho}}$ is not $\Sigma_{\wp}^{\tau}$-very classical, and hence by definition, $I_{\tau}^c(\alpha,\ul{h}_{\Sigma_{\wp}})\cong I\big((\chi_{\rho})_{\tau}^c\delta_B^{-1}\big)$ is a subrepresentation of $\Pi(\ul{k}_{\Sigma_{\wp}^{\tau}},w)^{\cH^p=\lambda_{\rho}}$, which concludes the proof of Thm.\ref{thm: lpl-giao} (1).

(2) We use the same arguments as in \cite[\S 4.3]{Ding3}. The injectivity of (\ref{equ: lpl-kbij}) \big(with $\widehat{\pi}(\rho)$ replaced by $\Pi^{\cH^p=\lambda_{\rho}}$\big) follows from the fact that ${z_{\rho}}$ (as a classical point of $\cE(\ul{k}_{S_c},w)_{\overline{\rho}}$) does not have $S$-companion point for $\emptyset \neq S \subseteq S_n$. Indeed, if (\ref{equ: lpl-kbij}) is not injective, by results on the Jordan-Holder factors of $\Sigma(\alpha,\ul{h}_{\Sigma_{\wp}})$ (e.g. see \cite[Thm.4.1]{Br}), we see either $F(\alpha,\ul{h}_{\Sigma_{\wp}})$ is a subrepresentation of $\Pi^{\cH^p=\lambda}$, or there exists $\emptyset\neq S\subseteq S_n$ such that $I((\chi_{\rho})_S^c\delta_B^{-1})$ is a subrepresentation of $\Pi^{\cH^p=\lambda}$ which are both impossible since the locally algebraic representation $F(\alpha,\ul{h}_{\Sigma_{\wp}})$ can not be injected into $\Pi^{\cH^p=\lambda}$ by (\ref{equ: lpl-psr}), and ${z_{\rho}}$ is non-$S_n$-critical hence $S_n$-very classical by Cor.\ref{cor: lpl-ncvc}.

Let $\ul{h}'_{S_n}:=\big(\frac{2-k_{\sigma}}{2},\frac{k_{\sigma}}{2}\big)_{\sigma\in S_n}=\ul{h}_{S_n}-\big(\frac{w}{2},\frac{w}{2}\big)_{\sigma\in S_n}$, thus by definition $\St(\alpha,\ul{h}_{\Sigma_{\wp}})\cong \St(\alpha,\ul{h}'_{S_n})\otimes_E W(\ul{k}_{S_c},w)^{\vee}$, $\Sigma(\alpha,\ul{h}_{\Sigma_{\wp}})\cong \Sigma(\alpha,\ul{h}'_{S_n})\otimes_E W(\ul{k}_{S_c},w)^{\vee}$, $\chi(\alpha,\ul{h}_{\Sigma_{\wp}})=\chi(\alpha,\ul{h}'_{S_n})\chi(\ul{k}_{S_c},w)$,  and $\Sigma(\alpha,\ul{h}_{\Sigma_{\wp}}, \ul{\cL}_{S_n})\cong \Sigma(\alpha,\ul{h}'_{S_n},\ul{\cL}_{S_n})\otimes_E W(\ul{k}_{S_c},w)^{\vee}$ (see Rem.\ref{rem: lpl-tagia} (2)). By Lem.\ref{lem: lpl-lav}, to prove
\begin{equation*}\Hom_{\GL_2(F_{\wp})}\big(\Sigma(\alpha,\ul{h}_{\Sigma_{\wp}},\ul{\cL}_{S_n}),\Pi^{\cH^p=\lambda_{\rho}}\big)
    \lra \Hom_{\GL_2(F_{\wp})} \big(\St(\alpha,\ul{h}_{\Sigma_{\wp}}),\Pi^{\cH^p=\lambda_{\rho}}\big)
\end{equation*}
is surjective, it's sufficient to prove the restriction map
\begin{multline}\label{equ: lpl-ehh}
\Hom_{\GL_2(F_{\wp})}\Big(\Sigma(\alpha,\ul{h}_{S_n}',\ul{\cL}_{S_n}),\widetilde{H}^1_{\et}\big(K^p,W(\ul{k}_{S_c},w)\big)_{S_n-\an}^{\cH^p=\lambda_{\rho}}\Big)
   \\ \lra \Hom_{\GL_2(F_{\wp})} \Big(\St(\alpha,\ul{h}_{S_n}'),\widetilde{H}^1_{\et}\big(K^p,W(\ul{k}_{S_c},w)\big)_{S_n-\an}^{\cH^p=\lambda_{\rho}}\Big)
\end{multline}
is surjective.

It's convenient to work with a ``twist" of the eigenvariety $\cE(\ul{k}_{S_c},w)_{\overline{\rho}}$: as in \S \ref{sec: lpl-4.2.2}, one can construct an eigenvariety $\cE$ together with a coherent sheaf $\cM$ from the essentially admissible representation of $T(F_{\wp})$
\begin{equation}\label{equ: lpl-BHtg}J_{B}\big(\widetilde{H}^1_{\et}\big(K^p,W(\ul{k}_{S_c},w)\big)_{\overline{\rho},S_n-\an}^{Z_1}\big),\end{equation}
such that
\begin{equation*}
  \Gamma\big(\cE,\cM\big)\cong J_{B}\big(\widetilde{H}^1_{\et}\big(K^p,W(\ul{k}_{S_c},w)\big)_{\overline{\rho},S_n-\an}^{Z_1}\big)^{\vee};
\end{equation*}
the natural morphism $\cE \ra \widehat{T}$ would factor through $\widehat{T}_{S_n}$ (since (\ref{equ: lpl-BHtg}) is locally $S_n$-analytic); moreover, by the isomorphism (\ref{equ: lpl-NwwJ}), one has a commutative diagram
\begin{equation}\label{equ: lpl-twcd}
  \begin{CD}
    \cE @>(\chi,\lambda)\mapsto (\chi\cdot\chi(\ul{k}_{S_c},w),\lambda)>> \cE(\ul{k}_{S_c},w)_{\overline{\rho}}\\
    @VVV @VVV \\
    \widehat{T}_{S_n} @> \chi\mapsto \chi\cdot\chi(\ul{k}_{S_c},w) >> \widehat{T}(\ul{k}_{S_c},w) \\
    @VVV @VVV \\
    \cW_{1,S_n} @> \chi \mapsto \chi\prod_{\sigma\in S_c}\sigma^{k_{\sigma}-2}>> \cW_1(\ul{k}_J)
  \end{CD}
\end{equation}
where the upper and outer square are Cartesian. Moreover $\cM$ equals the pull-back of $\cM(\ul{k}_{S_c},w)$ via the top horizontal map. Denote by $z_{\rho}'=(\chi_{\rho}',\lambda_{\rho})$ the preimage of $z_{\rho}$ in $\cE$, where $\chi_{\rho}'=\chi(\alpha,\ul{h}'_{S_n})\delta_B$. There exists an admissible open $U$ of $z_{\rho}$ in $\cE(\ul{k}_{S_c},w)_{\overline{\rho}}$ satisfying
\begin{itemize}
  \item any closed point of $U$ is non-$S_n$-critical (Prop.\ref{prop: lpl-ncne}),
  \item $U$ is strictly quasi-Stein (cf. \cite[Lem.6.3.12]{Ding}, see \cite[Def.2.1.17 (iv)]{Em04} for definition),
  \item $\Gamma(U, \cM(\ul{k}_{S_c},w))$ is a torsion free $\co(\cW_1(\ul{k}_{S_c}))$-module (Prop.\ref{prop: lpl-evb}).
\end{itemize}
Take $\cU$ to be the preimage of $U$ in $\cE$, which satisfies hence
\begin{enumerate}
  \item for $z=(\chi,\lambda)\in \cU$, $S\subseteq C(\chi)\cap S_n$, $z_S^c:=(\chi_S^c,\lambda)$ does not lie in $\cE$,
  \item $\cU$ is strictly quasi-Stein,
  \item $\Gamma(\cU,\cM)$ is a torsion free $\co(\cW_{1,S_n})$-module.
\end{enumerate}
The natural restriction map (which has dense image) $\Gamma(\cE,\cM)\ra \Gamma(\cU,\cM)$ induces (by taking the dual)
\begin{equation}\label{equ lpl-rmot}
  \Gamma(\cU,\cM)^{\vee}\hooklongrightarrow \Gamma(\cE,\cM)^{\vee}\cong J_{B}\big(\widetilde{H}^1_{\et}\big(K^p,W(\ul{k}_{S_c},w)\big)_{\overline{\rho},S_n-\an}^{Z_1}\big).
\end{equation}
Note by assumption $\Gamma(\cU,\cM)^{\vee}$ is a locally $S_n$-analytic representation of $T(F_{\wp})$ equipped with a continuous action of $\cH^p$ which commutes with $T(F_{\wp})$. By the adjunction formula in families \cite[Cor.5.3.31]{Ding} \big(note (\ref{equ lpl-rmot}) is balanced by property (1) of $\cU$, see \cite[Lem.6.3.14]{Ding}; $\Gamma(\cU,\cM)^{\vee}$ is allowable since $\cU$ is strictly quasi-Stein, see \cite[Ex.5.3.16]{Ding}\big), (\ref{equ lpl-rmot}) induces a $\GL_2(F_{\wp})\times \cH^p$-invariant morphism
\begin{equation}\label{equ: lpl-aff}
  \big(\Ind_{\overline{B}(F_{\wp})}^{\GL_2(F_{\wp})} \Gamma(\cU,\cM)^{\vee}\otimes_E \delta_B^{-1}\big)^{S_n-\an} \lra \widetilde{H}^1_{\et}\big(K^p,W(\ul{k}_{S_c},w)\big)_{\overline{\rho},S_n-\an}^{Z_1}.
\end{equation}

Let $\tau\in S_n$, and $\cW_{1,S_n}(\ul{k}_{S_n^{\tau}})$ denote the closed rigid subspace of $\cW_{1,S_n}$ parameterizing characters moreover with fixed weights $k_{\sigma}-2$ for $\sigma\in S_n^{\tau}$. Put $\cE_{\tau}:=\cE \times_{\cW_{1,S_n}} \cW_{1,S_n}(\ul{k}_{S_n^{\tau}})$. Note $z_{\rho}'\in \cE_{\tau}$ for all $\tau\in S_n$. Moreover, since $\cE$ is \'etale over $\cW_{1,S_n}$ at $z_{\rho}'$, $\cE_{\tau}$ is \'etale over $\cW_{1,S_n}(\ul{k}_{S_n^{\tau}})$ at $z_{\rho}'$. Let $t_{\tau}: \Spec E[\epsilon]/\epsilon^2\ra \cE_{\tau}$ be a non-zero element in the tangent space of $\cE_{\tau}$ at $z_{\rho}'$, the composition $t_{\tau}: \Spec E[\epsilon]/\epsilon^2 \ra \cE_{\tau}\ra \widehat{T}$ thus gives a locally $S_n$-analytic character $\widetilde{\chi}'_{\rho,\sigma}: T(F_{\wp})\ra (E[\epsilon]/\epsilon^2)^{\times}$ satisfying that $\widetilde{\chi}'_{\rho,\tau}\equiv \chi_{\rho}' \pmod{\epsilon}$, $\widetilde{\chi}'_{\rho,\tau}|_{Z_1}=1$, and $\widetilde{\chi}'_{\rho,\tau}(\chi_{\rho}')^{-1}$ is locally $\tau$-analytic.

Consider $(t_{\tau}^* \cM)^{\vee}$, which is a subrepresentation of $\Gamma(\cU,\cM)^{\vee}$ \big(since the restrcition map $\Gamma(\cU,\cM)\ra t_{\tau}^* \cM$ is surjective\big) of $T(F_{\wp})$ equipped with a continuous action of $\cH^p$. By the second part of Thm.\ref{thm: lpl-eta} \big(note we have a similar result for $(\cE,\cW_{1,S_n})$ thus for $(\cE_{\tau},\cW_{1,S_n}(\ul{k}_{S_n^{\tau}}))$\big), we have
\begin{enumerate}
  \item there exists $r$ such that $(t_{\tau}^* \cM)^{\vee}\cong \big(\widetilde{\chi}_{\rho,\tau}'\big)^{\oplus r}$ as $T(F_{\wp})$-representations,
  \item $(t_{\tau}^* \cM)^{\vee}$ is a generalized $\lambda_{\rho}$-eigenspace.
\end{enumerate}
The map (\ref{equ: lpl-aff}) thus induces
\begin{multline*}
\big(\Ind_{\overline{B}(F_{\wp})}^{\GL_2(F_{\wp})} (t_{\tau}^* \cM)^{\vee}\otimes_E \delta_B^{-1}\big)^{S_n-\an}  \hooklongrightarrow \big(\Ind_{\overline{B}(F_{\wp})}^{\GL_2(F_{\wp})} \Gamma(\cU,\cM)^{\vee}\otimes_E \delta_B^{-1}\big)^{S_n-\an} \\ \lra \widetilde{H}^1_{\et}\big(K^p,W(\ul{k}_{S_c},w)\big)_{\overline{\rho},S_n-\an}^{Z_1}.
\end{multline*}In particular, each vector  not killed by $\epsilon$ in $(t_{\tau}^* \cM)^{\vee}$ induces a morphism
\begin{equation*}
 \big(\Ind_{\overline{B}(F_{\wp})}^{\GL_2(F_{\wp})} \widetilde{\chi}_{\rho,\tau}' \delta_B^{-1}\big)^{S_n-\an} \lra \widetilde{H}^1_{\et}\big(K^p,W(\ul{k}_{S_c},w)\big)_{\overline{\rho},S_n-\an}^{Z_1}[\cH^p=\lambda_{\rho}].
\end{equation*}
Since $\widetilde{\chi}_{\rho,\tau}'$ is an extension of $\chi_{\rho}'=\chi(\rho,\ul{h}'_{S_n})\delta_B$ by itself, one has an exact sequence
\begin{multline*}
  0 \lra \big(\Ind_{\overline{B}(F_{\wp})}^{\GL_2(F_{\wp})} \chi(\rho,\ul{h}'_{S_n})\big)^{S_n-\an} \lra \big(\Ind_{\overline{B}(F_{\wp})}^{\GL_2(F_{\wp})} \widetilde{\chi}_{\rho,\tau}\delta_B^{-1}\big)^{S_n-\an}\\  \xlongrightarrow{s} \big(\Ind_{\overline{B}(F_{\wp})}^{\GL_2(F_{\wp})} \chi(\rho,\ul{h}'_{S_n})\big)^{S_n-\an} \lra 0.
\end{multline*}
Let $\Sigma_{\tau}:=s^{-1}\big(F(\alpha,\ul{h}'_{S_n})\big)/F(\alpha,\ul{h}'_{S_n})$. By the same argument as in \cite[\S 4.3]{Ding3} (see in particular the arguments after \cite[Lem.4.16]{Ding3}), we can prove the restriction map
\begin{equation*}
\Hom_{\GL_2(F_{\wp})}\big(\Sigma_{\tau}, \widetilde{H}^1_{\et}\big(K^p,W(\ul{k}_{S_c},w)\big)_{S_n-\an}^{\cH^p=\lambda_{\rho}}\big) \lra \Hom_{\GL_2(F_{\wp})}\big(\St(\alpha,\ul{h}'_{S_n}), \widetilde{H}^1_{\et}\big(K^p,W(\ul{k}_{S_c},w)\big)_{S_n-\an}^{\cH^p=\lambda_{\rho}}\big)
\end{equation*}
is surjective. However, by Prop.\ref{prop: lpl-cri} below, one has $\Sigma_{\tau}\cong \Sigma(\alpha,\ul{h}'_{S_n},\cL_{\tau})$. Thus for any $\tau\in S_n$, the restriction map
\begin{multline*}
  \Hom_{\GL_2(F_{\wp})}\big(\Sigma(\alpha,\ul{h}'_{S_n},\cL_{\tau}),\widetilde{H}^1_{\et}\big(K^p,W(\ul{k}_{S_c},w)\big)_{S_n-\an}^{\cH^p=\lambda_{\rho}}\big) \\ \lra \Hom_{\GL_2(F_{\wp})}\big(\St(\alpha,\ul{h}'_{S_n}), \widetilde{H}^1_{\et}\big(K^p,W(\ul{k}_{S_c},w)\big)_{S_n-\an}^{\cH^p=\lambda_{\rho}}\big)
\end{multline*}
is surjective. From which, together with Rem.\ref{rem: lpl-tagia} (4), we see (\ref{equ: lpl-ehh}) is surjective. This concludes the proof of Thm.\ref{thm: lpl-giao} (2) (assuming Prop.\ref{prop: lpl-cri}).
\end{proof}
\begin{proposition}\label{prop: lpl-cri}
  Keep the notation as in the proof of Thm.\ref{thm: lpl-giao}, there exists a locally $\tau$-analytic character $\psi_{\tau}$ such that
  \begin{equation*}
    \widetilde{\chi}_{\rho,\tau}'(\chi_{\rho}')^{-1}\cong  \begin{pmatrix}
      1 & \log_{\tau,-\cL_{\tau}}(ad^{-1}) + \psi_{\tau}(ad) \\
      0 & 1
    \end{pmatrix}
  \end{equation*}
  as ($2$-dimensional) representations of $T(F_{\wp})$. Consequently (by Rem.\ref{rem: lpl-tagia} (3)), $\Sigma_{\tau}\cong \Sigma(\alpha,\ul{h}'_{S_n},\cL_{\tau})$.
\end{proposition}
%This proposition, combined with Rem.\ref{rem: lpl-tagia} (3), implies:
%\begin{corollary}\label{cor: lpl-kcor}
%  $\Sigma_{\tau}\cong \Sigma(\alpha,\ul{h}'_{S_n},\cL_{\tau})$.
%\end{corollary}
The rest of the paper is devoted to the proof of Prop.\ref{prop: lpl-cri}. Note the image $\cE_{\tau}'$ of $\cE_{\tau}$ in $\cE(\ul{k}_{S_c},w)_{\overline{\rho}}$  is an one-dimensional rigid space containing $\cE(\ul{k}_{\Sigma_{\wp}^{\tau}},w)_{\overline{\rho}}$ as closed subspace. Since both $\cE_{\tau}'$ and $\cE(\ul{k}_{\Sigma_{\wp}^{\tau}},w)_{\overline{\rho}}$  are \'etale over $\cW_1(\ul{k}_{\Sigma_{\wp}^{\tau}})$ at $z_{\rho}$, and have the same residue field $E$, we see they are locally isomorphic at $z_{\rho}$. In particular, the composition $\Spec E[\epsilon]/\epsilon^2 \xrightarrow{t_{\tau}} \cE_{\tau}\ra \cE(\ul{k}_{S_c},w)_{\overline{\rho}}$ gives a non-zero element in the tangent space of $\cE(\ul{k}_{\Sigma_{\wp}^{\tau}},w)_{\overline{\rho}}$ at $z_{\rho}$, still denoted by $t_{\tau}: \Spec E[\epsilon]/\epsilon^2 \ra \cE(\ul{k}_{\Sigma_{\wp}^{\tau}},w)_{\overline{\rho}}$, moreover it's straightforward to see (e.g. by (\ref{equ: lpl-twcd})) the character of $T(F_{\wp})$ induced by this map is given by $\widetilde{\chi}_{\rho,\tau}:=\widetilde{\chi}_{\rho,\tau}'\chi(\ul{k}_{S_c},w)$. Note $\widetilde{\chi}_{\rho,\tau}\chi_{\rho}^{-1}=\widetilde{\chi}_{\rho,\tau}'(\chi_{\rho}')^{-1}$. %We use global triangulation theory to prove  $\widetilde{\chi}_{\rho,\tau}\chi_{\rho}^{-1}$ with
Since $\widetilde{\chi}_{\rho,\tau}\chi_{\rho}^{-1}$ is locally $\tau$-analytic, there exist $\gamma$, $\eta\in E$, $\mu\in E^{\times}$ such that (cf. \S \ref{sec: lpl-1.3.1})
  \begin{equation*}\widetilde{\chi}\chi_{\rho}^{-1}=(1+\gamma \epsilon \psi_{\ur}+\mu\epsilon\psi_{\tau,p}) \otimes(1+\eta \epsilon \psi_{\ur} -\mu \epsilon \psi_{\tau,p}).
  \end{equation*}
It's sufficient to prove
\begin{equation}\label{equ: lpl-ums}
  \gamma-\eta=-2\cL_{\tau}\mu.
\end{equation}
Indeed, if (\ref{equ: lpl-ums}) holds, we get
\begin{multline*}\label{equ: lpl-gg+}
  \widetilde{\chi}_{\rho,\tau}\chi_{\rho}^{-1}\cong \big(1+ \mu\epsilon(-\cL_{\tau}\psi_{\ur}+\psi_{\tau,p})+\frac{(\gamma+\eta)\epsilon}{2} \psi_{\ur}\big) \otimes \big(1- \mu\epsilon(-\cL_{\tau}\psi_{\ur}+\psi_{\tau,p})+\frac{(\gamma+\eta)\epsilon}{2} \psi_{\ur}\big) \\
  \cong (1+\log_{\tau,-\cL_{\tau}}\epsilon + \psi_{\tau}\epsilon)\otimes (1-\log_{\tau,-\cL_\tau}\epsilon +\psi_{\tau}\epsilon),
\end{multline*}
with $\psi_{\tau}=\frac{\gamma+\eta}{2\mu} \psi_{\ur}$, from which Prop.\ref{prop: lpl-cri} follows.

We show (\ref{equ: lpl-ums}). Let $U$ be an affinoid neighborhood of $z_{\rho}$ in $\cE(\ul{k}_{\emptyset},w)_{\overline{\rho}}$ small enough such that Prop.\ref{prop: lpl-uoc} applies, we have thus a continuous representation $\rho_U: \Gal_F\ra \GL_2(\co(U_{\red}))$.

\textbf{Non-critical Case:} Suppose $S_n=\Sigma_{\wp}$, i.e. $z$ is non-$\Sigma_{\wp}$-critical. By Prop.\ref{prop: lpl-ncne}, shrinking $U$, we can assume any closed point in $U$ is non-$\Sigma_{\wp}$-critical. Let $U_{\tau}$ be the preimage of $U$ in $\cE(\ul{k}_{\Sigma_{\wp}^{\tau}},w)_{\overline{\rho}}$ \big(via the natural closed embedding $\cE(\ul{k}_{\Sigma_{\wp}^{\tau}},w)_{\overline{\rho}}\hookrightarrow \cE(\ul{k}_{\emptyset},w)_{\overline{\rho}}$\big), since $U_{\tau}$ is \'etale over $\cW_1(\ul{k}_{\Sigma_{\wp}^{\tau}})$ at $z_{\rho}$, shrinking $U_{\tau}$, we can assume $U_{\tau}$ is a smooth curve. Let $\rho_{U_{\tau}}:\Gal_F \ra \GL_2(\co(U_{\tau}))$ be the representation induced by $\rho_U$, $\chi_{U_{\tau}}: T(L)\ra \co(U_{\tau})^{\times}$ be the character induced by the natural morphism $U_{\tau}\ra \widehat{T}(\ul{k}_{\Sigma_{\wp}^{\tau}},w)$.  Applying \cite[Thm.6.3.9]{KPX} to $D_{\rig}(\rho_{U_{\tau},\wp})$ with $\rho_{U_\tau,\wp}:=\rho_{U_{\tau}}|_{\Gal_{F_{\wp}}}$ \big(see Thm.\ref{prop: lpl-gpa}, note $\Sigma_{z'}=\emptyset$ for all $z'\in U_{\tau}$ by the assumption on $U$\big), we get an exact sequence
\begin{equation*}
  0 \ra \cR_{\co(U_{\tau})}\big(\unr(q) \chi_{U_{\tau},1}\big) \ra D_{\rig}(\rho_{U_{\tau},\wp}) \ra \cR_{\co(U_{\tau})}\big(\chi_{U_{\tau},2}\prod_{\sigma\in \Sigma_{\wp}}\sigma^{-1}\big) \ra 0,
\end{equation*}which induces \big(where $\widetilde{\rho}_{\tau,\wp}:=t_{\tau}^* \rho_{U_{\tau}}|_{\wp}:\Gal_{F_{\wp}} \ra \GL_2(E[\epsilon]/\epsilon^2)$\big)
\begin{equation*}
  0 \ra \cR_{E[\epsilon]/\epsilon^2}\big(\unr(q) \widetilde{\chi}_{\rho,\tau,1}) \ra D_{\rig}(\widetilde{\rho}_{\tau,\wp}) \ra \cR_{E[\epsilon]/\epsilon^2}\big(\widetilde{\chi}_{\rho,\tau,2}\prod_{\sigma\in \Sigma_{\wp}}\sigma^{-1}\big) \ra 0.
\end{equation*}
Thus, (\ref{equ: lpl-ums}) follows from Thm.\ref{thm: lpl-fee}.

\textbf{Critical case:} Assume henceforth $S_c\neq \emptyset$.  We shrink $U$ such that the Prop.\ref{prop: lpl-zccs} applies, so $Z_{U,\sigma}$ (if non-empty) is a Zariski-closed subset in $U$ for any $\sigma\in \Sigma_{\wp}$. We know $z\in Z_{U,\sigma}$ if and only if $\sigma\in S_c$. By shrinking $U$ (as a neighborhood of $z$), one can assume $Z_{U,\sigma}=\emptyset$ for $\sigma\in S_n$. Let $\tau\in S_n$, $U_{\tau}$ be the preimage of $U$ in $\cE(\ul{k}_{\Sigma_{\wp}^{\tau}},w)_{\overline{\rho}}$, and shrink $U$ such that $U_{\tau}$ is a smooth curve. Let $Z_{U_{\tau},\sigma}$ the preimage of $Z_{U,\sigma}$ in $U_{\tau}$, which is a non-empty Zariski-closed subset for $\sigma\in S_c$, whose dimension is either $0$ or $1$ locally at $z$.  Denote by $S_0$ (resp. $S_1$) the subset of $S_c$ of embeddings $\sigma$ such that  $Z_{U_\tau,\sigma}$ is of dimension $0$ (resp. of dimension $1$) locally at $z_{\tau}$. By shrinking $U$ (and thus $U_{\tau}$, note $U_{\tau}$ is smooth), one can assume $Z_{U_\tau,\sigma}=\{z_{\tau}\}$ for $\sigma\in S_0$ and $Z_{U_\tau,\sigma} =U_{\tau}(\overline{E})$ for $\sigma\in S_1$. We define $\rho_{U_{\tau},\wp}$, $\chi_{U_{\tau}}$, $\widetilde{\rho}_{\tau,\wp}$  the same way as in the non-critical case.

\textbf{Critical case (1):} Suppose $S_0=\emptyset$. In this case, for any $z\in U_{\tau}$, $\Sigma_{z}=S_c$. By applying \cite[Thm.6.3.9]{KPX} to $D_{\rig}(\rho_{U_{\tau},{\wp}})$, we get
\begin{equation*}
  0 \ra \cR_{\co(U_{\tau})}\big(\unr(q) \chi_{U_{\tau},1}\prod_{\sigma\in S_c}\sigma^{1-k_{\sigma}}\big) \ra D_{\rig}(\rho_{U_{\tau},\wp}) \ra \cR_{\co(U_{\tau})}\big(\chi_{U_{\tau},2}\prod_{\sigma\in \Sigma_{\wp}}\sigma^{-1}\prod_{\sigma\in S_c} \sigma^{k_{\sigma}-1}\big) \ra 0,
\end{equation*}which induces
\begin{equation*}
  0 \ra \cR_{E[\epsilon]/\epsilon^2}\big(\unr(q) \widetilde{\chi}_{\rho,\tau,1}\prod_{\sigma\in S_c}\sigma^{1-k_{\sigma}}) \ra D_{\rig}(\widetilde{\rho}_{\tau,\wp}) \ra \cR_{E[\epsilon]/\epsilon^2}\big(\widetilde{\chi}_{\rho,\tau,2}\prod_{\sigma\in \Sigma_{\wp}}\sigma^{-1}\prod_{\sigma\in S_c} \sigma^{k_{\sigma}-1}\big) \ra 0.
\end{equation*}
On the other hand, by Prop.\ref{prop: lpl-uoc}, $\widetilde{\rho}_{\tau,\wp}$ is $\Sigma_{\wp}^{\tau}$-de Rham. We can hence apply Thm.\ref{thm: lpl-fee}, and (\ref{equ: lpl-ums}) follows.

\textbf{Critical case (2):}  Suppose $S_0 \neq \emptyset$. By assumption, for $z\in U_{\tau}(\overline{E})$, $z\neq z_{\rho}$, $\Sigma_{z}=S_1\subsetneq S_c=S_0\cup S_1$. By \cite[Thm.6.3.9]{KPX} (see in particular \cite[(6.3.14.1)]{KPX}) applied to $D_{\rig}(\rho_{U_{\tau},\wp})$, one gets an exact sequence
\begin{equation}\label{equ: lpl-ugr}
  0 \ra \cR_{\co(U_{\tau})}\big(\unr(q)\chi_{U_{\tau},1}\prod_{\sigma\in S_1}\sigma^{1-k_{\sigma}}\big) \ra D_{\rig}(\rho_{U,\wp}) \ra \cR_{\co(U_{\tau})}\big(\chi_{U_{\tau},2}\prod_{\sigma\in \Sigma_{\wp}}\sigma^{-1}\prod_{\sigma\in S_1}\sigma^{k_{\sigma}-1}\big) \ra Q\ra 0
\end{equation}
where $Q$ is a finitely generated $\cR_{\co(U_{\tau})}$-module killed by certain powers of $t$ \big($\in \cR_E$\big) and is supported at $z_{\rho}$. Tensoring (\ref{equ: lpl-ugr}) with $E[\epsilon]/\epsilon^2$ via $t_{\tau}$, one gets exact sequences (see \cite[Ex.6.3.14]{KPX})%, where the second exact sequence follows from the fact $\Tor^2_{\co(U_{\tau})}(Q,E[\epsilon]/\epsilon^2)=0$\big)
\begin{equation}\label{equ: lpl-gtt}
D_{\rig}(\widetilde{\rho}_{\tau,\wp}) \xlongrightarrow{f} \cR_{E[\epsilon]/\epsilon^2}\big(\widetilde{\chi}_{\rho,\tau,2}\prod_{\sigma\in \Sigma_{\wp}} \sigma^{-1}\prod_{\sigma\in S_1} \sigma^{k_{\sigma}-1}\big) \lra Q\otimes_{\co(U_{\tau}),t_{\tau}}E[\epsilon]/\epsilon^2 \lra 0,
\end{equation}
\begin{equation*}
  0 \lra \cR_{E[\epsilon]/\epsilon^2}\big(\unr(q)\widetilde{\chi}_{\rho,\tau,1}\prod_{\sigma\in S_1}\sigma^{1-k_{\sigma}}\big) \lra \Ker(f).
\end{equation*}
For simplicity, put
\begin{eqnarray*}
  \widetilde{\delta}=\widetilde{\delta}_1\otimes \widetilde{\delta}_2&:=& \big(\unr(q)\widetilde{\chi}_{\rho,\tau,1}\prod_{\sigma\in S_1}\sigma^{1-k_{\sigma}}\big) \otimes \big(\widetilde{\chi}_{\rho,\tau,2}\prod_{\sigma\in \Sigma_{\wp}} \sigma^{-1}\prod_{\sigma\in S_1} \sigma^{k_{\sigma}-1}\big),\\
  \delta=\delta_1\otimes \delta_2 &:=& \big(\unr(q)\chi_{\rho,1}\prod_{\sigma\in S_1}\sigma^{1-k_{\sigma}}\big) \otimes \big(\chi_{\rho,2}\prod_{\sigma\in \Sigma_{\wp}} \sigma^{-1}\prod_{\sigma\in S_1} \sigma^{k_{\sigma}-1}\big), \\
   \widetilde{\delta}'=\widetilde{\delta}'_1\otimes \widetilde{\delta}'_2&:=& \big(\unr(q)\widetilde{\chi}_{\rho,\tau,1}\prod_{\sigma\in S_c}\sigma^{1-k_{\sigma}}\big) \otimes \big(\widetilde{\chi}_{\rho,\tau,2}\prod_{\sigma\in \Sigma_{\wp}} \sigma^{-1}\prod_{\sigma\in S_c} \sigma^{k_{\sigma}-1}\big),\\
  \delta'=\delta_1'\otimes \delta_2' &:=& \big(\unr(q)\chi_{\rho,1}\prod_{\sigma\in S_c}\sigma^{1-k_{\sigma}}\big) \otimes \big(\chi_{\rho,2}\prod_{\sigma\in \Sigma_{\wp}} \sigma^{-1}\prod_{\sigma\in S_c} \sigma^{k_{\sigma}-1}\big),
\end{eqnarray*}
and note $\delta'$ is the trianguline parameter of $\rho_{\wp}$.

We see $\Ker(f)$ and $\Ima(f)$ (cf. (\ref{equ: lpl-gtt})) are $(\varphi,\Gamma)$-modules over $\cR_{E[\epsilon]/\epsilon^2}$ \big(i.e. $(\varphi,\Gamma)$-modules over $\cR_E$ equipped moreover an $E[\epsilon]/\epsilon^2$-action commuting with $\cR_E$, note that such modules may not be free over $\cR_{E[\epsilon]/\epsilon^2}$\big).
Denote by $f_0$ the map $D_{\rig}(\rho_{\wp}) \ra \cR_E\big(\delta_2\big)$ induced by (\ref{equ: lpl-ugr}) via the pull-back $z_{\rho}^*$, one has a commutative diagram (of $(\varphi,\Gamma)$-modules over $\cR_E$)
\begin{equation*}\begin{CD}
0 @>>> D_{\rig}(\rho_{\wp}) @> \epsilon >> D_{\rig}(\widetilde{\rho}_{\wp}) @>>> D_{\rig}(\rho_{\wp}) @>>> 0 \\
@. @V f_0 VV @V f VV @V f_0 VV @. \\
0 @>>> \cR_E(\delta_{2}) @> \epsilon >> \cR_{E[\epsilon]/\epsilon^2}(\widetilde{\delta}_2) @>>> \cR_E(\delta_{2}) @>>> 0\end{CD}
\end{equation*}
which induces thus a long exact sequence
\begin{equation*}
  0 \ra \Ker(f_0) \xrightarrow{\epsilon} \Ker(f) \xrightarrow{s} \Ker(f_0) \ra Q\otimes_{\co(U_{\tau}), z_{\rho}} E \xrightarrow{\epsilon} Q\otimes_{\co(U_{\tau}),t_{\tau}} E[\epsilon]/\epsilon^2 \ra Q \otimes_{\co(U_{\tau}),z_{\rho}} E \ra 0.
\end{equation*}
By discussions in \cite[Ex.6.3.14]{KPX}, one has (where we refer to \cite[Not.6.2.7]{KPX} for the $t_{\sigma}$'s)
\begin{equation*}
  \Ker(f_0)\cong \cR_E(\delta_{1}'),\ \Ima(f_0)\cong \cR_E(\delta_{2}'), \ Q\otimes_{\co(U_{\tau}),z_{\rho}} E\cong \cR_E(\delta_{1})/\Big(\prod_{\sigma\in S_0} t_{\sigma}^{k_{\sigma}-1}\Big),
\end{equation*}
thus there exist $r_{\sigma}\in \Z$, $0 \leq r_{\sigma} \leq k_{\sigma}-1$ for all $\sigma\in S_0$ such that $\Ima(s)=\cR_E(\delta_{1}'')$ where $\delta_{1}'':=\delta_{1}'\prod_{\sigma\in S_0} \sigma^{r_{\sigma}}$. However, since $\Ker(f)$ is a saturated sub-$(\varphi,\Gamma)$-module of $D_{\rig}(\widetilde{\rho}_{\tau,\wp})$, and the latter has Sen weight of the form $(-\frac{k_{\sigma}+w}{2}+a_{\sigma}\epsilon, \frac{k_{\sigma}-w-2}{2}+b_{\sigma}\epsilon)_{\sigma \in \Sigma_{\wp}}$, we see $r_{\sigma}=0$ or $k_{\sigma}-1$ for $\sigma\in S_0$.

 One has a natural isomorphism \big(translating these in terms of $E$-$B$-pairs, one can check this isomorphism by the same argument as in the proof of Lem.\ref{lem: lpl-tng}\big) \begin{equation*}\Ext^1(\cR_E(\delta_{1}'),\cR_E(\delta_{1}'))\xlongrightarrow{\sim}\Ext^1\big(\cR_E(\delta_{1}''), \cR_E(\delta_{1}')\big).\end{equation*}
We claim $[\Ker(f)]$ equals (up to scalars) the image of $[\cR_{E[\epsilon]/\epsilon^2}(\widetilde{\delta}_1')]$: Indeed one has isomorphisms
\begin{multline}\label{equ: lpl-gig}
\Ext^1(\cR_E(\delta_{1}'),\cR_E(\delta_{1}'))\xlongrightarrow{\sim}\Ext^1\big(\cR_E(\delta_{1}''), \cR_E(\delta_{1}')\big)\\ \xlongrightarrow{\sim} \Ext^1\big(\cR_E(\delta_{1}), \cR_E(\delta_{1}')\big) \cong \Ext^1\big(\cR_E(\delta_{1}),\cR_E(\delta_{1})\big).
\end{multline}
The composition $i$ in (\ref{equ: lpl-gig}) actually sends $[\cR_{E[\epsilon]/\epsilon^2}(\widetilde{\delta}_{1}')]$ to $[\cR_{E[\epsilon]/\epsilon^2}(\widetilde{\delta}_{1})]$ (up to scalars), since both $i([\cR_{E[\epsilon]/\epsilon^2}(\widetilde{\delta}_{1}')])$ and  $[\cR_{E[\epsilon]/\epsilon^2}(\widetilde{\delta}_{1})]$ fit ``$*$" in the following commutative diagram (with the maps on the left and right sides being the natural injections)
\begin{equation*}
  \begin{CD}
    0 @>>> \cR_E(\delta_{1}) @>>> * @>>> \cR_E(\delta_{1}) @>>> 0 \\
    @. @VVV @VVV @VVV @. \\
     0 @>>> \cR_E(\delta_{1}') @>>> \cR_{E[\epsilon]/\epsilon^2}(\widetilde{\delta}_{1}') @>>> \cR_E(\delta_{1}') @>>> 0
  \end{CD};
\end{equation*}on the other hand, since $\cR_{E[\epsilon]/\epsilon^2}(\widetilde{\delta}_{1})\hookrightarrow \Ker(f)$, one sees the composition of the last two morphisms in (\ref{equ: lpl-gig}) sends  $[\Ker(f)]$ to $[\cR_{E[\epsilon]/\epsilon^2}(\widetilde{\delta}_{1})]$ (up to scalars), the claim follows.

Similarly,  $\Ima(f)$ lies in an exact sequence of $(\varphi,\Gamma)$-modules over $\cR_E$:
\begin{equation*}
  0 \ra \cR_E(\delta_{2}'') \xrightarrow{\epsilon} \Ima(f) \ra \cR_E(\delta_{2}') \ra 0
\end{equation*}
with $\delta_{2}''=\delta_{2}'\prod_{\sigma\in S_0} \sigma^{-r_{\sigma}}$, and that the natural isomorphism
\begin{equation*}
  \Ext^1(\cR_E(\delta_{2}'),\cR_E(\delta_{2}')) \xlongrightarrow{\sim} \Ext^1(\cR_E(\delta_{2}'),\cR_E(\delta_{2}''))
\end{equation*}
sends $[\cR_{E[\epsilon]/\epsilon^2}(\widetilde{\delta}'_2)]$ to $\Ima(f)$.

\textbf{Claim:} There exists a $(\varphi,\Gamma)$-module $D$ free of rank $2$ over $\cR_{E[\epsilon]/\epsilon^2}$ such that

(1) $D$ lies in an exact sequence of $(\varphi,\Gamma)$-modules over $\cR_{E[\epsilon]/\epsilon^2}$:
\begin{equation*}
  0 \ra \cR_{E[\epsilon]/\epsilon^2}(\widetilde{\delta}_1') \ra D \ra \cR_{E[\epsilon]/\epsilon^2}(\widetilde{\delta}_2') \ra 0;
\end{equation*}

(2) $D\equiv D_{\rig}(\rho_{\wp}) \pmod{\epsilon}$;

(3) $D$ is $S_c$-de Rham.

Assuming the claim, since $\widetilde{\delta}_1'(\widetilde{\delta}_2')^{-1}=\delta_{1}' (\delta_{2}')^{-1}(1+(\gamma-\eta)\epsilon\psi_{\ur}+2\mu \epsilon \psi_{\tau,p})$,  one can deduce again from Thm.\ref{thm: lpl-fee} that $\gamma-\eta=-2\cL_{\tau}\mu$ (\ref{equ: lpl-ums}). In the rest of this section, we ``modify'' $D_{\rig}(\widetilde{\rho}_{\tau,\wp})$ to prove the claim.

%The natural morphism of $(\varphi,\Gamma)$-modules over $\cR_{E[\epsilon]/\epsilon^2}$:  $\Ker(f) \hookrightarrow \cR_{E[\epsilon]/\epsilon^2}\big(\widetilde{\delta}_1'\big)$ induces a morphism
%\begin{equation*}\Ext^1(\Ima(f), \Ker(f)) \lra \Ext^1\big(\Ima(f), \cR_{E[\epsilon]/\epsilon^2}(\widetilde{\delta}_1')\big)\end{equation*} \big(here $\Ext^1$ denotes the group of extensions of $(\varphi,\Gamma)$-modules over $\cR_{E[\epsilon]/\epsilon^2}$\big). Denote by $D'$ the image of $D_{\rig}(\widetilde{\rho}_{\wp})$ via this morphism. The natural morphism of $(\varphi,\Gamma)$-modules over $\cR_{E[\epsilon]/\epsilon^2}$: $\cR_{E[\epsilon]/\epsilon^2}\big(\widetilde{\delta}_2'\big)\hookrightarrow \Ima(f)$ induces a morphism
%\begin{equation*}\Ext^1\big(\Ima(f), \cR_{E[\epsilon]/\epsilon^2}(\widetilde{\delta}_1')\big) \lra \Ext^1(\cR_{E[\epsilon]/\epsilon^2}\big(\widetilde{\delta}_2'),\cR_{E[\epsilon]/\epsilon^2}(\widetilde{\delta}_1')\big),\end{equation*}
 %let $D$ be the image of $D'$ via this morphism. One needs only to check the property (2) in the claim.

The natural morphism of $(\varphi,\Gamma)$-modules over $\cR_{E[\epsilon]/\epsilon^2}$:  $\cR_{E[\epsilon]/\epsilon^2}\big(\widetilde{\delta}_2'\big)\hookrightarrow \Ima(f)$ induces a morphism
\begin{equation*}\Ext^1(\Ima(f), \Ker(f)) \lra \Ext^1\big(\cR_{E[\epsilon]/\epsilon^2}\big(\widetilde{\delta}_2'\big), \Ker(f)\big)\end{equation*}
\big(here $\Ext^1$ denotes the group of extensions of $(\varphi,\Gamma)$-modules over $\cR_{E[\epsilon]/\epsilon^2}$\big). Denote by $D'$ the image of $D_{\rig}(\widetilde{\rho}_{\tau,\wp})$ via this morphism. In fact, $D'$ is just the preimage of $\cR_{E[\epsilon]/\epsilon^2}(\widetilde{\delta}_2')\subset \Ima(f)$ via the natural projection $D_{\rig}(\widetilde{\rho}_{\tau,\wp})\twoheadrightarrow \Ima(f)$. The natural morphism of $(\varphi,\Gamma)$-modules over $\cR_{E[\epsilon]/\epsilon^2}$: $\Ker(f) \hookrightarrow \cR_{E[\epsilon]/\epsilon^2}(\widetilde{\delta}_1')$ induces a morphism
\begin{equation*}\Ext^1\big(\cR_{E[\epsilon]/\epsilon^2}(\widetilde{\delta}_2'), \Ker(f)\big) \lra \Ext^1\big(\cR_{E[\epsilon]/\epsilon^2}(\widetilde{\delta}_2'),\cR_{E[\epsilon]/\epsilon^2}(\widetilde{\delta}_1')\big),\end{equation*} let $D$ be the image of $D'$ via this morphism. We check $D$ satisfies the properties (2) and (3) in the claim.

We have a commutative diagram
\begin{equation*}
  \begin{CD}
    0 @>>> \Ker(f) @>>> D_{\rig}(\widetilde{\rho}_{\tau,\wp}) @>>> \Ima(f) @>>> 0 \\
    @. @V\epsilon VV @V \epsilon VV @V \epsilon VV @. \\
    0 @>>> \Ker(f) @>>> D_{\rig}(\widetilde{\rho}_{\tau,\wp}) @>>> \Ima(f) @>>> 0
  \end{CD}
\end{equation*}
which induces a long exact sequence
\begin{multline}\label{equ: lpl-z1a}
  0 \ra \cR_{E}(\delta_{1}') \ra D_{\rig}(\rho) \xrightarrow{r} \cR(\delta_{2}'') \ra \cR_E(\delta_{1}'') \oplus \Big(\cR_E(\delta_{2}'')/\prod_{\sigma\in S_0} t_\sigma^{r_{\sigma}}\Big)\\ \ra D_{\rig}(\rho) \ra \cR_E(\delta_{2}')\oplus \Big(\cR_E(\delta_{2}'')/\prod_{\sigma\in S_0} t_\sigma^{r_{\sigma}}\Big)\ra 0.
\end{multline}
For a $(\varphi,\Gamma)$-module $D''$ over $\cR_{E[\epsilon]/\epsilon^2}$, denote by $D''[\epsilon]$ the kernel of $\epsilon$ which is a saturated $(\varphi,\Gamma)$-submodule (over $\cR_E$) of $D''$. One sees the natural morphism $D'\hookrightarrow D_{\rig}(\widetilde{\rho}_{\tau,\wp})$ induces an isomorphism $D'[\epsilon]\cong D_{\rig}(\widetilde{\rho}_{\tau,\wp})[\epsilon]\cong D_{\rig}(\rho_{\wp})$. Indeed, one gets an injection $D'[\epsilon]\hookrightarrow D_{\rig}(\widetilde{\rho}_{\tau,\wp})[\epsilon]$, on the other hand, by (\ref{equ: lpl-z1a}), the image of $D_{\rig}(\widetilde{\rho}_{\tau,\wp})[\epsilon] \hookrightarrow D_{\rig}(\widetilde{\rho}_{\tau,\wp}) \twoheadrightarrow \Ima(f)$, which equals $\Ima(r)$, is contained in $\cR_{E[\epsilon]/\epsilon^2}(\widetilde{\delta}_2')$, thus $D_{\rig}(\widetilde{\rho}_{\tau,\wp})[\epsilon]\subseteq D'$ so  $D_{\rig}(\widetilde{\rho}_{\tau,\wp})[\epsilon]\subseteq D'[\epsilon]$, from which one gets the isomorphism. The $(\varphi,\Gamma)$-module $D'/D'[\epsilon]$ sits in an exact sequence
\begin{equation*}
  0 \ra \cR_E(\delta_{1}'') \ra D'/D'[\epsilon] \ra \cR_E(\delta_{2}') \ra 0
\end{equation*}
and is a submodule of $D_{\rig}(\widetilde{\rho}_{\tau,\wp})/D_{\rig}(\widetilde{\rho}_{\tau,\wp})[\epsilon]\cong D_{\rig}(\rho_{\wp})$. By the construction of $D$, one gets a natural morphism $D'\ra D$ which induces an isomorphism $D_{\rig}(\rho_{\wp})\cong D'[\epsilon]\xrightarrow{\sim} D[\epsilon]\cong D_{\rig}(\rho_{\wp})$, and thus an injection $D'/D'[\epsilon] \hookrightarrow D/D[\epsilon]$. One gets commutative diagrams
\begin{equation}\label{equ: lpl-eno}
  \begin{CD}
    0 @>>> \cR_E(\delta_{1}'') @>>> D'/D'[\epsilon]@>>> \cR_E(\delta_{2}') @>>> 0 \\
    @. @VVV @VVV @| @. \\
    0 @>>>\cR_E(\delta_{1}') @>>> D_* @>>>  \cR_E(\delta_{2}') @>>> 0
  \end{CD}
\end{equation}
for $D_*\in \{D/D[\epsilon], D_{\rig}(\rho_{\wp})\}$ (for $D/D[\epsilon]$, this follows from the construction of $D$; for $D_{\rig}(\rho_{\wp})$, this follows from the construction of $D'$ discussed as above). So $D/D[\epsilon]\cong D_{\rig}(\rho_{\wp})$ \big(which are both equal to the image of $D'/D'[\epsilon]$ via the natural morphism $\Ext^{1}(\cR_E(\delta_2'),\cR_E(\delta_1''))\ra \Ext^1(\cR_E(\delta_2'),\cR_E(\delta_1'))$\big), the property (2) follows.

To show $D$ is $S_c$-de Rham, one needs only to prove $D'$ is $S_c$-de Rham since $D'$ is a $(\varphi,\Gamma)$-submodule of $D$ with the same rank. Since $D'$ is a $(\varphi,\Gamma)$-submodule of $D_{\rig}(\widetilde{\rho}_{\tau,\wp})$, by the equivalence of categories of $B$-pairs and $(\varphi,\Gamma)$-modules (\cite[Thm.2.2.7]{Ber08}), one gets an injection $W(D') \hookrightarrow W(D_{\rig}(\widetilde{\rho}_{\tau,\wp}))$ of $E$-$B$-pairs where $W(D'')$ denotes the associated $B$-pairs for a $(\varphi,\Gamma)$-modules $D''$. Since $D'$ and $D_{\rig}(\widetilde{\rho}_{\tau,\wp})$ are both of rank $4$ (over $\cR_E$), one sees $W(D')_{\dR} \xrightarrow{\sim} W(D_{\rig}(\widetilde{\rho}_{\tau,\wp}))_{\dR}$. Since $D_{\rig}(\widetilde{\rho}_{\tau,\wp})$ is $S_c$-de Rham, so is  $D'$. This finishes the proof of the claim and thus (\ref{equ: lpl-ums}) in $S_0\neq \emptyset$-case.

\appendix
\section{Partially de Rham trianguline representations}\label{sec: lpl-app}In this appendix, we study some partially de Rham triangulable $E$-$B$-pairs, and show that partial non-criticalness implies partial de Rhamness for triangulable $E$-$B$-pairs. As an application, we get a partial de Rhamness result for finite slope overconvergent Hilbert modular forms. %Recall
%\begin{definition}[$\text{cf.\cite{Colm2}, \cite{Na}}$](1) An $E$-$B$-pair $W$ is called triangulable if it's an successive extension of rank $1$ $E$-$B$-pairs, i.e. $W$ admits an increasing filtration of $E$-$B$-sub-pairs
%such that $W_i/W_{i-1}$ is an $E$-$B$-pair of rank $1$ for $1\leq i \leq r$. The filtration (\ref{equ: pdeR-wqe}) is called a triangulation of $W$.
%
%
% (2) A finite dimensional continuous $\Gal_{F_{\wp}}$-representation $V$ over $E$ is called trianguline if the associated $E$-$B$-pair $W(V)$ is triangulable.
%\end{definition}

Let $F_{\wp}$ be a finite extension of $\Q_p$, $\Sigma_{\wp}$ the set of embeddings of $F_{\wp}$ in $\overline{\Q_p}$, $\Gal_{F_{\wp}}:=\Gal(\overline{\Q_p}/F_{\wp})$, $E$ a finite extension of $\Q_p$ sufficiently large containing all the embeddings of $F_{\wp}$ in $\overline{\Q_p}$. Let $\chi$ be a continuous character of $F_{\wp}^{\times}$ over $E$, recall that we have defined the weights $(\wt(\chi)_{\sigma})_{\sigma\in \Sigma_{\wp}}\in E^{|d|}$ of $\chi$ (cf. \S \ref{sec: lpl-2}); in fact, $(-\wt(\chi)_{\sigma})_{\sigma\in \Sigma_{\wp}}$ are equal to the \emph{generalized Hodge-Tate weights} of the associated $E$-$B$-pair $B_E(\chi)$ (cf. \cite[Def. 1.47]{Na}).
\begin{lemma}\label{lem: pdeR-iasr}Let $\chi$ be a continuous character of $F_{\wp}^{\times}$ over $E$, for $\sigma\in \Sigma_{\wp}$, $B_E(\chi)$ is $\sigma$-de Rham if and only if $\wt(\chi)_{\sigma}\in \Z$.
\end{lemma}
\begin{proof}
 The ``only if" part is clear. Suppose now $\wt(\chi)_{\sigma}\in \Z$, by multiplying $\chi$ by $\sigma^{-\wt(\chi)_{\sigma}}$  and then an unramified character of $F_{\wp}^{\times}$, one can assume that $\chi$ corresponds to a Galois character $\chi: \Gal_{F_{\wp}} \ra E^{\times}$ and $\wt(\chi)_{\sigma}=0$. In this case, by Sen's theory, one has $\bC_{p,\sigma} \otimes_E \chi\cong \bC_{p,\sigma}$ as $\Gal_{F_{\wp}}$-modules (since $\chi$ is of Hodge-Tate weight $0$ at $\sigma$). Consider the exact sequence
 \begin{equation*}
   0 \ra (t B_{\dR,\sigma}^+ \otimes_E \chi)^{\Gal_{F_{\wp}}} \ra (B_{\dR,\sigma}^+ \otimes_E \chi)^{\Gal_{F_{\wp}}} \ra (\bC_{p,\sigma} \otimes_E \chi)^{\Gal_{F_{\wp}}}\ra H^1(\Gal_{F_{\wp}}, tB_{\dR,\sigma}^+ \otimes_E \chi),
 \end{equation*}
 it's sufficient to prove $H^1(\Gal_{F_{\wp}}, tB_{\dR,\sigma}^+ \otimes_E \chi)=0$. For $i\in \Z_{>0}$, we claim $H^1(\Gal_{F_{\wp}},t^{i+1} B_{\dR,\sigma}^+ \otimes_E \chi) \ra H^1(\Gal_{F_{\wp}}, t^{i} B_{\dR,\sigma}^+ \otimes_E \chi)$ is an isomorphism: one has an exact sequence
 \begin{equation*}
    (\bC_{p,\sigma}(i) \otimes_E \chi)^{\Gal_{F_{\wp}}} \ra H^1(\Gal_{F_{\wp}},t^{i+1} B_{\dR,\sigma}^+ \otimes_E \chi) \ra H^1(\Gal_{F_{\wp}}, t^{i} B_{\dR,\sigma}^+ \otimes_E \chi) \ra H^1(\Gal_{F_{\wp}}, \bC_{p,\sigma}(i) \otimes_E \chi),
 \end{equation*}
 since $\bC_{p,\sigma}\otimes_E \chi\cong \bC_{p,\sigma}$, the first and fourth terms vanish when $i\geq 1$. We get thus  an isomorphism $H^1(\Gal_{F_{\wp}}, t B_{\dR,\sigma}^+ \otimes_E \chi) \xrightarrow{\sim} H^1(\Gal_{F_{\wp}}, t^n B_{\dR,\sigma}^+ \otimes_E \chi)$ for $n\gg0$, from which we deduce $H^1(\Gal_{F_{\wp}}, tB_{\dR,\sigma}^+ \otimes_E \chi)=0$.
\end{proof}
\begin{definition}[cf. $\text{\cite[Def.4.3.1]{Liu}}$]\label{def: pdeR-lir}
  Let $W$ be a triangulable $E$-$B$-pair of rank $r$ with a triangulation given by
  \begin{equation}\label{equ: pdeR-wqe}
  0=W_0 \subsetneq W_1 \subsetneq \cdots \subsetneq W_{r-1} \subsetneq W_r=W
\end{equation} with $W_{i+1}/W_{i}\cong B_E(\chi_i)$ for $0\leq i \leq r-1$ where the $\chi_i$'s are continuous characters of $F_{\wp}^{\times}$ in $E^{\times}$. For $\sigma\in \Sigma_{\wp}$, suppose $\wt(\chi_i)_{\sigma}\in \Z$ for all $0\leq i \leq r-1$, $W$ is called non $\sigma$-critical if \big(note the generalized Hodge-Tate weight of $B_E(\chi_i)$ at $\sigma$ is $-\wt(\chi_i)_{\sigma}$\big)
\begin{equation*}
  \wt(\chi_1)_{\sigma}> \wt(\chi_2)_{\sigma}>\cdots >\wt(\chi_r)_{\sigma};
\end{equation*}
for $\emptyset\neq  J\subseteq \Sigma_{\wp}$, suppose $\wt(\chi_i)_{\sigma}\in \Z$ for $0\leq i \leq r-1$, $\sigma\in J$, then $W$ is called non $J$-critical if $W$ is non $\sigma$-critical for all $\sigma\in J$.
\end{definition}
\begin{proposition}\label{prop: pdeR-tjL}
 Keep the notation in Def.\ref{def: pdeR-lir}, let $\emptyset \neq J\subseteq \Sigma_{\wp}$, suppose $W$ is non $J$-critical, then $W$ is $J$-de Rham.
\end{proposition}
\begin{proof}
  It's sufficient to prove if $W$ is non-$\sigma$-critical, then $W$ is $\sigma$-de Rham for $\sigma \in J$. Let $\sigma\in J$, we would use induction on $1\leq i \leq r-1$: by Lem.\ref{lem: pdeR-iasr}, $W_1$ is $\sigma$-de Rham; assume now $W_i$ is $\sigma$-de Rham, we show $W_{i+1}$ is also $\sigma$-de Rham. Note $[W_{i+1}]\in \Ext^1\big(W_i, B_E(\chi_{i+1})\big)$, let $W_i':=W_i\otimes B_E(\chi_{i+1}^{-1})$, $W_{i+1}':=W_{i+1} \otimes B_E(\chi_{i+1}^{-1})$, by Lem.\ref{lem: pdeR-iasr}, $W_{i+1}$ is $\sigma$-de Rham if and only if $W_{i+1}'$ is $\sigma$-de Rham. One has $[W_{i+1}']\in H^1(\Gal_{F_{\wp}}, W_{i}')$. On the other hand, since $\wt(\chi_j)_{\sigma}>\wt(\chi_{i+1})_{\sigma}$ for $1\leq j \leq i$, we see $H^0(\Gal_{F_{\wp}}, (W_i')_{\dR,\sigma}^+)=0$, thus by Lem.\ref{lem: lpl-tse}, $H^1_{g,\sigma}(\Gal_{F_{\wp}},W_i')\xrightarrow{\sim}H^1(\Gal_{F_{\wp}}, W_{i}')$. So $W_{i+1}'$ is $\sigma$-de Rham, and the proposition follows.
  %we see the generalized Hodge-Tate weights of $W_i'$ at $\sigma$ are negative integers. Thus the generalized Hodge-Tate weights of $(W_i')^{\vee}(1)$ at $\sigma$ are positive integers, so one has $D_{\dR}\big((W_i')^{\vee}(1)\big)_{\sigma}\cong D_{\dR}^+\big((W_i')^{\vee}(1)\big)_{\sigma}$, from which we see $H^1(\Gal_{F_{\wp}},W_i')=H^1_{g,\sigma}(\Gal_{F_{\wp}},W_i')$ by Prop.\ref{prop: lpl-sei} (and (\ref{equ: lpl-ewr})).
\end{proof}
\begin{example}
  Let $\chi_{\LT}: \Gal_{F_{\wp}}\ra F_{\wp}^{\times}$ be a Lubin-Tate character, $\sigma: F_{\wp}\hookrightarrow E$, and consider $H^1(\Gal_{F_{\wp}},\sigma\circ \chi_{\LT})$. By Prop.\ref{prop: pdeR-tjL}, any element in $H^1(\Gal_{F_{\wp}},\sigma\circ \chi_{\LT})$ is $\sigma$-de Rham, which generalizes the well-known fact that any extension of the trivial character by cyclotomic character is de Rham. In fact, suppose $F_{\wp}\neq \Q_p$, using (\ref{equ: lpl-jgw}), one  can actually calculate: $\dim_E H^1_{g,J}(\Gal_{F_{\wp}},\sigma\circ \chi_{\LT})=d-|J\setminus \{\sigma\}|$.
\end{example}
\addtocontents{toc}{\protect\setcounter{tocdepth}{1}}
\subsection*{Partially de Rham overconvergent Hilbert modular forms}
%As an application, we get a  partial classicality result (in terms of Galois representations) for finite slope overconvergent Hilbert modular forms.
Let $F$ be a totally real number field of degree $d_F$, $\Sigma_F$ the set of embeddings of $F$ in $\overline{\Q}$, $w\in \Z$, and $k_{\sigma}\in \Z_{\geq 2}$, $k_{\sigma}\equiv w\pmod{2}$ for all $\sigma\in \Sigma_F$. Let $\fc$ be a fractional ideal of $F$. Let $h$ be an overconvergent Hilbert eigenform of weights $(\ul{k}, w)$ (where we adopt Carayol's convention of weights as in \cite{Ca2})),  of tame level $N$ ($N\geq 4$, $p\nmid N$), of polarization $\fc$, with Hecke eigenvalues in $E$ (e.g. see  \cite[Def.1.1]{AIP2}, where $E$ is big enough to contain all the embeddings of $F$ in $\overline{\Q_p}$). For a place $\wp$ of $F$ above $p$, let $a_{\wp}$ denote the $U_{\wp}$-eigenvalue of $h$, and suppose $a_{\wp}\neq 0$ for all $\wp|p$. Denote by $\rho_h: \Gal_{F} \ra \GL_2(E)$ the associated (semi-simple)  Galois representation (enlarge $E$ if necessary) (e.g. see \cite[Thm.5.1]{AIP2}). For $\wp|p$, denote by $\rho_{h,\wp}$ the restriction of $\rho_h$ to the decomposition group at $\wp$, which is thus a continuous representation of $\Gal_{F_{\wp}}$ over $E$, where $F_{\wp}$ denotes the completion of $F$ at $\wp$. Let $\us_{\wp}: \overline{\Q_p} \ra \Q\cup\{+\infty\}$ be an additive valuation normalized by $\us_{\wp}(F_{\wp})=\Z \cup\{+\infty\}$. Denote by $\Sigma_{\wp}$ the set of embeddings of $F_{\wp}$ in $\overline{\Q_p}$. This section is devoted to prove
\begin{theorem}\label{thm: pdeR-jqa}
 With the above notation, and let $\emptyset \neq J \subseteq \Sigma_{\wp}$.

 (1) If $\us_{\wp}(a_{\wp})< \inf_{\sigma\in J}\{k_{\sigma}-1\}+\sum_{\sigma\in \Sigma_{\wp}} \frac{w-k_{\sigma}+2}{2}$,
  then $\rho_{h,\wp}$ is $J$-de Rham.

  (2) If $\us_{\wp}(a_{\wp})<\sum_{\sigma\in J} (k_{\sigma}-1) + \sum_{\sigma\in \Sigma_{\wp}} \frac{w-k_{\sigma}+2}{2}$,
  then there exists $\sigma\in J$ such that $\rho_{h,{\wp}}$ is $\sigma$-de Rham.
\end{theorem}
\begin{remark}
  This theorem gives evidence for Breuil's conjectures in \cite{Br00} (but in terms of Galois representations) (see in particular \cite[Prop.4.3]{Br00}). When $J=\Sigma_{\wp}$ (and $F_{\wp}$ unramified), the part (1) follows directly from the known classicality result in \cite{TX}.
\end{remark}
%By the global triangulation theory applied to the eigenvariety constructed in \cite{}, one has
One has as in Prop.\ref{prop: lpl-gpa}
\begin{proposition}
 For $\wp|p$, $\rho_{h,\wp}$ is trianguline with a triangulation given by
 \begin{equation*}
   0 \ra B_E(\delta_1) \ra W(\rho_{h,\wp}) \ra B_E(\delta_2) \ra 0,
 \end{equation*}
 with
 \begin{equation*}
   \begin{cases}
     \delta_1=\unr_{\wp}(a_{\wp}) \prod_{\sigma\in \Sigma_{\wp}} \sigma^{-\frac{w-k_{\sigma}+2}{2}}\prod_{\sigma\in \Sigma_h} \sigma^{1-k_{\sigma}}, \\
     \delta_2=\unr_{\wp}(q_{\wp}b_{\wp}/a_{\wp}) \prod_{\sigma\in \Sigma_{\wp}} \sigma^{-\frac{w+k_{\sigma}}{2}} \prod_{\sigma\in \Sigma_h} \sigma^{k_{\sigma}-1},
   \end{cases}
 \end{equation*}
 where $\unr_{\wp}(z)$ denotes the unramified character of $F_{\wp}^{\times}$ sending uniformizers to $z$, $q_{\wp}:=p^{f_{\wp}}$ with $f_{\wp}$ the degree of the maximal unramified extension inside $F_{\wp}$ (thus $\us_{\wp}(q_{\wp})=d_{\wp}$, the degree of $F_{\wp}$ over $\Q_p$), and $\Sigma_h$ is a certain subset of $\Sigma_{\wp}$.
\end{proposition}
\begin{proof}Consider the eigenvariety $\cE$ constructed  in \cite[Thm.5.1]{AIP2}, one can associate to $h$ a closed point $z_h$ in $\cE$.
  For classical Hilbert eigenforms, the result is known by Saito's results in \cite{Sa} and Nakamura's results on triangulations of $2$-dimensional semi-stable Galois representations (cf. \cite[\S 4]{Na}). Since the classical points are Zariski-dense in $\cE$ and accumulate over the point $z_h$ (here one uses the classicality results, e.g. in \cite{Bij}), the proposition follows from the global triangulation theory \cite[Thm.6.3.13]{KPX} \cite[Thm.4.4.2]{Liu}.
\end{proof}
Since $W(\rho_{\wp})$ is \'etale (purely of slope zero), by Kedlaya's slope filtration theroy (\cite[Thm.1.7.1]{Ked10}), one has (see also \cite[Lem.3.1]{Na})
\begin{lemma}
Let $\varpi_{\wp}$ be a uniformizer of $F_{\wp}$, then $\us_{\wp}(\delta_1(\varpi_{\wp}))\geq 0$.
\end{lemma}
\begin{proof}[Proof of Thm.\ref{thm: pdeR-jqa}]
  By the above lemma, one has $\us_{\wp}(a_{\wp}) \geq \sum_{\sigma\in \Sigma_h} (k_{\sigma}-1) + \sum_{\sigma\in \Sigma_{\wp}} \frac{w-k_{\sigma}+2}{2}$. Thus for $\emptyset \neq J \subseteq \Sigma_{\wp}$, if $\us_{\wp}(a_{\wp})< \inf_{\sigma\in J}\{k_{\sigma}-1\}+\sum_{\sigma\in \Sigma_{\wp}} \frac{w-k_{\sigma}+2}{2}$ \Big(resp. $\us_{\wp}(a_{\wp})<\sum_{\sigma\in S} (k_{\sigma}-1) + \sum_{\sigma\in \Sigma_{\wp}} \frac{w-k_{\sigma}+2}{2}$\Big), then $J\cap \Sigma_h=\emptyset$ \big(resp. $J\nsubseteq \Sigma_h$\big) and thus $\rho_{h,\wp}$ is non-$J$-critical \big(resp. there exists $\sigma\in J$ such that $\rho_{h,\wp}$ is non-$\sigma$-critical\big) (note $\Sigma_{F_{\wp}}\setminus \Sigma_h$ is exactly the set of embeddings where $\rho_{h,\wp}$ is non-critical). The theorem then follows from Prop.\ref{prop: pdeR-tjL}.
\end{proof}
We end this section by (conjecturally) constructing some partial de Rham families of Hilbert modular forms as closed subspaces of $\cE$ (\cite[Thm.5.1]{AIP2}). For $\wp|p$, denote by $\cW_{\wp}$ the rigid space over $E$ parameterizing locally $\Q_p$-analytic characters of $\co_{\wp}^{\times}$. One has a natural  morphism of rigid spaces $\cW_{\wp} \ra \bA^{|\Sigma_{\wp}|}$, $\chi\mapsto (\wt(\chi)_{\sigma})_{\sigma\in \Sigma_{\wp}}$. For $J\subseteq \Sigma_{\wp}$, $k_{\sigma}\in \Z$ for $\sigma\in J$, denote by $\cW_{\wp}(\ul{k}_J)$ the preimage of the rigid subspace of $\bA^{|\Sigma_{\wp}|}$ defined by fixing the $\sigma$-parameter to be $k_{\sigma}$ for $\sigma\in J$. Let $\cW_0$ denote the rigid space (over $E$) parameterizing locally $\Q_p$-analytic characters of $\Z_p^{\times}$. Recall (cf. \cite[Thm.5.1]{AIP2}), one has a natural morphism $\kappa: \cE \ra \prod_{\wp|p} \cW_{\wp} \times \cW_0$ (where the right hand is denoted by $\cW^{G}$ in \emph{loc. cit.}), mapping each point of $\cE$ (corresponding to overconvergent Hilbert eigenforms) to its weights.

Now fix $\wp|p$, $\emptyset \subsetneq J \subsetneq \Sigma_{\wp}$, $w\in \Z$, and $k_{\sigma}\in \Z_{\geq 2}$, $k_{\sigma}\equiv w\pmod{2}$ for all $\sigma\in J$. Consider the closed subspace
\begin{equation*}\cW_{\wp}(\ul{k}_{J}) \times \prod_{\substack{\wp'|p\\ \wp'\neq \wp}} \cW_{\wp'} \hooklongrightarrow \cW_{\wp} \times \prod_{\substack{\wp'|p\\ \wp'\neq \wp}} \cW_{\wp'} \hooklongrightarrow \prod_{\wp'|p} \cW_{\wp'} \times \cW_0
\end{equation*}
where the last map is induced by the $E$-point $(x\mapsto x^w)$ of $\cW_0$. Denote by $\cE(\ul{k}_J,w)'$ the pull-back of $\cW_{\wp}(\ul{k}_{J}) \times \prod_{\substack{\wp'|p\\ \wp'\neq \wp}} \cW_{\wp'}$ via $\kappa$, which is a closed rigid subspace of $\cE$ consisting of points with fixed weights $k_{\sigma}$ for $\sigma \in J$ and $w$. Let $\cE(\ul{k}_J,w)$ be the Zariski-closure of the classical points in $\cE(\ul{k}_J,w)'$.
\begin{conjecture}\label{conj: lpl-fyl}
  Keep the above notation, let $z\in \cE(\ul{k}_J,w)'(\overline{E})$, suppose the associated $\Gal_F$-representation $\rho_z$ is absolutely irreducible. Then $z\in \cE(\ul{k}_J,w)(\overline{E})$ if and only if $\rho_{z,\wp}:=\rho_z|_{\Gal_{F_{\wp}}}$ is $J$-de Rham.
\end{conjecture}

%\cW_{\wp,J}$ the rigid space parameterizing locally $J$-analytic characters of $\co_{\wp}^{\times}$ (which is thus a rigid closed  subspace of $\cW_{\wp}$)., and $\cW_{\wp,J}$ is exactly the preimage of the closed subspace of $\bA^{|\Sigma_{\wp}|}$ with the $\sigma$-parameter equal to zero for $\sigma\in J$.
%
%let $\cW$ denotes the rigid space parameterizing locally $\Q_p$-analytic characters of $(\co_F \otimes_{\Z} \Z_p)^{\times} \times \Z_p^{\times}=(\prod_{\wp|p} \co_{F_{\wp}}^{\times})\times \Z_p^{\times}$.

\end{document}